\documentclass[a4paper]{amsart}
\author[Brendle]{J{\"o}rg Brendle}
\thanks{The research of the first author was partially
 supported by Grants-in-Aid for Scientific Research (C) 15K04977 and (C) 18K03398, Japan Society
 for the Promotion of Science.
 Part of his work was done while he was visiting the Institute for Mathematical Sciences, National University of Singapore, during the program on {\em Aspects of Computation} (21  Aug – 15 Sep 2017); the visit was supported by the Institute.}
\address[Brendle]{Graduate School of System Informatics \\ Kobe University \\ Rokkodai 1-1, Nada, Kobe 657-8501, Japan}
\email{\href{brendle@kobe-u.ac.jp}{brendle@kobe-u.ac.jp}}
\author[Guzm\'{a}n]{Osvaldo Guzm\'{a}n}
\thanks{The second author was supported by by CONACyT
 scholarship 420090 and by NSERC of Canada.}
\address[Guzm\'{a}n]{Centro de Ciencias Matem\'aticas \\ UNAM, Campus Morelia, 58089, M\'exico}
\email{\href{oguzman@matmor.unam.mx}{oguzman@matmor.unam.mx}}
\urladdr{\url{https://www.matmor.unam.mx/~oguzman/}}
\author[Hru\v{s}\'{a}k]{Michael Hru\v{s}\'{a}k}
\thanks{The third author was supported by a
 PAPIIT grants IN 102311, IN 104220 and CONACyT grant A1-S-16164.}
\address[Hru\v{s}\'{a}k]{Centro de Ciencias Matem\'aticas \\ UNAM, Campus Morelia, 58089, M\'exico}
\email{\href{michael@matmor.unam.mx}{michael@matmor.unam.mx}}
\urladdr{\url{https://www.matmor.unam.mx/~michael/}}
\author[Raghavan]{Dilip Raghavan}
\thanks{This paper was completed when the fourth author was a Fields Research Fellow.
 The fourth author thanks the Fields Institute for its kind hospitality.}
\address[Raghavan]{Department of Mathematics \\
 National University of Singapore\\
 Singapore 119076.}
\email{\href{dilip.raghavan@protonmail.com}{dilip.raghavan@protonmail.com}}
\urladdr{\url{https://dilip-raghavan.github.io/}}
\date{\today}
\usepackage[only, llbracket, rrbracket]{stmaryrd}
\usepackage{amssymb, amsmath, amsthm, mathrsfs, enumitem, amsfonts, latexsym, bbm}
\usepackage{hyperref}
\usepackage{graphicx}%
\newtheorem{theorem}{Theorem}

\newtheorem{corollary}[theorem]{Corollary}

\newtheorem{definition}[theorem]{Definition}

\newtheorem{lemma}[theorem]{Lemma}

\newtheorem{problem}[theorem]{Problem}
\newtheorem{proposition}[theorem]{Proposition}

\theoremstyle{definition}

\theoremstyle{remark}

\newcounter{enuroman}
\renewcommand{\theenuroman}{\roman{enuroman}}

\newcounter{enuRoman}
\renewcommand{\theenuRoman}{(\Roman{enuRoman})}

\newcounter{enuAlph}
\renewcommand{\theenuAlph}{\Alph{enuAlph}}

\newcounter{enualph}
\renewcommand{\theenualph}{\alph{enualph}}

\newcounter{enuarabic}
\renewcommand{\theenuarabic}{\arabic{enuarabic}}

\newcommand{\forces}{\Vdash}

\newcommand{\A}{{\mathcal A}}
\newcommand{\B}{{\mathcal B}}
\newcommand{\C}{{\mathcal C}}
\newcommand{\D}{{\mathcal D}}

\newcommand{\I}{{\mathcal I}}
\newcommand{\J}{{\mathcal J}}
\newcommand{\K}{{\mathcal K}}
\newcommand{\LLL}{{\mathcal L}}
\newcommand{\M}{{\mathcal M}}
\newcommand{\N}{{\mathcal N}}

\newcommand{\R}{{\mathcal R}}

\newcommand{\Y}{{\mathcal Y}}
\newcommand{\Z}{{\mathcal Z}}

\newcommand{\MM}{{\mathbb M}}

\newcommand{\PP}{{\mathbb P}}
\newcommand{\QQ}{{\mathbb Q}}

\newcommand{\bb}{{\mathfrak b}}

\newcommand{\cc}{{\mathfrak c}}

\newcommand{\pp}{{\mathfrak p}}

\newcommand{\sss}{{\mathfrak s}}

\newcommand{\sub}{\subseteq}
\newcommand{\sem}{\setminus}

\newcommand{\omloms}{[\omega]^{<\omega}}

\newcommand{\omoms}{[\omega]^\omega}

\theoremstyle{theorem}
\newtheorem{dilipTheorem}[theorem]{Theorem}
\newtheorem{dilipClaim}[theorem]{Claim}
\newtheorem{dilipLemma}[theorem]{Lemma}
\newtheorem{dilipCor}[theorem]{Corollary}
\newtheorem{dilipconj}[theorem]{Conjecture}
\newtheorem{dilipQuestion}[theorem]{Question}
\theoremstyle{definition}
\newtheorem{dilipDef}[theorem]{Definition}
\theoremstyle{remark}
\newtheorem{dilipremark}[theorem]{Remark}

\newcommand{\dilipforces}{\Vdash}
\newcommand{\diliprestrict}{\mathord\upharpoonright}
\newcommand{\dilipforallbutfin}{{\forall}^{\infty}}
\newcommand{\dilipexistsinf}{{\exists}^{\infty}}
\newcommand{\dilipbb}{\mathfrak{b}}
\newcommand{\dilipsss}{\mathfrak{s}}
\newcommand{\dilipPP}{\mathbb{P}}
\newcommand{\dilipQQ}{\mathbb{Q}}
\newcommand{\diliplc}{\left|}
\newcommand{\diliprc}{\right|}
\newcommand{\dilipZFC}{\mathrm{ZFC}}
\newcommand{\dilipFIN}{\mathrm{FIN}}
\newcommand{\dilipGCH}{\mathrm{GCH}}
\newcommand{\dilipB}{\mathscr{B}}
\newcommand{\dilipBS}{{\omega}^{\omega}}
\DeclareMathOperator{\dilipnon}{non}
\DeclareMathOperator{\dilipdom}{dom}
\DeclareMathOperator{\dilipran}{ran}
\newcommand{\dilipPset}{\mathcal{P}}
\newcommand{\dilipMMM}{\mathcal{M}}
\newcommand{\dilipA}{{\mathscr{A}}}
\newcommand{\dilipAA}{{\mathbb{A}}}
\newcommand{\dilipCC}{{\mathbb{C}}}
\newcommand{\dilipGGG}{{\mathcal{G}}}
\newcommand{\dilipEEE}{{\mathcal{E}}}
\newcommand{\dilipIII}{{\mathcal{I}}}
\newcommand{\dilipLL}{{\mathbb{L}}}
\newcommand{\dilipFFF}{{\mathcal{F}}}
\newcommand{\dilipV}{{\mathbf{V}}}
\newcommand{\dilipVG}{{{\mathbf{V}}[G]}}
\newcommand{\dilipwo}{{<}_{\mathtt{wo}}}
\newcommand{\dilipSS}{\mathbb{S}}
\newcommand{\dilippr}[2]{\left\langle #1, #2 \right\rangle}
\newcommand{\dilipseq}[4]{\left\langle {#1}_{#2}: #2 #3 #4 \right\rangle}
\newcommand{\diliphf}{H({\aleph}_{0})}
\newcommand{\dilippc}[2]{{\left[#1\right]}^{#2}}
\begin{document}
\title{Combinatorial properties of MAD families}
\begin{abstract}
We study some strong combinatorial properties of \textsf{MAD }families. An
ideal $\mathcal{I}$ is Shelah-Stepr\={a}ns if for every set $X\subseteq\left[
\omega\right]  ^{<\omega}$ there is an element of $\mathcal{I}$ that either
intersects every set in $X$ or contains infinitely many members of it. We
prove that a Borel ideal is Shelah-Stepr\={a}ns if and only if it is
Kat\v{e}tov above the ideal \textsf{fin}$\times$\textsf{fin}. We prove that
Shelah-Stepr\={a}ns \textsf{MAD} families have strong indestructibility
properties (in particular, they are both Cohen and random indestructible). We
also consider some other strong combinatorial properties of \textsf{MAD} families.
Finally, it is proved that it is consistent to have $\dilipnon(\dilipMMM) = {\aleph}_{1}$ and no Shelah-Stepr{\= a}ns families of size ${\aleph}_{1}$.
\end{abstract}
\maketitle




\section{Introduction and preliminaries}

In \cite{ProductsofFilters} Kat\v{e}tov introduced a preorder on ideals. The
Kat\v{e}tov order is a very powerful tool for studying ideals over countable
sets. For the convenience of the reader, we will now recall the definition of
this order: Let $X\ $and $Y$ be two countable sets,$\ \mathcal{I},\mathcal{J}$
ideals on $X$ and $Y$ respectively and $f:Y\longrightarrow X$. We say $f$ is a
\emph{Kat\v{e}tov-morphism from }$\left(  Y,\mathcal{J}\right)  $ \emph{to}
$\left(  X,\mathcal{I}\right)  $ if $f^{-1}\left(  A\right)  \in\mathcal{J}$
for every $A\in\mathcal{I}.$ We say $\mathcal{I}\leq_{K}$
$\mathcal{J}$ ($\mathcal{I}$ is \emph{Kat\v{e}tov smaller} \emph{than}
$\mathcal{J}$ or $\mathcal{J}$ is \emph{Kat\v{e}tov above} $\mathcal{I}$) if
there is a Kat\v{e}tov-morphism from $\left(  Y,\mathcal{J}\right)  $ to
$\left(  X,\mathcal{I}\right)  .$ We say $\mathcal{I\simeq}_{K}$ $\mathcal{J}$
($\mathcal{I}$ \emph{is Kat\v{e}tov equivalent to} $\mathcal{J}$) if
$\mathcal{I}\leq_{K}$ $\mathcal{J}$ and $\mathcal{J}\leq_{K}$ 
$\mathcal{I}.$ The \emph{Kat\v{e}tov-Blass order} (denoted by
$\leq_{KB}$) is defined in the same way as the Kat\v{e}tov order, but
with the additional demand that the function $f$ must be finite to one. The
Kat\v{e}tov order has been applied successfully in classifying non definable
objects such as ultrafilters and \textsf{MAD} families. Just to mention a
couple of examples, an ultrafilter $\mathcal{U}$ is a $P$-point if and only if
the dual ideal $\mathcal{U}^{\ast}$ is not Kat\v{e}tov above the ideal
\textsf{fin}$\times$\textsf{fin }and\textsf{ }$\mathcal{U}$ is a Ramsey
ultrafilter if and only if $\mathcal{U}^{\ast}$ is not Kat\v{e}tov above the
ideal $\mathcal{ED}.$ In fact, most of the usual properties of ultrafilters
can be characterized with the Kat\v{e}tov order. In the case of ultrafilters,
the upward cones of definable ideals in the Kat\v{e}tov order play a
fundamental role, while the downward cones of definable ideals are very
important in the study of \textsf{MAD} families. The reader may consult
\cite{KatetovOrderonBorelIdeals} for a survey on the Kat\v{e}tov order on
Borel ideals.

Recall that a family $\mathcal{A}\subseteq\left[  \omega\right]  ^{\omega}$ is
\emph{almost disjoint (\textsf{AD}) }if the intersection of any two different
elements of $\mathcal{A}$ is finite and a \emph{\textsf{MAD} family } is
a maximal almost disjoint family. If $\mathcal{A}$ is an \textsf{AD}
family, we denote by $\mathcal{I}\left(  \mathcal{A}\right)  $ the ideal
generated by $\mathcal{A}$ and all finite sets of $\omega.$ This paper is part
of a larger project (initiated in \cite{OrderingMADFamiliesalaKatetov}) whose
goal is to study and classify (the ideals generated by) \textsf{MAD} families
using the Kat\v{e}tov order.

If $\mathcal{A}$ is a \textsf{MAD} family, we say that a forcing notion
$\mathbb{P}$ \emph{destroys }$\mathcal{A}$ if $\mathcal{A}$ is no longer
maximal after forcing with $\mathbb{P}.$ The destructibility of \textsf{MAD}
families has been extensively studied (see
\cite{ForcingIndestructibilityofMADFamilies}, \cite{KurilicMAD} and
\cite{ForcingwithQuotients}). Nevertheless, many fundamental
questions are still open. For example, the following problems still remain unsolved:

\begin{problem}  [Stepr\={a}ns]  \label{Stepransproblem}
Is there a Cohen indestructible \textsf{MAD} family?
\end{problem}

\begin{problem}  [Hru\v{s}\'{a}k]    \label{Hrusakproblem}
Is there a Sacks indestructible \textsf{MAD} family?
\end{problem}

The answer to both questions is positive under many additional axioms, but it
is currently unknown if it is possible to build such families on the basis of
\textsf{ZFC} alone. It is easy to see that there is a \textsf{MAD} family that
is destroyed by every forcing adding a new real, so  the main interest is
to construct \textsf{MAD} families that are indestructible under certain
forcings adding reals. It is known that every forcing adding a dominating real
will destroy every ground model \textsf{MAD} family.

In this article, we study some strong combinatorial properties of \textsf{MAD}
families: Shelah-Stepr\={a}ns, strongly tight and raving (see the next sections
for the definitions). We will prove that Shelah-Stepr\={a}ns \textsf{MAD}
families have very strong indestructibility properties, in fact, they are
indestructible by most definable forcings that do not add dominating reals
(see Proposition~\ref{SSindestructibility}). The notion of strongly
tight is a strengthening of Cohen indestructibility, yet they may be random
destructible. Raving is a strengthening of both Shelah-Stepr\={a}ns and strong tightness.

If $\mathcal{A}$ is an \textsf{AD }family, we denote by $\mathcal{A}^{\perp}$
the set of all $B\subseteq\omega$ that are almost disjoint with every element
of $\mathcal{A}.$ If $\mathcal{I}$ is an ideal, we denote by $\mathcal{I}^{+}$
the family of subsets of $\omega$ that are not in $\mathcal{I}.$ $\I^* = \{ \omega \sem A : A \in \I \}$
is the \emph{dual filter} of $\I$. We say that a
forcing $\mathbb{P}$ \emph{destroys (or diagonalizes)} $\mathcal{I}$ if
$\mathcal{I}$ is no longer tall after forcing with $\mathbb{P}$\footnote{An
ideal $\mathcal{I}$ is \emph{tall }if for every $X\in\left[  \omega\right]
^{\omega}$ there is $A\in\mathcal{I}$ such that $A\cap X$ is infinite.}. It is
easy to see that a forcing $\mathbb{P}$ destroys a \textsf{MAD} family
$\mathcal{A}$ if and only if it destroys $\mathcal{I}\left(  \mathcal{A}%
\right)  .$ The Kat\v{e}tov order is a fundamental tool for studying the
indestructibility of \textsf{MAD} families and ideals. The following notion is
needed in order to connect the Kat\v{e}tov order and the notion of indestructibility:

\begin{definition}
For every $a\subseteq\omega^{<\omega}$ we define $\pi\left(  a\right)
=\left\{  f\in\omega^{\omega}\mid\exists^{\infty}n\left(  f\upharpoonright
n\in a\right)  \right\}  .$ If $\mathcal{I}$ is a $\sigma$-ideal in
$\omega^{\omega}$ (or $2^{\omega}$) we define 
\emph{the trace ideal $tr\left(  \mathcal{I}\right)  $ of }$\mathcal{I}$ (which will be an ideal in
$\omega^{<\omega}$ or $2^{<\omega}$) such that $a\in tr\left(  \mathcal{I}\right)
$ if and only if $\pi\left(  a\right)  \in\mathcal{I}.$
\end{definition}

Note that if $a\subseteq\omega^{<\omega}$ then $\pi\left(  a\right)  $ is a
$G_{\delta}$ set (furthermore, every $G_{\delta}$ set is of this form). While
both $tr\left(  \mathcal{M}\right)  $ and $tr\left(  \mathcal{N}\right)  $ are
Borel (where $\mathcal{M}$ denotes the ideal of meager sets of $\omega
^{\omega}$ and $\mathcal{N}$ is the ideal of all null sets), in general, the
trace ideals are not Borel\ (see \cite{ForcingwithQuotients} for more
information). The relevance of the trace ideals in the study of
destructibility is the following result of Hru\v{s}\'{a}k and Zapletal (see
also \cite{ForcingIndestructibilityofMADFamilies}):

\begin{proposition}
[\cite{ForcingwithQuotients}]
\label{destructibilitycharacterization}
Let $\mathcal{I}$ be a $\sigma$-ideal in
$\omega^{\omega}$ such that $\mathbb{P}_{\mathcal{I}}=Borel\left(
\omega^{\omega}\right)  /\mathcal{I}$ is proper and has the continuous reading
of names\footnote{see \cite{ForcingIdealized} for the definition of continuous
reading of names.}. If $\mathcal{J}$ is an ideal on $\omega,$ then the
following are equivalent:
\begin{enumerate}
\item There is a condition $B\in\mathbb{P}_{\mathcal{I}}$ such that $B$ forces
that $\mathcal{J}$ is not tall.

\item There is $a\in tr\left(  \mathcal{I}\right)  ^{+}$ such that
$\mathcal{J}$ $\mathcal{\leq}_{K}$ $tr\left(  \mathcal{I}\right)
\upharpoonright a.$
\end{enumerate}
\end{proposition}

Usually, we assume ideals are proper and contain all finite sets. However,
there is an exception to this convention, which we will point out (see Subsection~\ref{genericMAD}). For every
$n\in\omega$ we define $C_{n}=\left\{  \left(  n,m\right)  \mid m\in
\omega\right\}  $ and if $f:\omega\longrightarrow\omega$ let $D\left(
f\right)  =\left\{  \left(  n,m\right)  \mid m\leq f(n)\right\}  .$ The ideal
\textsf{fin}$\times$\textsf{fin} is the ideal on $\omega\times\omega$
generated by $\left\{  C_{n}\mid n\in\omega\right\}  \cup\left\{  D\left(
f\right)  \mid f\in\omega^{\omega}\right\}  $ and $\emptyset\times
$\textsf{fin} is the ideal on $\omega\times\omega$ generated by $\left\{
D\left(  f\right)  \mid f\in\omega^{\omega}\right\}  .$ Note that
\textsf{fin}$\times$\textsf{fin }is a tall ideal while $\emptyset\times
$\textsf{fin }is not. It is well known that a forcing destroys \textsf{fin}%
$\times$\textsf{fin }if and only if it adds a dominating real. 
By \textsf{nwd }we denote the ideal of
all nowhere dense subsets of the rational numbers. The
\emph{density zero }ideal is defined as $\mathcal{Z}=\{A\subseteq\omega\mid
\lim\frac{\left\vert A\cap2^{n}\right\vert }{2^{n}}=0\}$. It is well known that
$\mathcal{Z}$ is an analytic $P$-ideal. The \emph{summable ideal} is defined as
$\mathcal{J}_{1/n}=\{A\subseteq\omega\mid{\textstyle\sum\limits_{n\in A}}
\frac{1}{n+1}<\omega\}.$ It is well known and easy to see that this is an
$F_{\sigma}$-ideal. In fact, $F_\sigma$-ideals and analytic $P$-ideals have a canonical representation:

\begin{definition}
We say $\varphi:\wp\left(  \omega\right)  \longrightarrow\mathbb{R\cup
}\left\{  \infty\right\}  $ is a \emph{lower semicontinuous submeasure} if the
following hold:

\begin{enumerate}
\item $\varphi\left(  \emptyset\right)  =0.$

\item $\varphi\left(  A\right)  \leq\varphi\left(  B\right)  $ whenever
$A\subseteq B.$

\item $\varphi\left(  A\cup B\right)  \leq\varphi\left(  A\right)
+\varphi\left(  B\right)  $ for every $A,B\subseteq X.$

\item \textrm{(lower semicontinuity)} if $A\subseteq\omega$ then $\varphi\left(
A\right)  = \sup\left\{  \varphi\left(  A\cap n\right)  \mid n\in\omega\right\}.$
\end{enumerate}
\end{definition}

Given a lower semicontinuous submeasure $\varphi$ we define \textsf{fin}$\left(  \varphi\right)  $ 
as those subsets of $\omega$ with finite submeasure
and $\textsf{Exh} \left(  \varphi\right)  =\left\{  A\subseteq\omega\mid \lim\left(
\varphi\left(  A\setminus n\right)  \right)  =0\right\}  .$ The following are
two fundamental results:

\begin{proposition}
Let $\mathcal{I}$ be an ideal in $\omega.$

\begin{enumerate}
\item {\rm (Mazur \cite{Mazur})} $\mathcal{I}$ is an $F_{\sigma}$-ideal if
and only if there is a lower semicontinuous submeasure $\varphi$ such that
$\mathcal{I=}$ \textsf{fin}$\left(  \varphi\right)  .$

\item {\rm (Solecki \cite{AnalyticIdealsandtheirApplications})
}$\mathcal{I}$ is an analytic $P$-ideal if and only if there is a lower
semicontinuous submeasure $\varphi$ such that $\mathcal{I=}$ $\textsf{Exh}\left(
\varphi\right)  $ (in particular, every analytic $P$-ideal is $F_{\sigma
\delta}$).
\end{enumerate}
\end{proposition}

For the definition of the
cardinal invariants used in this paper, the reader may consult
\cite{HandbookBlass}.
\section{Notation} \label{sec:notation}
We will use the following notational conventions for iterated forcing in Sections \ref{sec:big}, \ref{sec:partial}, and \ref{sec:mainresult}.
Let $\dilippr{\dilipPP}{{\leq}_{\dilipPP}}$ and $\dilippr{\dilipQQ}{ {\leq}_{\dilipQQ}}$ be posets and let $\pi: \dilipQQ \rightarrow \dilipPP$ be a projection.
If $G \subseteq \dilipPP$ is a $(\dilipV, \dilipPP)$-generic filter, then in $\dilipVG$ we define the poset ${\dilipQQ} \slash {G} = \{q \in \dilipQQ: \pi(q) \in G \}$ ordered by ${\leq}_{\dilipQQ}$.
In $\dilipV$, we let ${\dilipQQ} \slash {\mathring{G}}$ be a full $\dilipPP$-name for ${\dilipQQ} \slash {G}$.

Suppose $\langle {\dilipPP}_{\alpha}; {\mathring{\dilipQQ}}_{\alpha}: \alpha \leq \gamma \rangle$ is an iteration.
If ${G}_{\gamma}$ is a $(\dilipV, {\dilipPP}_{\gamma})$-generic filter, then for any $\alpha \leq \gamma$, ${G}_{\alpha}$ denotes $\{p \diliprestrict \alpha: p \in {G}_{\gamma}\}$, and it is a $(\dilipV, {\dilipPP}_{\alpha})$-generic filter.
Furthermore, if $\mathring{x}$ is any ${\dilipPP}_{\alpha}$-name, then $\mathring{x}$ is also a ${\dilipPP}_{\beta}$-name for every $\alpha \leq \beta \leq \gamma$ and $\mathring{x}\left[{G}_{\alpha}\right] = \mathring{x}\left[{G}_{\beta}\right]$.

If $\langle {\dilipPP}_{\alpha}; {\mathring{\dilipQQ}}_{\alpha}: \alpha \leq \gamma \rangle$ is an iteration, then for each $\alpha \leq \gamma$ the map ${\pi}_{\gamma\alpha}: {\dilipPP}_{\gamma} \rightarrow {\dilipPP}_{\alpha}$ given by ${\pi}_{\gamma\alpha}(p) = p \diliprestrict \alpha$ is a projection.
Therefore, if $\alpha \leq \gamma$ and if ${G}_{\gamma} \subseteq {\dilipPP}_{\gamma}$ is a $(\dilipV, {\dilipPP}_{\gamma})$-generic filter, then there is a $(\dilipV\left[{G}_{\alpha}\right], {{\dilipPP}_{\gamma}} \slash {{G}_{\alpha}})$-generic filter $H$ so that in $\dilipV\left[{G}_{\gamma}\right]$, ${G}_{\gamma} = {G}_{\alpha} \ast H$ holds.
In fact, this $H$ is equal to ${G}_{\gamma}$.

Suppose $\langle {\dilipPP}_{\alpha}; {\mathring{\dilipQQ}}_{\alpha}: \alpha \leq \gamma \rangle$ is an iteration and let $\alpha \leq \gamma$.
We may think of any ${\dilipPP}_{\gamma}$-name as a ${\dilipPP}_{\alpha}$-name for a ${\dilipPP}_{\gamma} \slash {\mathring{G}}_{\alpha}$-name.
Thus, given a ${\dilipPP}_{\gamma}$ name $\mathring{x}$, we use ${\mathring{x}}[{\mathring{G}}_{\alpha}]$ to denote a canonical ${\dilipPP}_{\alpha}$-name for a ${\dilipPP}_{\gamma} \slash {\mathring{G}}_{\alpha}$-name representing $\mathring{x}$.
If ${G}_{\alpha}$ is a $(\dilipV, {\dilipPP}_{\alpha})$-generic filter, we will write ${\mathring{x}}\left[{G}_{\alpha}\right]$ to denote the evaluation of ${\mathring{x}}[{\mathring{G}}_{\alpha}]$ by ${G}_{\alpha}$.
Therefore, if ${G}_{\gamma}$ is a $(\dilipV, {\dilipPP}_{\gamma})$-generic filter, then in $\dilipV\left[{G}_{\gamma}\right]$, ${\mathring{x}}\left[{G}_{\gamma}\right] = {\mathring{x}}\left[{G}_{\alpha}\right]\left[{G}_{\gamma}\right]$ holds.




\section{Shelah-Stepr\={a}ns ideals}

Let $\mathcal{I}$ be an ideal in $\omega$. By $\I^{< \omega}$ we denote the ideal
of subsets $X$ of $\omloms \sem \{ \emptyset \}$ for which there exists $A \in \I$ such that
$s \cap A \neq \emptyset$ for all $s \in X$. Thus $\left(  \mathcal{I}^{<\omega
}\right)  ^{+}$ is the set of all $X\subseteq\left[  \omega\right]
^{<\omega}\setminus\left\{  \emptyset\right\}  $ such that for every
$A\in\mathcal{I}$ there is $s\in X$ such that $s\cap A=\emptyset.$ The
following notion will play a fundamental role in this paper:

\begin{definition}
An ideal $\mathcal{I}$ is called \emph{Shelah-Stepr\={a}ns }if for every
$X\in\left(  \mathcal{I}^{<\omega}\right)  ^{+}$ there is $Y\in\left[
X\right]  ^{\omega}$ such that $\bigcup Y\in\mathcal{I}.$
\end{definition}

In other words, an ideal $\mathcal{I}$ is Shelah-Stepr\={a}ns if for every
$X\subseteq\left[  \omega\right]  ^{<\omega}\setminus\left\{  \emptyset
\right\}  $ either there is $A\in\mathcal{I}$ such that $s\cap A\neq\emptyset$
for every $s\in X$ or there is $B\in\mathcal{I}$ that contains infinitely many
elements of $X.$ This notion, introduced by Raghavan in
\cite{AModelwithnoStronglySeparableAlmostDisjointFamilies} for almost disjoint
families, is connected to the notion of \textquotedblleft strongly
separable\textquotedblright\ introduced by Shelah and Stepr\={a}ns in
\cite{MasasintheCalkinAlgebra}.

\begin{lemma}
Every non-meager ideal is Shelah-Stepr\={a}ns.
\end{lemma}

\begin{proof}
Let $\mathcal{I}$ be a non-meager ideal and $X\in\left(  \mathcal{I}^{<\omega
}\right)  ^{+}.$ Note that since $X\in\left(  \mathcal{I}^{<\omega}\right)
^{+}$ (and $\mathcal{I}$ contains every finite set) for every $n\in\omega$
there is $s\in X$ such that $s\cap n=\emptyset.$ In this way we can find
$Z=\left\{  s_{n}\mid n\in\omega\right\}  \subseteq X$ such that if $n\neq m$
then $s_{n}\cap s_{m}=\emptyset.$ We then define $M=\left\{  A\subseteq
\omega\mid\forall^{\infty}n\left(  s_{n}\nsubseteq A\right)  \right\}  $ which
is clearly a meager set and thus there must be $A\in\mathcal{I}$ such that
$A\notin M$. Hence there is $Y\in\left[  X\right]  ^{\omega}$ such that
$\bigcup Y\subseteq A\in\mathcal{I}.$
\end{proof}

Nevertheless, there are meager ideals that are also Shelah-Stepr\={a}ns as the
following result shows:

\begin{lemma}
\textsf{fin}$\times$\textsf{fin} is Shelah-Stepr\={a}ns.
\end{lemma}

\begin{proof}
It is easy to see that if $X\in(($\textsf{fin}$\times$\textsf{fin}$)^{<\omega})^{+}$ then
there must be infinitely many elements of $X$ that are below the graph of a
function, so there must be $Y\in\left[  X\right]  ^{\omega}$ such that
$\bigcup Y\in$ \textsf{fin}$\times$\textsf{fin}. 
\end{proof}

We now show that the property of being Shelah-Stepr\={a}ns is upward
closed in the Kat\v{e}tov order:

\begin{lemma}     \label{SSKatetovupwards}
Let $\mathcal{I}$ and $\mathcal{J}$ be two ideals on $\omega.$ If the ideal
$\mathcal{I}$ is Shelah-Stepr\={a}ns and $\mathcal{I}\leq_{K}$
$\mathcal{J}$ then $\mathcal{J}$ is also Shelah-Stepr\={a}ns.
\end{lemma}

\begin{proof}
Let $f:\omega\longrightarrow\omega$ be a Kat\v{e}tov-morphism from $\left(
\omega,\mathcal{J}\right)  $ to $\left(  \omega,\mathcal{I}\right)  .$ Letting
$X\in\left(  \mathcal{J}^{<\omega}\right)  ^{+}$ we must find $Y\in\left[
X\right]  ^{\omega}$ such that $\bigcup Y\in\mathcal{J}.$ Defining
$X_{1}=\left\{  f\left[  s\right]  \mid s\in X\right\}  ,$ we will first argue
that \thinspace$X_{1}\in\left(  \mathcal{I}^{<\omega}\right)  ^{+}.$ To prove
this fact, let $A\in\mathcal{I}.$ Since $f$ is a Kat\v{e}tov-morphism,
$f^{-1}\left(  A\right)  \in\mathcal{J}$ so there is $s\in X$ for which $s\cap
f^{-1}\left(  A\right)  =\emptyset$ and then $f\left[  s\right]  \cap
A=\emptyset.$ Since $\mathcal{I}$ is Shelah-Stepr\={a}ns, there is $Y_{1}%
\in\left[  X_{1}\right]  ^{\omega}$ such that $\bigcup Y_{1}\in\mathcal{I}.$
Finally if $Y\in\left[  X\right]  ^{\omega}$ is such that $Y_{1}=\left\{
f\left[  s\right]  \mid s\in Y\right\}  $ then $\bigcup Y\in\mathcal{J}.$
\qquad\ \ 
\end{proof}

We will need the following game designed by Claude Laflamme: Letting $\mathcal{I}$
be an ideal on $\omega,$ define the game $\mathcal{L}\left(  \mathcal{I}
\right)  $ between players $\mathsf{I}$ and $\mathsf{II}$ as follows:

\begin{center}%
\begin{tabular}
[c]{|l|l|l|l|l|l|}\hline
$\mathsf{I}$ & $...$ & $A_{n}$ &  & $...$ & \\\hline
$\mathsf{II}$ & $...$ &  & $s_{n}$ & $...$ & $\bigcup s_{n}\in\mathcal{I}^{+}%
$\\\hline
\end{tabular}

\end{center}
At round $n\in\omega$ player $\mathsf{I}$ plays $A_{n}\in\mathcal{I}$ and
$\mathsf{II}$ responds with $s_{n}\in\lbrack\omega\backslash$ $A_{n}%
]^{<\omega}.$ The player $\mathsf{II}$ wins in case $\bigcup s_{n}%
\in\mathcal{I}^{+}$. The following is a result of Laflamme.

\begin{proposition}
[Laflamme \cite{FilterGamesandCombinatorialPropertiesofStrategies}]Let
$\mathcal{I}$ be an ideal on $\omega.$

\begin{enumerate}
\item The following are equivalent:

\begin{enumerate}
\item $\mathsf{I}$ has a winning strategy in $\mathcal{L}\left(
\mathcal{I}\right)  .$

\item For every $\left\{  F_{s}\mid s\in\omega^{<\omega}\right\}
\subseteq\mathcal{I}^{\ast},$ there is an increasing function $f\in
\omega^{\omega}$ such that $%
{\textstyle\bigcup\limits_{n\in\omega}}
\left(  F_{f\upharpoonright n}\cap f\left(  n\right)  \right)  \in
\mathcal{I}^{+}.$

\item Every countable subset of $\mathcal{I}^{\ast}$ has a pseudointersection
in $\mathcal{I}^{+}.$

\item \textsf{fin}$\times$\textsf{fin }$\leq_{\mathsf{K}}\mathcal{I}.$
\end{enumerate}

\item The following are equivalent:

\begin{enumerate}
\item $\mathsf{II}$ has a winning strategy in $\mathcal{L}\left(
\mathcal{I}\right)  .$

\item There is $\left\{  X_{n}\mid n\in\omega\right\}  \subseteq\left(
\mathcal{I}^{<\omega}\right)  ^{+}$ such that for every $A\in\mathcal{I}$
there is $n\in\omega$ such that $A\cap X_{n}=\emptyset.$
\end{enumerate}
\end{enumerate}
\end{proposition}

\begin{proof}
Everything in the Proposition is either contained in
\cite{FilterGamesandCombinatorialPropertiesofStrategies} or is trivial, with
the exception of (c) implies (b) of 1, which is only mentioned in
\cite{FilterGamesandCombinatorialPropertiesofStrategies} without proof. We
will provide a brief sketch of the proof of this implication.

\qquad\qquad\qquad\ \ \ 

Let $\left\{  F_{s}\mid s\in\omega^{<\omega}\right\}  \subseteq\mathcal{I}%
^{\ast}.$ Without lost of generality, we may assume that for every
$s,t\in\omega^{<\omega},$ if $t\subseteq s,$ then $F_{s}\subseteq F_{t}.$ For
every $n\in\omega,$ define:\bigskip

\hfill%
\begin{tabular}
[c]{l}%
$H_{n}=%
{\textstyle\bigcap}
\{F_{s}\mid s\in n^{\leq n}\wedge\forall i\in dom\left(  s\right)  \left(
s\left(  i\right)  \leq n\right)  \}$%
\end{tabular}
\hfill\qquad\ 

\medskip

\qquad\qquad\ \ \ \ \ \ 

Note that $H_{n}\in\mathcal{I}^{\ast}$ for every $n\in\omega.$ By point (c),
there is $X\in\mathcal{I}^{+}$ a pseudointersection of $\left\{  H_{n}\mid
n\in\omega\right\}  .$ We now find a sequence $\left\langle n_{i}\right\rangle
_{i\in\omega}$ such that for every $i\in\omega,$ the following holds:

\begin{enumerate}
\item $n_{i}<n_{i+1}.$

\item $X\setminus n_{i+1}\subseteq H_{n_{i}}.$
\end{enumerate}

\qquad\ \ \ \ 

We now define the following sets:\bigskip

\hfill%
\begin{tabular}
[c]{c}%
$Y_{0}=X\cap\left(
{\textstyle\bigcup}
\{[n_{i},n_{i+1})\mid i\text{ is even}\}\right)  $\\
$Y_{1}=X\cap\left(
{\textstyle\bigcup}
\{[n_{i},n_{i+1})\mid i\text{ is odd}\}\right)  $%
\end{tabular}
\hfill\qquad\ 

\medskip

\qquad\qquad\ \ \ \ \ \ 

Since $X\in\mathcal{I}^{+},$ there is $i\in2$ such that $Y_{i}\in
\mathcal{I}^{+}$. If $Y_{0}\in\mathcal{I}^{+}$, define $f\in\omega^{\omega}$
where $f\left(  m\right)  =n_{2m+1}$ and if $Y_{1}\in\mathcal{I}^{+},$ define
$f$ by $f\left(  m\right)  =n_{2\left(  m+1\right)  }.$ In either case, it is
easy to see that $%
{\textstyle\bigcup\limits_{m\in\omega}}
\left(  F_{f\upharpoonright m}\cap f\left(  m\right)  \right)  \in
\mathcal{I}^{+}.$
\end{proof}

If $s_{0},...,s_{n}$ are finite non-empty subsets of $\omega$, we say $a=\left\{
k_{0},...,k_{n}\right\}  \in\left[  \omega\right]  ^{<\omega}$ is a
\emph{selector }of $\left(  s_{0},...,s_{n}\right)  $ if $k_{i}\in s_{i}$ for
every $i\leq n.$

\begin{proposition}
If $\mathcal{I}$ is Shelah-Stepr\={a}ns then $\mathsf{II}$ does not have a
winning strategy in $\mathcal{L}\left(  \mathcal{I}\right)  .$
\end{proposition}

\begin{proof}
Let $\mathcal{I}$ be an ideal for which $\mathsf{II}$ has a winning strategy
in $\mathcal{L}\left(  \mathcal{I}\right)  .$ We will prove that $\mathcal{I}$
is not Shelah-Stepr\={a}ns. Let $\left\{  X_{n}\mid n\in\omega\right\}
\subseteq\left(  \mathcal{I}^{<\omega}\right)  ^{+}$ such that for every
$A\in\mathcal{I}$ there is $n\in\omega$ such that $A$ does not contain any
element of $X_{n}.$ For every $n\in\omega$ enumerate\ $X_{n}=\left\{
t_{n}^{i}\mid i\in\omega\right\}  $ and $\prod\limits_{j<n}X_{j}=\{p_{n}%
^{i}\mid i<\omega\}.$


For every $n,m\in\omega$ and a selector $a\in\left[  \omega\right]  ^{<\omega
}$ of $\left(  t_{n}^{0},...,t_{n}^{m}\right)  $ we define $F_{\left(
n,m,a\right)  }=p_{n}^{m}\left(  0\right)  \cup...\cup p_{n}^{m}\left(
n-1\right)  \cup a$ (recall $p_{n}^{m}\in\prod\limits_{j<n}X_{j}$). Clearly
each $F_{\left(  n,m,a\right)  }$ is a non-empty finite set. Let $X$ be the
collection of all the $F_{\left(  n,m,a\right)  };$ we will prove that $X$
witnesses that $\mathcal{I}$ is not Shelah-Stepr\={a}ns.

We will first prove that $X\in\left(  \mathcal{I}^{<\omega}\right)  ^{+}.$
Letting $A\in\mathcal{I}$ we first find $n\in\omega$ such that $A$ does not
contain any element of $X_{n}.$ Since each $X_{j}\in\left(  \mathcal{I}%
^{<\omega}\right)  ^{+}$ for every $j<\omega$ there is $m\in\omega$ such that
$A$ is disjoint with $p_{n}^{m}\left(  0\right)  \cup...\cup p_{n}^{m}\left(
n-1\right)  .$ Finally, by the assumption of $X_{n}$ we can find a selector
$b$ of $\left(  t_{n}^{0},...,t_{n}^{m}\right)  $ such that $b\cap
A=\emptyset$ and therefore $A\cap F_{\left(  n,m,b\right)  }=\emptyset.$

Letting $Y\in\left[  X\right]  ^{\omega}$ we will show that $B=\bigcup
Y\in\mathcal{I}^{+}.$ There are two cases to consider: first assume there is
$n\in\omega$ for which there are infinitely many $\left(  m,a\right)  $ such
that $F_{\left(  n,m,a\right)  }\in Y.$ In this case, $B$ intersects every
element of $X_{n},$ hence $B\in\mathcal{I}^{+}.$ Now assume that for every
$n\in\omega$ there are only finitely many $\left(  m,a\right)  $ such that
$F_{\left(  n,m,a\right)  }\in Y.$ In this case, there must be infinitely many
$n\in\omega$ for which there is $\left(  m,a\right)  $ such that $F_{\left(
n,m,a\right)  }\in Y,$ hence $B$ must contain (at least) one element of every
$X_{k}.$ We can then conclude that $B\in\mathcal{I}^{+}.$
\end{proof}

As a consequence we obtain (the equivalence of items 2 and 3 was proved
by Laczkovich and Rec\l aw in
\cite{IdeallimitsofSequencesofContinuousFunctions}; we include the proof for
the convenience of the reader):

\begin{corollary}
Let $\mathcal{I}$ be an ideal on $\omega.$ The following are equivalent:

\begin{enumerate}
\item $\mathcal{I}$ is not Shelah-Stepr\={a}ns.

\item The Player $\mathsf{II}$ has a winning strategy in $\mathcal{L}\left(
\mathcal{I}\right)  .$

\item There is an $F_{\sigma}$ set $F\subseteq\wp\left(  \omega\right)  $ such
that $\mathcal{I}$ $\subseteq F$ and $\mathcal{I}^{\ast}\cap F=\emptyset.$
\end{enumerate}
\end{corollary}

\begin{proof}
By the previous result, we know that 2 implies 1. We first prove that 1
implies 3. Assume that $\mathcal{I}$ is not Shelah-Stepr\={a}ns, so there is
$X=\left\{  s_{n}\mid n\in\omega\right\}  \in\left(  \mathcal{I}^{<\omega
}\right)  ^{+}$ such that ${\textstyle\bigcup}
Y\in\mathcal{I}^{+}$ for every $Y\in\left[  X\right]  ^{\omega}.$ We now
define the set $F=\left\{  W\subseteq\omega \mid \forall^{\infty}n \left(  s_{n}
\nsubseteq W\right)  \right\}  .$ It is easy to see that $F$ has the desired properties.

We finally prove that 3 implies 2. Assume there is an increasing sequence of
closed sets $\left\langle C_{n}\mid n\in\omega\right\rangle $ such that $F=%
{\textstyle\bigcup\limits_{n\in\omega}}
C_{n}$ contains $\mathcal{I}$ and is disjoint from $\mathcal{I}^{\ast}.$ We
will now describe a winning strategy for Player $\mathsf{II}$ in $\mathcal{L} (\I)$: In the first
round, if Player $\mathsf{I}$ plays $A_{0}\in\mathcal{I}$ then Player
$\mathsf{II}$ finds an initial segment $s_{0}$ of $\omega\setminus A_{0}$ such
that $\left\langle s_{0}\right\rangle =\left\{  Z\mid s_{0}\sqsubseteq
Z\right\}  $ is disjoint from $C_{0}$ (where $s_{0}\sqsubseteq Z$ means that
$s_{0}$ is an initial segment of $Z$). At round round $n+1,$ if Player
$\mathsf{I}$ plays $A_{n+1}\in\mathcal{I}$ then Player $\mathsf{II}$ finds
$s_{n+1}$ such that $t={\textstyle\bigcup\limits_{i\leq n+1}}
s_{i}$ is an initial segment of $\left(  \omega\setminus A_{n}\right)  \cup%
{\textstyle\bigcup\limits_{j<n+1}} s_{j}$ (we may assume $%
{\textstyle\bigcup\limits_{j<n+1}}
s_{j}\subseteq A_{n}$) and $\left\langle t\right\rangle $ is disjoint from
$C_{n+1}.$ It is easy to see that this is a winning strategy.
\end{proof}

Since every game with Borel payoff is determined, we can give a
characterization of the Borel ideals that are Shelah-Stepr\={a}ns.

\begin{theorem}   \label{SSBorel}
If $\mathcal{I}$ is a Borel ideal then $\mathcal{I}$ is Shelah-Stepr\={a}ns if
and only if \textsf{fin}$\times$\textsf{fin }$\leq_{K}\mathcal{I}.$
\end{theorem}

We can extend this theorem under some large cardinal assumptions.
Fix a tree $T$ of height $\omega,$ $f:\left[  T\right]  \longrightarrow
\wp\left(  \omega\right)  $ a continuous function (where $\left[  T\right]  $
denotes the set of branches of $T$) and $\mathcal{W}\subseteq\wp\left(
\omega\right)  .$ We then define the game $\mathcal{G}\left(  T,f,\mathcal{W}%
\right)  $ as follows:

\begin{center}%
\begin{tabular}
[c]{|l|l|l|l|l|}\hline
$\mathsf{I}$ & $...$ & $x_{n}$ &  & $...$\\\hline
$\mathsf{II}$ & $...$ &  & $y_{n}$ & $...$\\\hline
\end{tabular}

\end{center}
At round $n\in\omega$ player $\mathsf{I}$ plays $x_{n}$ and $\mathsf{II}$
responds with $y_{n}$ with the requirement that $\left\langle x_{0}%
,y_{0},...,x_{n},y_{n}\right\rangle \in T.$ Player $\mathsf{I}$ wins if
$f\left(  b\right)  \in\mathcal{W}$ where $b$ is the branch constructed during
the game. The following is a well known extension of Martin's result (see
\cite{ForcingIdealized}):

\begin{proposition}
[\textsf{LC}]If $\mathcal{W}\in$ \textsf{L}$\left(  \mathbb{R}\right)  $ then
$\mathcal{G}\left(  T,f,\mathcal{W}\right)  $ is determined (\textsf{L}%
$\left(  \mathbb{R}\right)  $ denotes the smallest transitive model of
$\mathsf{ZFC}$ that contains all reals)
\end{proposition}

Here \textsf{LC }denotes a large cardinal assumption. In this case, it is enough to assume that there is a proper class of Woodin cardinals. The reader may consult the first chapter of \cite{ForcingIdealized}) for more information. We can conclude
the following:

\begin{theorem}
[\textsf{LC}]
\label{SSdefinable}
\begin{enumerate}
\item Let $\mathcal{I}\in$ \textsf{L}$\left(  \mathbb{R}\right)  $ be an ideal
on $\omega.$ Then $\mathcal{I}$ is Shelah-Stepr\={a}ns if and only if
\textsf{fin}$\times$\textsf{fin }$\leq_{K}\mathcal{I}.$

\item Let $\mathcal{J}$ be a $\sigma$-ideal in $\omega^{\omega}$ such that
$\mathcal{J}\in$ \textsf{L}$\left(  \mathbb{R}\right)  $ and $X\in tr\left(
\mathcal{J}\right)  ^{+}.$ Then $tr\left(  \mathcal{J}\right)  \upharpoonright
X$ is Shelah-Stepr\={a}ns if and only if \textsf{fin}$\times$\textsf{fin
}$\leq_{K}tr\left(  \mathcal{J}\right)  \upharpoonright X.$
\end{enumerate}
\end{theorem}

\begin{proof}
To prove the first item, let $Y$ be the set of all sequences $\left\langle
A_{0},s_{0},...,A_{n},s_{n}\right\rangle $ such that $A_{n}\in\mathcal{I}$ and
$s_{n}\in\left[  \omega\setminus A_{n}\right]  ^{<\omega}$ and $\max\left(
s_{i}\right)  \subseteq A_{i+1}$ if $i<n.$ Let $T$ be the tree obtained by
closing $Y$ under restrictions. We define $f:\left[  T\right]
\longrightarrow\wp\left(  \omega\right)  $ by $f\left(  b\right)  =%
{\textstyle\bigcup\limits_{n\in\omega}}
b\left(  2n+1\right)  $ where $b\in\left[  T\right]  .$ Clearly $\mathcal{L}%
\left(  \mathcal{I}\right)  $ is a game equivalent to $\mathcal{G}\left(
T,f,\mathcal{I}\right)  $, so the result follows from the previous results. The
second item is a consequence of the first.
\end{proof}

The following result will be useful later, see Lemma~\ref{ShelahSteprans-Hurewicz}:

\begin{lemma}  \label{SScharacterization}
Let $\mathcal{I}$ be an ideal on $\omega.$ The following are equivalent:

\begin{enumerate}
\item $\mathcal{I}$ is Shelah-Stepr\={a}ns.

\item For every $\left\{  X_{n}\mid n\in\omega\right\}  \subseteq\left(
\mathcal{I}^{<\omega}\right)  ^{+}$ there is $B\in\mathcal{I}$ such that
$X_{n}\cap\left[  B\right]  ^{<\omega}$ is infinite for every $n\in\omega.$
\end{enumerate}
\end{lemma}

\begin{proof}
Clearly 2 implies 1 and if 2 fails then it is easy to see that Player
$\mathsf{II}$ has a winning strategy in $\mathcal{L}\left(  \mathcal{I}%
\right)  ,$ so 1 also fails.
\end{proof}




\section{Strong  properties of \textsf{MAD} families}

In this section, we will study some strong combinatorial properties of
\textsf{MAD }families and we will clarify the relationship between them. 
The basic notions and implications will be presented in the more
general context of ideals (Subsection~\ref{ideals-basic}), and for existence, non-existence 
(see Subsections~\ref{existence-nonexistence}, \ref{diamond-existence} and~\ref{p=c-existence}), and
non-implications (Subsection~\ref{non-implications}) we will consider the special case of \textsf{MAD}
families. We shall also consider generic \textsf{MAD} families (Subsection~\ref{genericMAD}).


\subsection{Combinatorial properties of ideals: definitions and implications}
\label{ideals-basic}


We start with:

\begin{definition}
Let $\I$ be an ideal on $\omega$. 
\begin{enumerate}
\item $\I$ is \emph{tight} if for every
$\left\{  X_{n}\mid n\in\omega\right\}  \subseteq\mathcal{I}^{+}$ there is $A\in\mathcal{I}$ 
such that $A\cap X_{n}\neq\emptyset$ for every $n\in\omega.$
\item $\I$ is \emph{weakly tight} if for every
$\left\{  X_{n}\mid n\in\omega\right\}  \subseteq\mathcal{I}^{+}$ there is $A\in\mathcal{I}$ 
such that $|A\cap X_{n}| = \omega$ for infinitely many $n\in\omega.$
\item $\I$ is \emph{strongly tight} if for every
$\left\{  X_{n}\mid n\in\omega\right\}  \subseteq \omoms$ such that 
$\{ n \mid X_n \sub^* Y\}$ is finite for every $Y \in \I$, there is $A\in\mathcal{I}$ 
such that $A\cap X_{n}\neq\emptyset$ for every $n\in\omega.$
\item We say a family $X=\left\{  X_{n}\mid n\in\omega\right\}  $ such that
$X_{n}\subseteq\left[  \omega\right]  ^{<\omega}$ is \emph{locally finite
according to }$\mathcal{I}$ if for every $B\in\mathcal{I}$ for almost all
$n\in\omega$ there is $s\in X_{n}$ such that $s\cap B=\emptyset.$
\item $\mathcal{I}$ is \emph{raving }if for every family $X=\left\{
X_{n}\mid n\in\omega\right\}  $ that is locally finite according to\emph{
}$\mathcal{I}$ there is $A\in\mathcal{I}$ such that $A$ contains at least one
element of each $X_{n}.$ 
\end{enumerate}
\end{definition}

Obviously, strongly tight implies tight, which in turn implies weakly tight.
Also, it is easy to see that raving implies both Shelah-Stepr\= ans and
strongly tight. 
Furthermore, by Lemma~\ref{SScharacterization}, every Shelah-Stepr\=ans ideal is tight. 
We shall see later on (in particular in Subsection~\ref{non-implications}) that all 
these properties are (consistently) distinct, even for \textsf{MAD} families.



For a \textsf{MAD} family $\A$ and a property X of ideals, we say {\em $\A$ has property X} whenever
the corresponding ideal $\I (\A)$ has property X.

Tight \textsf{MAD} families are Cohen indestructible. Although Cohen
indestructibility does not imply tightness, it is true that every Cohen
indestructible \textsf{MAD} family has a restriction that is tight (see
\cite{OrderingMADFamiliesalaKatetov}). Thus, existence-wise, the two properties
are on the same level.

In~\cite{SplittingFamiliesandCompleteSeparability} and
\cite{OnWeaklytightFamilies} it was proved that weakly tight \textsf{MAD}
families exist under $\mathfrak{s\leq b}.$

Strongly tight \textsf{MAD} families have the following characterization.

\begin{lemma}   \label{stronglytightMAD}
$\A$ is strongly tight iff whenever $\{ B_n \mid n \in \omega \} \sub \omoms$ is such that there is
$\{ A_n \mid n \in \omega \} \sub \A$ such that
\begin{itemize}
\item $B_n \sub A_n$ for all $n$,
\item for all $A \in \A$, the set $\{ n \mid A_n = A \}$ is finite,
\end{itemize}
there is $A \in \I (\A)$ such that $A \cap B_n \neq \emptyset$ for all $n$.
\end{lemma}

\begin{proof}
First assume $\A$ is strongly tight. Assume $B_n$ is given as in the lemma. Clearly $X_n = B_n$ satisfies
the assumption in the definition of ``strongly tight" and we obtain $A \in \I (\A)$ as required.

On the other hand, if $X_n$ are given as in the definition of ``strongly tight", first use the maximality
of $\A$ to find $A_n \in \A$ such that $A_n \cap X_n$ is infinite and let $B_n = A_n \cap X_n$. 
Whenever possible, choose $A_n$ distinct from $A_i, i < n$. The assumption on the $X_n$ then guarantees
that every $A \in \A$ is chosen only finitely often. The $A \in \I (\A)$ given by the lemma is as required.
\end{proof}

Like the Shelah-Stepr\=ans property, the properties considered here are upwards closed in the Kat\v etov order
(see Lemma~\ref{SSKatetovupwards}).

\begin{lemma}   \label{tightKatetovupwards}
Assume $\I$ and $\J$ are ideals on $\omega$.
\begin{enumerate}
\item If  $\I \leq_K \J$ and $\I$ is tight (weakly tight, raving, resp.), then so is $\J$.
\item If $\I \leq_{KB} \J$ and $\I$ is strongly tight, then so is $\J$.
\item Assume $\A$ and $\B$ are \textsf{MAD} families. If
$\mathcal{A}$ is strongly tight and $\mathcal{I}\left(  \mathcal{A}\right)
\leq_{K}\mathcal{I}\left(  \mathcal{B}\right)  $ then $\mathcal{B}$ is
strongly tight.
\end{enumerate}
\end{lemma}

\begin{proof}
1. This is a standard argument. Let $f : \omega\to\omega$ be a Kat\v etov reduction.
Assume first $\I$ is tight. Take $\{ X_n \mid n \in \omega \} \sub \J^+$ and let $Y_n = f [X_n]$.
Clearly the $Y_n$ belong to $\I^+$, and therefore there is $B \in \I$ such that $B \cap Y_n \neq \emptyset$
for all $n$. Letting $A = f^{-1} [B] \in \J$ we see that $A \cap X_n \neq \emptyset $ for all $n$. 
The proofs for ``weakly tight" and ``raving" are similar.

2. This is also similar.

3. Fix a Kat\v{e}tov-morphism $f$ from $\left(  \omega,\mathcal{I}\left(
\mathcal{B}\right)  \right)  $ to $\left(  \omega,\mathcal{I}\left(
\mathcal{A}\right)  \right)  $ and a family $\mathcal{W}=\left\{  X_{n}\mid
n\in\omega\right\}  $ such that for every $n\in\omega$ there is $B_{n}%
\in\mathcal{B}$ such that $X_{n}\subseteq B_{n}$ and for every $B\in
\mathcal{B}$ the set $\left\{  n\mid B_{n}=B\right\}  $ is finite. Let
$\mathcal{W}_{1}=\left\{  X\in\mathcal{W}\mid f\left[  X\right]  \in\left[
\omega\right]  ^{<\omega}\right\}  $ and for every $X\in\mathcal{W}_1$ we choose
$b_{X}\in f\left[  X\right]  $ such that $f^{-1}\left(  \left\{
b_{X}\right\}  \right)  $ is infinite. We first claim that the set $Y=\left\{
b_{X}\mid X\in\mathcal{W}_1\right\}  $ is finite. If this was not the case, we
could find $A\in\mathcal{A}$ such that $A\cap Y$ is infinite. Since $f$ is a
Kat\v{e}tov-morphism, we conclude that $f^{-1}\left(  A\right)  \in
\mathcal{I}\left(  \mathcal{B}\right)  $ and $\left\{  X\in\mathcal{W}\mid
f^{-1}\left(  A\right)  \cap X\in\left[  \omega\right]  ^{\omega}\right\}  $
is infinite, but this is a contradiction. Using that $Y$ is finite, it is easy
to see that $\mathcal{W}_{1}$ must also be finite.

Letting $\mathcal{W}_{2}=\mathcal{W}\setminus\mathcal{W}_{1},$ for every
$X\in\mathcal{W}_{2}$ we choose $A_{X}\in\mathcal{I}\left(  \mathcal{A}%
\right)  $ such that $Y_{X}=A_{X}\cap f\left[  X\right]  $ is infinite. Note
that if $A\in\mathcal{A}$ then the set $\left\{  X\in\mathcal{W}_{2}\mid
A=A_{X}\right\}  $ must be finite. Since $\mathcal{A}$ is strongly tight we
can find $A\in\mathcal{I}\left(  A\right)  $ such that $A\cap Y_{X}%
\neq\emptyset$ for every $X\in\mathcal{W}_{2}.$ Since $f$ is a 
Kat\v{e}tov-morphism, we may conclude that $B=f^{-1}\left(  A\right)  $
belongs to $\mathcal{I}\left(  \mathcal{B}\right)  $ and $B\cap
X\neq\emptyset$ for every $X\in\mathcal{W}_{2}.$ Clearly $B\cup%
{\textstyle\bigcup}
\mathcal{W}_{1}$ has the desired properties.
\end{proof}


\begin{center}
$\star\star\star$
\end{center}


\noindent We next define properties of ideals $\I$ relevant for the investigation of 
Mathias forcing $\MM (\I)$:

\begin{definition}
Let $\mathcal{I}$ be an ideal in $\omega.$
\begin{enumerate}
\item We say $\mathcal{I}$ is \emph{Canjar }if and only if for every $\left\{
X_{n}\mid n\in\omega\right\}  \subseteq\left(  \mathcal{I}^{<\omega}\right)
^{+}$ there are $Y_{n}\in\left[  X_{n}\right]  ^{<\omega}$ such that $%
{\textstyle\bigcup\limits_{n\in\omega}}
Y_{n}\in\left(  \mathcal{I}^{<\omega}\right)  ^{+}$.
\item We say $\mathcal{I}$ is \emph{Hurewicz }if and only if for every
$\left\{  X_{n}\mid n\in\omega\right\}  \subseteq\left(  \mathcal{I}^{<\omega
}\right)  ^{+}$ there are $Y_{n}\in\left[  X_{n}\right]  ^{<\omega}$ such that
$%
{\textstyle\bigcup\limits_{n\in A}}
Y_{n}\in\left(  \mathcal{I}^{<\omega}\right)  ^{+}$ for every $A\in\left[
\omega\right]  ^{\omega}.$
\end{enumerate}
\end{definition}

Clearly, every Hurewicz ideal is Canjar. Moreover, 
every $F_{\sigma}$ ideal is Hurewicz~\cite{MobandMad}, and every Borel Canjar ideal is
$F_{\sigma}$ \cite{MathiasForcingandCombinatorialCoveringPropertiesofFilters,
CanjarFilters}, so that for Borel ideals, $F_\sigma$, Canjar, and Hurewicz agree.
In general, the two notions are quite different: Canjar~\cite{Canjar} constructed {\em Canjar ultrafilters}
(i.e. ultrafilters whose dual maximal ideals are Canjar) under \textsf{CH}, while it
is easy to see that no maximal ideal can be Hurewicz.

If $\mathcal{I}$ is an ideal, we denote by $\mathbb{M}\left(  \mathcal{I}\right)
$ the \emph{Mathias forcing with }$\mathcal{I}$, that is, the set of all pairs $\left(
s,A\right)  $ such that $s\in\left[  \omega\right]  ^{<\omega}$ and
$A\in\mathcal{I}$, ordered by $\left(  s,A\right)  \leq\left(
t,B\right)  $ if $t\subseteq s,$ $B\subseteq A$ and $\left(  s\setminus
t\right)  \cap B=\emptyset,$ where $\left(  s,A\right) , \left(t,B\right)  \in \MM (\I)$.
It is easy to see that $\mathbb{M}\left(\mathcal{I}\right)  $ destroys the tallness of $\mathcal{I}.$ 

We mention the following important results regarding Canjar and Hurewicz ideals:
\begin{itemize}
\item $\mathcal{I}$ is Canjar if and only if $\mathbb{M}\left(  \mathcal{I}\right)  $ does 
not add a dominating real~\cite{MathiasPrikryandLaverPrikryTypeForcing}.
\item $\mathcal{I}$ is Hurewicz if and only if $\mathbb{M}\left(\mathcal{I}\right)  $ preserves 
all unbounded families of the ground model~\cite{MathiasForcingandCombinatorialCoveringPropertiesofFilters}.
\item $\mathcal{I}$ is Canjar if and only if $\mathcal{I}$ is a Menger subspace of 
$\wp\left(  \omega\right)$~\cite{MathiasForcingandCombinatorialCoveringPropertiesofFilters}.
\item $\mathcal{I}$ is Hurewicz if and only if $\mathcal{I}$ is a Hurewicz subspace of 
$\wp\left(  \omega\right)$~\cite{MathiasForcingandCombinatorialCoveringPropertiesofFilters}.
\end{itemize}

For \textsf{MAD} families we have:

\begin{proposition}   \label{ShelahSteprans-Hurewicz}
Every Shelah-Stepr\={a}ns \textsf{MAD} family is Hurewicz.
\end{proposition}

\begin{proof}
Let $\mathcal{A}$ be a Shelah-Stepr\={a}ns \textsf{MAD} family and $\left\{
X_{n}\mid n\in\omega\right\}  \subseteq\left(  \mathcal{I}\left(
\mathcal{A}\right)  ^{<\omega}\right)  ^{+}.$ Note that if $B\in
\mathcal{I}\left(  \mathcal{A}\right)  $ then $\left\{  X_{n}\setminus\left[
B\right]  ^{<\omega}\mid n\in\omega\right\}  \subseteq\left(  \mathcal{I}%
\left(  \mathcal{A}\right)  ^{<\omega}\right)  ^{+}.$ Using the  Lemma~\ref{SScharacterization}, 
we can thus recursively
find $\left\{  B_{n}\mid n\in\omega\right\}  \subseteq\mathcal{I}\left(
\mathcal{A}\right)  $ with the following properties:
\begin{enumerate}
\item If $n\neq m$ then there is no $A\in\mathcal{A}$ that has infinite
intersection with both $B_{n}$ and $B_{m}.$
\item $X_n \cap [B_m \sem k]^{< \omega}$ is infinite for all $n,m,k \in \omega$.
\end{enumerate}
For every $n\in\omega$ let $Y_{n}\in\left[  X_{n}\right]  ^{<\omega}$ such
that $Y_{n}\cap\left[  B_{i}\right]  ^{<\omega}\neq\emptyset$ for every $i\leq
n$.  
It is then easy to see that if $D\in\left[  \omega\right]  ^{\omega}$ then
${\textstyle\bigcup\limits_{n\in D}}
Y_{n}\in\left(  \mathcal{I}\left(  \mathcal{A}\right)  ^{<\omega}\right)
^{+}.$
\end{proof}

This implication clearly fails for ideals in general because, for example,
\textsf{fin}$\times$\textsf{fin} is Shelah-Stepr\=ans, but not Canjar.


\begin{center}
$\star\star\star$
\end{center}


\noindent We will consider some more properties of ideals:

\begin{definition}
Let $\mathcal{J}$ be an ideal.

\begin{enumerate}
\item $\mathcal{J}$ is \emph{Laflamme }if $\mathcal{J}$ can not be extended to
an $F_{\sigma}$ ideal$.$

\item $\mathcal{J}$ is \emph{not-P }if $\mathcal{J}$ can not be
extended to an analytic $P$-ideal$.$

\item $\mathcal{J}$ is \emph{\textsf{fin}$\times$\textsf{fin}-like }if $\mathcal{J}
\nleq_{K}\mathcal{I}$ for every analytic
ideal $\I$ such that \textsf{fin}$\times$\textsf{fin }$\nleq_{K}\mathcal{I}$.
\end{enumerate}
\end{definition}

By results of Solecki, Laczkovich and Rec\l aw it can be proved that no
$F_{\sigma\delta}$ ideal is Kat\v{e}tov above \textsf{fin}$\times$\textsf{fin}
(see \cite{FiltersandSequences} and
\cite{IdeallimitsofSequencesofContinuousFunctions}).
Since every analytic $P$-ideal is $F_{\sigma\delta}$ (see
\cite{AnalyticIdealsandtheirApplications}),  it follows that every
\textsf{fin}$\times$\textsf{fin-}like ideal is not-P. Furthermore, in
\cite{KatetovandKatetovBlassOrdersFsigmaIdeals} it was proved that every
$F_{\sigma}$ ideal is contained in an analytic $P$-ideal, so every
not-P ideal is Laflamme. By Theorem~\ref{SSdefinable}, every Shelah-Stepr\=ans
ideal is \textsf{fin}$\times$\textsf{fin-}like.\footnote{This uses \textsf{LC}; note that if in the
definition of \textsf{fin}$\times$\textsf{fin-}like we only quantified over Borel ideals,
then this implication would be true in \textsf{ZFC} by Theorem~\ref{SSBorel}.}

We have the following:

\begin{lemma}    \label{Laflamme-notP-char}
Let $\mathcal{J}$ be an ideal.

\begin{enumerate}
\item $\mathcal{J}$ is Laflamme if and only if $\mathcal{J} \nleq_{K}\mathcal{I}$ for every $F_{\sigma}%
$-ideal $\mathcal{I}.$

\item $\mathcal{J}$ is not-P if and only if $\mathcal{J} \nleq_{KB}\mathcal{I}$ for every analytic
$P$-ideal $\mathcal{I}.$
\end{enumerate}
\end{lemma}

\begin{proof}
Clearly if $\mathcal{J} \nleq_{K}\mathcal{I}$ for every $F_{\sigma}$-ideal
$\mathcal{I}$ then $\mathcal{J}$ is Laflamme. So assume that $\mathcal{J}$ is
Laflamme, let $\mathcal{I}$ be an $F_{\sigma}$ ideal, and let $f:\omega
\longrightarrow\omega$. We must show that $f$ is not a Kat\v{e}tov-morphism
from $\left(  \omega,\mathcal{I}\right)  $ to $\left(  \omega,\mathcal{J} \right).$ Define $\mathcal{I'=}\left\{
X\mid f^{-1}\left(  X\right)  \in\mathcal{I}\right\}  .$ Let $\varphi$ be a
lower semicontinuous submeasure such that $\mathcal{I}=$ \textsf{fin}$\left(
\varphi\right)  .$ For every $n\in\omega\,\ $we define $C_{n}=\left\{
X\mid\varphi\left(  f^{-1}\left(  X\right)  \right)  \leq n\right\}  .$ It is
easy to see that each $C_{n}$ is a closed set and $\mathcal{I'}=%
{\textstyle\bigcup\limits_{n\in\omega}} C_{n}.$ Since $\mathcal{J} $ is not contained in
$\mathcal{I'}$ the result follows.

For the second part, it is clear that if $\mathcal{J} \nleq_{KB}\mathcal{I}$ for every analytic $P$-ideal $\mathcal{I}$ then
$\mathcal{J}$ is not-P. Assume $\mathcal{J}$ is not-P, let $\mathcal{I}$ be an analytic $P$-ideal, and let
$f:\omega\longrightarrow\omega$ be finite to one. We must show that $f$ is not a
Kat\v{e}tov-morphism from $\left(  \omega,\mathcal{I}\right)  $ to $\left(
\omega,\mathcal{J}  \right)  .$ Let $\varphi$ be a
lower semicontinuous submeasure such that $\mathcal{I}=$ \textsf{Exh}$\left(
\varphi\right)  .$ Define $\sigma:\wp\left(  \omega\right)
\longrightarrow\mathbb{R\cup}\left\{  \infty\right\}  $  by
$\sigma\left(  A\right)  =\varphi\left(  f^{-1}\left(  A\right)  \right)  .$
It is easy to see that $\sigma$ is a lower semicontinuous submeasure (it is a submeasure by definition and the lower semicontinuity follows since $f$ is finite to one). Since
$\mathcal{J}$ is not-P, there is $B\in\mathcal{J}$ such that $B\notin \textsf{Exh}\left(  \sigma\right)  $
which implies that $f^{-1}\left(  B\right)  \notin\mathcal{I}.$
\end{proof}

Note in this context that we could have defined ``\textsf{fin}$\times$\textsf{fin}-like" as we did define 
``Laflamme" and ``not-P": namely, $\mathcal{J}$ is \textsf{fin}$\times$\textsf{fin}-like  if $\mathcal{J}
\not\subseteq \mathcal{I}$ for every analytic
ideal $\I$ such that \textsf{fin}$\times$\textsf{fin }$\nleq_{K}\mathcal{I}$. To see the nontrivial
direction of this equivalence, assume there is an analytic ideal $\I$ such that \textsf{fin}$\times$\textsf{fin} $\nleq_K  \I$
and $\J \leq_K \I$ as witnessed by the Kat\v etov reduction $f$. Define $\I ' = \{ A | f^{-1} [A] \in \I \}$
and note that $\J \subseteq \I' \leq_K \I$, that \textsf{fin}$\times$\textsf{fin} $\nleq_K  \I '$ and that $\I'$ is still analytic.

Also notice that all three properties defined here are (trivially) upwards closed in the Kat\v etov order.

We now investigate the connection between \textsf{fin}$\times$\textsf{fin}-likeness and weak tightness. Define $C_{n}=\left\{  \left(  n,m\right)
\mid m\in\omega\right\}  .$

\begin{definition}
We define the ideal $\mathcal{WT}$ on $\omega\times\omega$ as follows:

\begin{enumerate}
\item $\mathcal{WT}\upharpoonright C_{n}$ is a copy of \textsf{fin}$\times
$\textsf{fin} (for every $n\in\omega$). 

\item $A \in \mathcal{WT}$ iff $A \cap C_n \in \mathcal{WT}\upharpoonright C_{n}$ for all $n$ and $A \cap C_n$ is finite for all but finitely many $n$.
\end{enumerate}
\end{definition}
Clearly, $\mathcal{WT}$ is not weakly tight as witnessed by the $C_n$.
Also note that if $B\subseteq\omega\times\omega$ has infinite intersection with
infinitely many columns then $B\in\mathcal{WT}^{+}.$\ \ \ 

\begin{proposition}   \label{WTfintimesfin}
$\mathcal{WT}$ is strictly Kat\v{e}tov below \textsf{fin}$\times$%
\textsf{fin}$.$
\end{proposition}

\begin{proof}
Note that the identity mapping witnesses $\mathcal{WT}\leq_{K}$\textsf{fin}%
$\times$\textsf{fin}$.$ Now, we will show $\mathsf{II}\ $has a winning
strategy in $\mathcal{L}\left(  \mathcal{WT}\right)  .$ This is easy, since
no $C_{n}$ belongs to $\mathcal{WT}$ and therefore $\mathsf{II}\ $can play in such a way that the
set she constructed at the end intersects infinitely often all the $C_{n},$ so it
can not be an element of $\mathcal{WT}.$
\end{proof}

For \textsf{MAD} families we have the following characterization:

\begin{lemma}
If $\mathcal{A}$ is a \textsf{MAD} family then $\mathcal{A}$ is weakly tight
if and only if $\mathcal{I}\left(  \mathcal{A}\right)  $ $\mathcal{\nleq}_{K}$
$\mathcal{WT}.$
\end{lemma}

\begin{proof}
We will prove that $\mathcal{A}$ is not weakly tight if and only if
$\mathcal{I}\left(  \mathcal{A}\right)  $ $\mathcal{\leq}_{K}$ $\mathcal{WT}.$
One direction follows from the fact that $\mathcal{WT}$ is not weakly tight and from the
upwards closure of weak tightness in the Kat\v etov order (see Lemma~\ref{tightKatetovupwards}).

So assume $\mathcal{A}$ is not weakly tight. So there is a partition
$X=\left\{  X_{n}\mid n\in\omega\right\}  \subseteq\mathcal{I}\left(
\mathcal{A}\right)  ^{+}$ such that if $A\in\mathcal{A}$ then $A\cap X_{n}$ is
finite for almost all $n\in\omega.$ Since $\mathcal{A}\upharpoonright X_{n}$
is an \textsf{AD} family, we know $\mathcal{A}\upharpoonright X_{n}$
$\mathcal{\leq}_{KB}$ \textsf{fin}$\times$\textsf{fin}, so for each $n\in\omega$
fix a Kat\v{e}tov-Blass-morphism $h_{n}:C_{n}\longrightarrow X_{n}$ from $\left(
C_{n},\mathcal{WT\upharpoonright}C_{n}\right)  $ to $\left(  X_{n}%
,\mathcal{A}\upharpoonright X_{n}\right)$ (in fact, it is well-known that we can choose
one to one $h_n$).  Letting $h={\textstyle\bigcup}
h_{n}$ we will show $h$ is a Kat\v{e}tov-morphism from $\left(  \omega
\times\omega,\mathcal{WT}\right)  $ to $\left(  \omega,\mathcal{A}\right)  .$
If $A\in\mathcal{A}$ then we can find $B\in X^{\perp}$ and a finite set
$F\subseteq\omega$ such that $A={\textstyle\bigcup\limits_{n\in F}}
\left(  A\cap X_{n}\right)  \cup B.$ Clearly $h^{-1}\left(  B\right)  \in$
$\mathcal{WT}$ since $h^{-1}\left(  B\right)  \in$ $\emptyset\times
$\textsf{fin} and $h^{-1}\left(  A\cap X_{n}\right)  =h_{n}^{-1}\left(  A\cap
X_{n}\right)  $ which is an element of $\mathcal{WT}$ since $h_{n}$ is a
Kat\v{e}tov-morphism. Therefore $h^{-1}\left(  A\right)  \in\mathcal{WT}.$
\end{proof}

We conclude:

\begin{corollary}
If $\mathcal{A}$ is \textsf{fin}$\times$\textsf{fin}-like then $\mathcal{A}$
is weakly tight.
\end{corollary}

This implication fails for ideals in general:

\begin{proposition}
There is a \textsf{fin}$\times$\textsf{fin}-like ideal that is not weakly tight.
\end{proposition}

\begin{proof}
Let $\J$ be any ideal not contained in an analytic ideal (e.g., any maximal ideal).
Define the ideal $\mathcal{WT} (\J)$ on $\omega\times\omega$ as follows:
\begin{enumerate}
\item $\mathcal{WT} (\J)\upharpoonright C_{n}$ is a copy of $\J$.
\item $A \in \mathcal{WT} (\J)$ iff $A \cap C_n \in \mathcal{WT}(\J)\upharpoonright C_{n}$ for all $n$ and $A \cap C_n$ is finite for all but finitely many $n$.
\end{enumerate}
As for $\mathcal{WT}$ we see that the $C_n$ witness that $\mathcal{WT}(\J)$ is not weakly tight.
Also note that as in Proposition~\ref{WTfintimesfin} we see that $\mathcal{WT}(\J)$ is strictly Kat\v etov below \textsf{fin}$\times$\textsf{fin}.
To see that $\mathcal{WT}(\J)$ is \textsf{fin}$\times$\textsf{fin}-like, let $\I$ be an analytic ideal on $\omega \times \omega$
with \textsf{fin}$\times$\textsf{fin} $ \nleq_K \I$. We need to see that $\mathcal{WT}(\J) \not\subseteq \I$. Since 
\textsf{fin}$\times$\textsf{fin} $ \nleq_K \I$ there is $A \in$ \textsf{fin}$\times$\textsf{fin} with $A \notin \I$, say
$A = B \cup C$ where $B$ meets only finitely many $C_n$ and $C \cap C_n$ is finite for all $n$.
If $C \notin \I$, we are done because $C \in \mathcal{WT}(\J)$. So assume $B \notin \I$. Thus, for some $n$,
$B \cap C_n \notin \I$ and $\I \upharpoonright C_n$ is a proper analytic ideal. Then $\mathcal{WT} (\J)\upharpoonright C_{n} \not\subseteq \I \upharpoonright C_{n}$,
and we are also done.
\end{proof}


\begin{center}
$\star\star\star$
\end{center}


\noindent We briefly discuss the connection between the properties introduced so far and indestructibility by forcing:

\begin{proposition}   \label{SSindestructibility}
Let $\I$ be an ideal. Also let $\mathcal{J}$ be a $\sigma$-ideal in $\omega^{\omega}$
for which the forcing $\mathbb{P}_{\mathcal{J}}=Borel\left(  \omega^{\omega}\right)
/\mathcal{J}$ is proper, has the continuous reading of names and does not add a
dominating real (below any condition). 
\begin{enumerate}
\item Assume $\I$ is  \textsf{fin}$\times$\textsf{fin}-like and $tr (\J)$ is an analytic ideal. 
Then $\mathcal{I}$ is $\mathbb{P}_{\mathcal{J}}$-indestructible.

\item \textsf{(LC)} If $\I$ is Shelah-Stepr\=ans and $\mathcal{J}\in$ \textsf{L}$\left(  \mathbb{R}\right)  $
then $\mathcal{I}$ is $\mathbb{P}_{\mathcal{J}}$-indestructible.

\item In particular, if $\I$ is Shelah-Stepr\=ans (or just \textsf{fin}$\times$\textsf{fin}-like), then $\I$ is Cohen, random, and
Sacks indestructible. 

\item If $\I$ is a not-P ideal then $\I$ is random and Sacks indestructible.
\end{enumerate}
\end{proposition}

\begin{proof}
1. Let $\mathcal{J}$ be a $\sigma$-ideal in $\omega^{\omega}$ such
that $\mathbb{P}_{\mathcal{J}}$ is proper and has the continuous
reading of names. If there is $B\in\mathbb{P}_{\mathcal{J}}$ such that forcing
below $B$ destroys $\mathcal{I},$ then there is $X\in tr\left(  \mathcal{J}%
\right)  ^{+}$ such that $\mathcal{I}  \leq
_{K}tr\left(  \mathcal{J}\right)  \upharpoonright X$ (see Proposition~\ref{destructibilitycharacterization}). 
Since $\I$ is \textsf{fin}$\times$\textsf{fin}-like and $tr (\J)$ is analytic, \textsf{fin}$\times
$\textsf{fin }$\leq_{\mathsf{K}}tr\left(  \mathcal{J}\right)  \upharpoonright X$ follows,
and so $\mathbb{P}_{\mathcal{J}}$ must add a dominating real below some
condition.

2. This is similar, using Theorem~\ref{SSdefinable}.

3. Noting that the ideals $tr (\M)$, $tr (\N)$, and $tr (\mathsf{ctble})$ are all Borel, this follows from the previous items.

4. In~\cite[Theorem 3.4]{CardinalInvariantsofAnalyticPIdeals} it was proved
that $tr\left(  \mathcal{N}\right)  \leq_{{K}}\mathcal{Z}$. So every
not-P ideal is random indestructible. It is well known that random
indestructibility entails Sacks indestructibility (because $tr  (\mathsf{ctble}) \leq_K tr (\N)$, see also~\cite{ForcingIndestructibilityofMADFamilies}).
\end{proof}

\begin{center}%
\begin{figure}
[ptb]
\begin{center}
\includegraphics[scale=0.40]{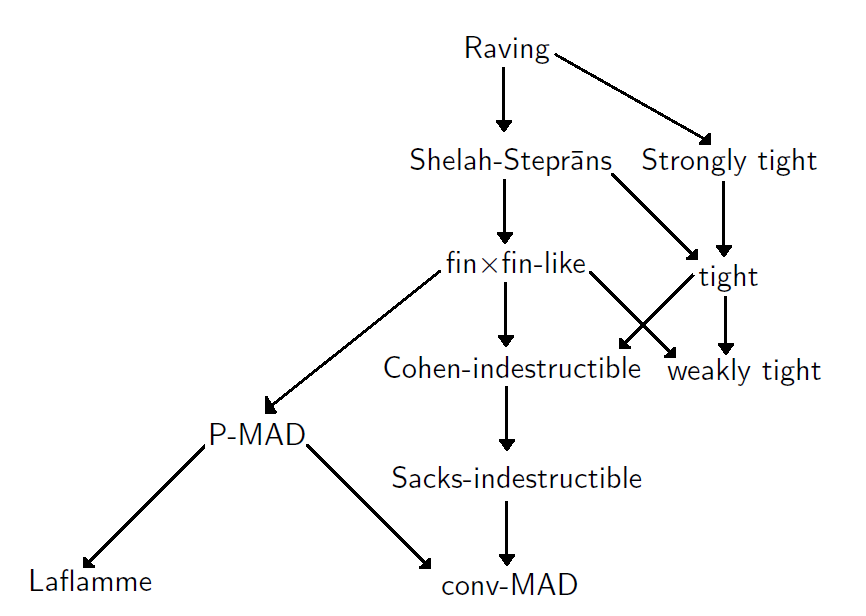}
\end{center}
\end{figure}

\end{center}


\subsection{Generic \textsf{MAD} families}
\label{genericMAD}


Let $\mathcal{I}$ be a (perhaps improper) tall ideal. We define $\mathbb{P}%
_{\mathsf{MAD}}\left(  \mathcal{I}\right)  $ as the set of all countable
\textsf{AD} families contained in $\mathcal{I}$, ordered by inclusion. It is easy to see
that this is a $\sigma$-closed forcing adding a \textsf{MAD} family
contained in $\mathcal{I},$ which we will denote by $\mathcal{A}_{gen}\left(
\mathcal{I}\right)  .$ By $\mathbb{P}_{\mathsf{MAD}}$ we denote $\mathbb{P}%
_{\mathsf{MAD}}\left(  \wp\left(  \omega\right)  \right)  ,$ which is the set
of all countable \textsf{AD} families ordered by inclusion; we will 
denote by $\mathcal{A}_{gen}$ the generic \textsf{MAD} family.
We shall use such generic objects later for non-implications  (see Subsection~\ref{non-implications}).

\begin{definition}
Let $\mathcal{I}$ be an ideal. We say that $\mathcal{I}$ is \emph{nowhere
Shelah-Stepr\={a}ns }if no restriction of$\mathcal{\ I}$ is Shelah-Stepr\={a}ns.
\end{definition}

It is easy to see that \textsf{nwd}$,$ $tr($\textsf{ctble}$),$ $tr\left(
\mathcal{N}\right)  ,$ $tr\left(  \mathcal{K}_{\sigma}\right)  $ and every
$F_{\sigma}$-ideal are nowhere Shelah-Stepr\={a}ns.

\begin{lemma}   \label{Katetovgenericlemma}
Let $\mathcal{I},\mathcal{J}$ be two ideals such that $\mathcal{I}$ is nowhere
Shelah-Stepr\={a}ns and $\mathcal{J}\nleq_{K}\mathcal{I}.$ Let
$\mathcal{A\subseteq J}$ be a countable \textsf{AD} family and let $f:\left(
\omega,\mathcal{I}\right)  \longrightarrow\left(  \omega,\mathcal{I}\left(
\mathcal{A}\right)  \right)  $ be a Kat\v{e}tov-morphism. Then there is
$B\in\mathcal{J}\cap\mathcal{A}^{\perp}$ such that $f^{-1}\left(  B\right)
\in\mathcal{I}^{+}.$
\end{lemma}

\begin{proof}
Let $\mathcal{A}=\left\{  A_{n}\mid n\in\omega\right\}  .$ We know $f$ is a
Kat\v{e}tov-morphism, so the set $\left\{  f^{-1}\left(  A_{n}\right)  \mid
n\in\omega\right\}  $ is contained in $\mathcal{I}.$ Since $\mathcal{J}%
\nleq_{K}\mathcal{I}$ there is $D\in\mathcal{J}$ such that $C=f^{-1}\left(
D\right)  \in\mathcal{I}^{+}.$ Since $\mathcal{I}\upharpoonright C$ is not
Shelah-Stepr\={a}ns, there is $X\in\left(  \left(  \mathcal{I}\upharpoonright
C\right)  ^{<\omega}\right)  ^{+}$ such that no element of $\mathcal{I}$
contains infinitely many elements of $X.$ For each $n\in\omega$ we choose
$s_{n}\in X$ such that $s_{n}\cap\left(  f^{-1}\left(  A_{0}\cup
... \cup A_{n}\right)  \right)  =\emptyset.$ We then know that $E={\textstyle\bigcup}
s_{n}\in\mathcal{I}^{+}.$ It is easy to see that $B=f\left[  E\right]  $ has
the desired properties.
\end{proof}

We conclude:

\begin{corollary}    \label{Katetovgenericcorollary}
Let $\mathcal{I},\mathcal{J}$ be ideals such that $\mathcal{I}$ is nowhere
Shelah-Stepr\={a}ns and $\mathcal{J}\nleq_{K}\mathcal{I}.$

\begin{enumerate}
\item $\mathbb{P}_{\mathsf{MAD}}\left(  \mathcal{J}\right)  $ forces that
$\mathcal{A}_{gen}\left(  \mathcal{J}\right)  $ is not Kat\v{e}tov below
$\mathcal{I}.$

\item The Continuum Hypothesis implies that there is a \textsf{MAD} family
$\mathcal{A}\subseteq\mathcal{J}$ such that $\mathcal{I}\left(  \mathcal{A}%
\right)  \nleq_{K}\mathcal{I}.$
\end{enumerate}
\end{corollary}

In particular, $\mathcal{A}_{gen}($\textsf{nwd}$)$ is a Cohen destructible
\textsf{MAD} family that is Miller and random indestructible. For more on this
type of results, the reader may consult
\cite{ForcingIndestructibilityofMADFamilies}.

\begin{proposition}
$\mathbb{P}_{\mathsf{MAD}}$ forces that$\ \mathcal{A}_{gen}$ is raving.
\end{proposition}

\begin{proof}
Let $\mathcal{B\in}$ $\mathbb{P}_{\mathsf{MAD}}$ and $X=\left\{  X_{n}\mid
n\in\omega\right\}  $ such that $\mathcal{B}$ forces that $X$ is locally
finite according to $\mathcal{I}\left(  \mathcal{A}_{gen}\right)  .$ Let
$\mathcal{B}=\left\{  B_{n}\mid n\in\omega\right\}  $ and we define
$E_{n}=B_{0}\cup...\cup B_{n}$ for every $n\in\omega.$ We can then find an
interval partition $\mathcal{P}=\left\{  P_{n}\mid n\in\omega\right\}  $ of $\omega$ such
that if $i\in P_{n+1}$ then $E_{n}$ does not intersect every element of
$X_{i}.$ For every $i\in\omega$ we choose $s_{i}\in X_{i}$ as follows: if
$i\in P_{0}$ let $s_{i}$ be any element of $X_{i}$ and if $i\in P_{n+1}$ we
choose $s_{i}\in X_{i}$ such that $s_{i}\cap E_{n}=\emptyset.$ Let $A={\textstyle\bigcup\limits_{n\in\omega}}
s_{n}$; then $A\in\mathcal{B}^{\perp}$ and the condition $\mathcal{B\cup
}\left\{  A\right\}  \in\mathbb{P}_{\mathsf{MAD}}$ is the extension of
$\mathcal{B}$ we were looking for.
\end{proof}

Our motivation for studying the forcing $\mathbb{P}_{\mathsf{MAD}}$ comes
from the following results about generic ultrafilters:

\begin{theorem}
[Todorcevic, see~\cite{SemiselectiveCoideals}]
An ultrafilter $\mathcal{U}$ is
$\wp\left(  \omega\right)  \setminus$ \textsf{fin} generic over \textsf{L}%
$\left(  \mathbb{R}\right)  $ if and only if $\mathcal{U}$ is Ramsey.
\end{theorem}

\begin{theorem}
[Chodounsk\'y, Zapletal, see~\cite{IdealsGenUF}]
Let $\mathcal{I}$ be an $F_{\sigma}$-ideal and
$\mathcal{U}$ an ultrafilter. $\mathcal{U}$ is $\wp\left(  \omega\right)
\setminus\mathcal{I}$ generic over \textsf{L}$\left(  \mathbb{R}\right)  $ if
and only if $\mathcal{I\cap U=\emptyset}$ and for every closed set
$\mathcal{C}$ if $\mathcal{C}\cap\mathcal{U=\emptyset}$ then there is
$A\in\mathcal{U}$ such that $A\cap Y\in\mathcal{I}$ for every $Y\in$
$\mathcal{C}.$
\end{theorem}

It would be interesting to find a similar characterization of the
$\mathbb{P}_{\mathsf{MAD}}$ generics over \textsf{L}$\left(  \mathbb{R}%
\right)  :$

\begin{problem}
Is there a combinatorial characterization of $\mathcal{A}$ (or of
$\mathcal{I}\left(  \mathcal{A}\right)  $)$\ $where $\mathcal{A}%
\ $is$\ \mathbb{P}_{\mathsf{MAD}}$ generic over \textsf{L}$\left(
\mathbb{R}\right)  $?
\end{problem}

A natural candidate would be ``raving", but even this strong property might still
be too weak to capture the full extent of genericity.


\subsection{Existence versus non-existence}
\label{existence-nonexistence}


An important result of Raghavan says:

\begin{theorem}
[\cite{AModelwithnoStronglySeparableAlmostDisjointFamilies}]It is consistent
that there are no Shelah-Stepr\={a}ns \textsf{MAD} families.
\end{theorem}


For strongly tight families, we first prove:

\begin{proposition}
If $\mathcal{A}$ is strongly tight then $\mathfrak{d\leq}$ $\left\vert
\mathcal{A}\right\vert .$
\end{proposition}

\begin{proof}
Let $\left\{  A_{n}\mid n\in\omega\right\}  $ be a partition of $\omega$
contained in $\mathcal{A}$ and for each $n\in\omega$ let $P_{n}=\left\{
A_{n}\left(  i\right)  \mid i\in\omega\right\}  $ be a partition of $A_{n}$ in
infinite pieces. Given $A\in\mathcal{I}\left(  \mathcal{A}\right)  $ we define
a function $f_{A}:\omega\longrightarrow\omega$ given by $f_{A}\left(
n\right)  =0$ if $A\cap A_{n}$ is infinite and in the other case $f_{A}\left(
n\right)  =\max\left\{  i\mid A\cap A_{n}\left(  i\right)  \neq\emptyset
\right\}  +1.$ We claim that $\left\{  f_{A}\mid A\in\mathcal{I}\left(
\mathcal{A}\right)  \right\}  $ is a dominating family. Assume this is not the
case, so there is $g:\omega\longrightarrow\omega$ not dominated by any of the
$f_{A}.$ For each $n\in\omega$ define $X_{n}=A_{n}\left(  g\left(  n\right)
\right)  $ and $X=\left\{  X_{n}\mid n\in\omega\right\}  .$ Since
$\mathcal{A}$ is strongly tight there must be $A\in\mathcal{I}\left(
\mathcal{A}\right)  $ such that $A\cap X_{n}\neq\emptyset$ for every
$n\in\omega.$ Pick any $m$ such that $f_{A}\left(  m\right)  <g\left(
m\right)  $; then $A\cap A_{m}\left(  g\left(  m\right)
\right)  =\emptyset$ so that $A\cap X_{m}=\emptyset$, which is a contradiction.
\end{proof}

We conclude:

\begin{corollary}
There are no strongly tight \textsf{MAD} families in the Cohen model.
\end{corollary}

\begin{proof}
If there were a strongly tight \textsf{MAD}, it would have size continuum, by the previous
proposition. But since it would also be tight,
it should have size $\omega_{1}$ (recall that tight \textsf{MAD} families are
Cohen indestructible).
\end{proof}

We will later prove that there are Shelah-Stepr\={a}ns \textsf{MAD} families
in the Cohen model, so Shelah-Stepr\={a}ns does not imply strong tightness
(see the discussion after Theorem~\ref{diamondMAD}). As mentioned in the
Introduction (Problem~\ref{Stepransproblem}), it is still open whether tight
\textsf{MAD} families exist in \textsf{ZFC}.

To provide an example for a class of \textsf{MAD} families existing in \textsf{ZFC}, 
we make the following definition:

\begin{definition}
Let $\mathcal{I}$ be an ideal and $\mathcal{A}$ a \textsf{MAD }family. We say
that $\mathcal{A}$ is $\mathcal{I}$\emph{-\textsf{MAD} }if $\mathcal{I}\left(
\mathcal{A}\right)  \nleq_{K}\mathcal{I}.$
\end{definition}

It is well known that no \textsf{MAD} family is \textsf{fin}$\times
$\textsf{fin}-\textsf{MAD}. Clearly $\A$ is Laflamme iff it is $\I$-\textsf{MAD} for every
$F_\sigma$-ideal $\I$, and if $\A$ is $\I$-\textsf{MAD} for every analytic $P$-ideal then 
$\A$ is not-P (Lemma~\ref{Laflamme-notP-char}). Also note that being Cohen indestructible is
equivalent to being \textsf{nwd}-\textsf{MAD }(recall that \textsf{nwd }denotes the
ideal of nowhere dense sets of the rational numbers). 

We denote by \textsf{conv
}the ideal in $\left[  0,1\right]  \cap\mathbb{Q}$ generated by all sequences
converging to a real number. Since \textsf{conv} $\leq_K tr (\mathsf{ctble})$, every
Sacks indestructible \textsf{MAD} family is \textsf{conv}-\textsf{MAD}. On the other hand,
if $\A$ is a $\PP_{\mathsf{MAD}} (tr (\mathsf{ctble}))$-generic \textsf{MAD} family, then
by Corollary~\ref{Katetovgenericcorollary}, $\A$ is a Sacks destructible \textsf{conv}-\textsf{MAD} family
because $tr (\mathsf{ctble}) \not\leq_K $ \textsf{conv}. We will now prove that there is a
\textsf{conv}-\textsf{MAD} family; this result is based on the proof of
Proposition 2 of \cite{MADFamiliesandtheRationals}. We need the following lemma:  

\begin{lemma}
Let $\mathcal{A}$ be an \textsf{AD }family of size less than $\mathfrak{c}$.
If $f$ is a Kat\v{e}tov-morphism from $(\left[  0,1\right]  \cap\mathbb{Q}%
,$\textsf{conv}$)$ to $\left(  \omega,\mathcal{I}\left(  \mathcal{A}\right)
\right)  $ then there is $B\in\mathcal{A}^{\perp}$ such that $f^{-1}\left(
B\right)  \notin$ \textsf{conv}.
\end{lemma}

\begin{proof}
For every $A\in\mathcal{A}$ let $F_{A}$ be the set of accumulation points of
$f^{-1}\left(  A\right)  ,$ and note that each $F_{A}$ is finite since
$f^{-1}\left(  A\right)  $ can be covered by finitely many converging
sequences. Since $\left[  0,1\right]  $ can be partitioned into $\mathfrak{c}%
$-many perfect pairwise disjoint sets, we can find a perfect set
$C\subseteq\left[  0,1\right]  $ such that $C\cap F_{A}=\emptyset$ for every
$A\in\mathcal{A}.$ Let $D \subseteq [0,1] \cap \QQ$ be such that $C$ is the set of
accumulation points of $D$ and note that $D \cap f^{-1} (A)$ is finite for all $A \in \mathcal{A}$.
It is easy to see that $B=f\left[  D\right]  $ has the
desired properties.
\end{proof}

We conclude:

\begin{corollary}
There is a \textsf{conv}-\textsf{MAD} family.
\end{corollary}

The ideal \textsf{conv }is one of the few Borel ideals $\mathcal{I}$ for which we
can prove that there are $\mathcal{I}$-\textsf{MAD }families.


\subsection{Existence under diamond principles}
\label{diamond-existence}


Parametrized diamonds are strong guessing principles which can be used to construct 
\textsf{MAD} families with strong combinatorial properties. We first recall the
principles $\Diamond\left(  \mathfrak{b}\right)  $ and $\Diamond\left(
\mathfrak{d}\right)  $ from \cite{ParametrizedDiamonds}.

\begin{description}
\item[$\Diamond\left(  \mathfrak{b}\right)  $)] For every $\left\langle
F_{\alpha}:2^{\alpha}\longrightarrow\omega^{\omega}\right\rangle
_{\alpha<\omega_{1}}$ such that each $F_{\alpha}$ is Borel, there is
$g:\omega_{1}\longrightarrow\omega^{\omega}$ such that for every
$R\in2^{\omega_{1}},$ the set $\{\alpha\mid F_{\alpha}\left(  R\upharpoonright
\alpha\right)  $ $^{\ast}\ngeq g\left(  \alpha\right)  \}$ is stationary.

\item[$\Diamond\left(  \mathfrak{d}\right)  $)] For every $\left\langle
F_{\alpha}:2^{\alpha}\longrightarrow\omega^{\omega}\right\rangle
_{\alpha<\omega_{1}}$ such that each $F_{\alpha}$ is Borel, there is
$g:\omega_{1}\longrightarrow\omega^{\omega}$ such that for every
$R\in2^{\omega_{1}},$ the set $\{\alpha\mid F_{\alpha}\left(  R\upharpoonright
\alpha\right)  $ $\leq^{\ast}g\left(  \alpha\right)  \}$ is stationary.
\end{description}
Clearly $\Diamond\left(  \mathfrak{d}\right)  $ implies $\Diamond\left(
\mathfrak{b}\right)  .$ In \cite{ParametrizedDiamonds} it was proved that
$\Diamond\left(  \mathfrak{b}\right)  $ implies that $\mathfrak{a}=\omega_{1}$
and in \cite{GenericExistenceofMADfamilies} it was shown that $\Diamond\left(
\mathfrak{b}\right)  $ implies the existence of a tight \textsf{MAD} family.
We will now improve this result:

\begin{theorem}   \label{diamondMAD}
\begin{enumerate}
\item $\Diamond\left(  \mathfrak{b}\right)  $ implies there is a
Shelah-Stepr\={a}ns \textsf{MAD} family.

\item $\Diamond\left(  \mathfrak{d}\right)  $ implies there is a raving
\textsf{MAD} family.
\end{enumerate}
\end{theorem}

\begin{proof}
For every $\alpha<\omega_{1}$ fix an enumeration $\alpha=\left\{  \alpha
_{n}\mid n\in\omega\right\}  .$ We will first show that $\Diamond\left(
\mathfrak{b}\right)  $ implies there is a Shelah-Stepr\={a}ns \textsf{MAD}
family. With a suitable coding, the coloring $C$ will be defined for pairs
$t=\left(  \mathcal{A}_{t},X_{t}\right)  $ where $\mathcal{A}_{t}=\left\langle
A_{\xi}\mid\xi<\alpha\right\rangle $ and $X_{t}\subseteq\left[  \omega\right]
^{<\omega}$. We define $C\left(  t\right)  $ to be the constant $0$ function
in case $\mathcal{A}_{t}$ is not an almost disjoint family or $X_{t}%
\notin\left(  \mathcal{I}\left(  \mathcal{A}_{t}\right)  ^{<\omega}\right)
^{+}.$ In the other case, define an increasing function $C\left(  t\right)
:\omega\longrightarrow\omega$ such that if $n\in\omega$ then there is $s\in
X_{t}$ such that $s\subseteq C\left(  t\right)  \left(  n\right)  $ and
$s\cap\left(  A_{\alpha_{0}}\cup...\cup A_{\alpha_{n}}\cup n\right)
=\emptyset.$

Using $\Diamond\left(  \mathfrak{b}\right)  $ let $G:\omega_{1}\longrightarrow
\omega^{\omega}$ be a guessing sequence for $C.$ By changing $G$ if necessary,
we may assume that all the $G\left(  \alpha\right)  $ are increasing and if
$\alpha<\beta$ then $G\left(  \alpha\right)  <^{\ast}G\left(  \beta\right)  .$
We will now define our \textsf{MAD} family: start by taking $\left\{
A_{n}\mid n\in\omega\right\}  $ a partition of $\omega.$ Having defined
$A_{\xi}$ for all $\xi<\alpha,$ we proceed to define $A_{\alpha}%
=\bigcup\limits_{n\in\omega}\left(  G\left(  \alpha\right)  \left(  n\right)
\backslash A_{\alpha_{0}}\cup...\cup A_{\alpha_{n}}\right)  $ in case this is
an infinite set, otherwise just take any $A_{\alpha}$ that is almost disjoint
from $\{ A_\beta \mid \beta < \alpha \}.$ We will see that $\mathcal{A}$ is a
Shelah-Stepr\={a}ns \textsf{MAD} family. Let $X\in\left(  \mathcal{I}\left(
\mathcal{A}\right)  ^{<\omega}\right)  ^{+}$. Consider the branch $R=\left(
\left\langle A_{\xi}\mid\xi<\omega_{1}\right\rangle ,X\right)  $ and pick
$\beta>\omega$ such that $C\left(  R\upharpoonright\beta\right)  $ $^{\ast
}\ngeq G\left(  \beta\right)  .$ It is easy to see that $A_{\beta}$ contains
infinitely many elements of $X.$

Now we will prove that $\Diamond\left(  \mathfrak{d}\right)  $ implies there
is a raving \textsf{MAD} family. With a suitable coding, the coloring $C$ will
be defined for pairs $t=\left(  \mathcal{A}_{t},X_{t}\right)  $ where
$\mathcal{A}_{t}=\left\langle A_{\xi}\mid\xi<\alpha\right\rangle $ and
$X_{t}=\left\{  X_{n}^{t}\mid n\in\omega\right\}  \subseteq\left[
\omega\right]  ^{<\omega}$. We define $C\left(  t\right)  $ to be the constant
$0$ function in case $\mathcal{A}_{t}$ is not an almost disjoint family or
$X_{t}$ is not locally finite according to $\mathcal{I}\left(  \mathcal{A}%
_{t}\right)  .$ We will describe what to do in the other case. For every
$n\in\omega$ define $B_{n}=\bigcup\limits_{i<n}A_{\alpha_{i}}$ (hence
$B_{0}=\emptyset$) and let $d\left(  n\right)  $ be the smallest $k\geq n$
such that if $\ell\geq k$ then $B_{n}$ does not intersect every element of
$X_{\ell}^{t}.$ We define an increasing function $C\left(  t\right)
:\omega\longrightarrow\omega$ such that for every $n,i\in\omega,$
if\ $d\left(  n\right)  \leq i<d\left(  n+1\right)  $ then $C\left(  t\right)
\left(  n\right)  $ $\backslash$ $B_{n}$ contains an element of $X_{i}^t.$ The
rest of the proof is similar as in the case of $\Diamond\left(  \mathfrak{b}%
\right)  .$
\end{proof}

It is known that $\Diamond\left(  \mathfrak{b}\right)  $ holds in the Cohen
model (see \cite{ParametrizedDiamonds}) so there are Shelah-Stepr\={a}ns
\textsf{MAD} families in this model but as we saw earlier, there is no
strongly tight \textsf{MAD} family, so being Shelah-Stepr\={a}ns does not
imply being strongly tight. We will later see that strong tightness does not
imply being Shelah-Stepr\={a}ns.


\subsection{Existence under forcing axioms}
\label{p=c-existence}


We will now prove two results:
$\mathfrak{p=c}$, which is equivalent to MA($\sigma$-centered), implies the existence of a Shelah-Stepr\={a}ns \textsf{MAD}
family and of a strongly tight \textsf{MAD} family. 
In \cite{KatetovandKatetovBlassOrdersFsigmaIdeals} it was proved that Laflamme
\textsf{MAD} families exist under $\mathfrak{p=c}.$ The following is a strengthening:

\begin{proposition}   \label{ShelahStepransunderp=c}
If $\mathfrak{p=c}$ then every \textsf{AD} family of size less than
$\mathfrak{c}$ can be extended to a Shelah-Stepr\={a}ns \textsf{MAD} family.
\end{proposition}

\begin{proof}
Let $\mathcal{A}$ be an \textsf{AD} family of size less than $\mathfrak{c}$
and $X=\left\{  s_{n}\mid n\in\omega\right\}  \in\left(  \mathcal{I}\left(
\mathcal{A}\right)  ^{<\omega}\right)  ^{+}.$ We define the forcing
$\mathbb{P}$ as the set of all $p=\left(  t_{p},\mathcal{F}_{p}\right)  $
where $t_{p}\in2^{<\omega}$ and $\mathcal{F}_{p}\in\left[  \mathcal{A}\right]
^{<\omega}.$ If $p=\left(  t_{p},\mathcal{F}_{p}\right)  $ and $q=\left(
t_{q},\mathcal{F}_{q}\right)  $ then $p\leq q$ if the following holds:

\begin{enumerate}
\item $t_{q}\subseteq t_{p}$ and $\mathcal{F}_{q}\subseteq\mathcal{F}_{p}.$

\item In case $n\in dom\left(  t_{p}\right)  $ $\backslash dom\left(
t_{q}\right)  $ and $A\in\mathcal{F}_{q}$ if $t_{p}\left(  n\right)  =1$ then
$s_{n}\cap A=\emptyset.$
\end{enumerate}
For any $n\in\omega$ and $A\in\mathcal{A}$ let $D_{n,A}\subseteq\mathbb{P}$ be
the set of conditions $p=\left(  t_{p},\mathcal{F}_{p}\right)  $ such that
$t_{p}^{-1}\left(  1\right)  $ has size at least $n$ and $A\in\mathcal{F}%
_{p}.$ Since $X\in\left(  \mathcal{I}\left(  \mathcal{A}\right)  ^{<\omega
}\right)  ^{+}$ each $D_{n,A}$ is open dense. Clearly $\mathbb{P}$ is
$\sigma$-centered and since $\mathcal{A}$ has size less than $\mathfrak{p}$ we
can then force and find $Y\in\left[  X\right]  ^{\omega}$ such that $\bigcup
Y$ is almost disjoint with every element of $\mathcal{A}.$\qquad
\ \ \ \qquad\ \ \ \qquad\ \ \ 
\end{proof}

We prove a similar result for strongly tight \textsf{MAD} families:

\begin{lemma}
Let $\mathcal{A}$ be an \textsf{AD} family of size less than $\mathfrak{p}.$
Let $\left\{  X_{n}\mid n\in\omega\right\}  $ be a family of infinite
subsets of $\omega$ such that for every $A\in\mathcal{I}\left(  \mathcal{A}%
\right)  $ the set $\left\{  n\mid X_n\subseteq^{\ast}A\right\}  $ is finite.
Then there is $B\in\mathcal{A}^{\perp}$ such that $B\cap X_{n}\neq\emptyset$
for every $n\in\omega.$
\end{lemma}

\begin{proof}
We may assume that for every $n\in\omega$ there is $A_{n}\in\mathcal{A}$ such
that $X_{n}\subseteq A_{n}$ (note that if $A\in\mathcal{A}$ then the set
$\left\{  n\mid A_{n}=A\right\}  $ is finite). Let $\mathcal{B}=\left\{
A_{n}\mid n\in\omega\right\}  $ and $\mathcal{D}=\mathcal{A}\setminus
\mathcal{B}.$ We now define the forcing $\mathbb{P}$ whose elements are sets
of the form $p=\left(  s_{p},F_{p},G_{p}\right)  $ with the following properties:

\begin{enumerate}
\item $s_{p}\in\omega^{<\omega},$ $F_{p}\in\left[  \mathcal{D}\right]
^{<\omega}$ and $G_{p}\in\left[  \mathcal{B}\right]  ^{<\omega}.$

\item If $i\in dom\left(  s_{p}\right)  $ then $s_{p}\left(  i\right)  \in
X_{i}.$
\end{enumerate}
For $p,q\in\mathbb{P}$ we let $p\leq q$ if the following conditions hold:

\begin{enumerate}
\item $s_{q}\subseteq s_{p},$ $F_{q}\subseteq F_{p}$ and $G_{q}\subseteq
G_{p}.$

\item For every $i\in dom\left(  s_{p}\right)  \setminus dom\left(
s_{q}\right)  $ the following holds:

\begin{enumerate}
\item $s_{p}\left(  i\right)  \notin%
{\textstyle\bigcup} F_{q}.$

\item If $B\in G_{q}$ and $A_{i}\neq B$ then $s_{p}\left(  i\right)  \notin
B.$
\end{enumerate}
\end{enumerate}
It is easy to see that $\mathbb{P}$ is a $\sigma$-centered forcing and adds a
set almost disjoint with $\mathcal{A}$ that intersects every $X_{n}.$ Since
$\mathcal{A}$ has size less than $\mathfrak{p},$ the result follows.
\end{proof}

We conclude:

\begin{proposition}
If $\mathfrak{p=c}$ then every \textsf{AD} family of size less than $\cc$ can be
extended to a strongly tight \textsf{MAD} family.
\end{proposition}

We strongly conjecture the following has a positive answer:

\begin{problem}
Does $\pp = \cc$ imply the existence of a raving  \textsf{MAD} family?
\end{problem}


\subsection{Non-implications}
\label{non-implications}


Under \textsf{CH} we provide a number of examples for \textsf{MAD} families
satisfying some of the properties of Subsection~\ref{ideals-basic} while failing others.

\noindent We will now show that (consistently) strong tightness does not imply being
Laflamme or random indestructible. Recall that the summable ideal is defined as
$\mathcal{J}_{1/n}=\{A\subseteq\omega\mid{\textstyle\sum\limits_{n\in A}}
\frac{1}{n+1}<\omega\}.$  We start with the following lemma:

\begin{lemma}   \label{summablestrongtight}
Let $\mathcal{A}$ be a countable \textsf{AD} family contained in the summable
ideal. Let $X=\left\{  X_{n}\mid n\in\omega\right\}  \subseteq\left[
\omega\right]  ^{\omega}$ such that all $X_n$ are contained in some member of $\A$ and
$\{ n \mid X_n \sub A \}$ is finite for all $A \in \A$. 
Then there is $D\in\mathcal{A}^{\perp}\cap\mathcal{J}_{1/n}$ such
that $D\cap X_{n}\neq\emptyset$ for every $n\in\omega.$
\end{lemma}

\begin{proof}
Let $\mathcal{A}=\left\{  A_{n}\mid n\in\omega\right\}  $. For each $n\in
\omega$ we define $F_{n}=\left\{  X_{i}\mid X_{i}\subseteq A_{n}\right\}  .$
We construct a sequence of finite sets $\left\{  s_{n}\mid n\in\omega\right\}
\subseteq\left[  \omega\right]  ^{<\omega}$ such that:\ 

\begin{enumerate}
\item $\max\left(  s_{n}\right)  <\min\left(  s_{n+1}\right)  .$

\item ${\textstyle\sum\limits_{i\in s_{n}}} \frac{1}{1+i}<\frac{1}{2^{n+1}}.$

\item $s_{n}$ has non empty intersection with every element of $F_{n}.$

\item If $m<n$ then $s_{n}$ is disjoint from $A_{m}.$
\end{enumerate}
Assume we are at step $n,$ and let $r$ be such that $F_{n}=\left\{  X_{n_{1}%
},...,X_{n_{r}}\right\}  .$ Find $m$ such that $\frac{r}{1+m}<\frac{1}%
{2^{n+1}}$ and $s_{i}\subseteq m$ for every $i<n.$ For every $i\leq r$ we
choose $k_{i}>m$ such that $k_{i}\in X_{n_{i}}\setminus{\textstyle\bigcup\limits_{j<n}}
A_{j}$ and let $s_{n}=\left\{  k_{i}\mid i\leq r\right\}  .$ It is easy to see
that $D={\textstyle\bigcup\limits_{n\in\omega}}
s_{n}$ has the desired properties.
\end{proof}

In \cite{CardinalInvariantsofAnalyticPIdeals} it was proved that random
forcing destroys the summable ideal (in fact, $\J_{1/n} \leq_K tr(\N)$). We therefore conclude, using Lemma~\ref{stronglytightMAD}:

\begin{proposition}
[$\mathsf{CH}$]There is a strongly tight \textsf{MAD }family contained in the summable
ideal $\mathcal{J}_{1/n}$ (in particular, it is random destructible and not Laflamme).
\end{proposition}

Note that the $\PP_{\mathsf{MAD}} (\J_{1 / n} )$-generic \textsf{MAD} family has all these properties.


\begin{center}
$\star\star\star$
\end{center}


\noindent We saw that every \textsf{fin}$\times$\textsf{fin}-like \textsf{MAD} family is
Cohen indestructible (Proposition~\ref{SSindestructibility}). 
However, we will now show that \textsf{fin}$\times
$\textsf{fin}-like does not imply tightness (so in particular, it is a weaker
notion than Shelah-Stepr\={a}ns).

\begin{proposition}
[$\mathsf{CH}$]There is a \textsf{fin}$\times$\textsf{fin}-like \textsf{MAD}
family that is not tight.
\end{proposition}

\begin{proof}
Let $\left\{  \mathcal{I}_{\alpha}\mid\omega\leq\alpha<\omega_{1}\right\}  $
be an enumeration of all analytic ideals $\I$ with \textsf{fin}$\times$\textsf{fin} $\not\leq_K \I$ and let $\left\{  X_{n}\mid n\in
\omega\right\}  $ be a partition of $\omega$ into infinite sets. We will
recursively construct an \textsf{AD} family $\mathcal{A}=\left\{  A_{\alpha
}\mid\alpha<\omega_{1}\right\}  $ such that for every $\alpha$ the following
conditions hold:\qquad\ \ \ 

\begin{enumerate}
\item $\left\{  A_{n}\mid n\in\omega\right\}  $ is a partition of $\omega$
refining $\left\{  X_{n}\mid n\in\omega\right\}  $ and every $X_{n}$ contains
infinitely many of the $A_{m}.$

\item There is $\xi\leq\alpha$ such that $A_{\xi}\notin\mathcal{I}_{\alpha}.$

\item If $B\in\mathcal{I}\left(  \mathcal{A}\right)  $ then there is
$n\in\omega$ such that $B\cap X_{n}$ is finite.
\end{enumerate}
Note that by the comment after the proof of Lemma~\ref{Laflamme-notP-char}, $\A$ will indeed be \textsf{fin}$\times$\textsf{fin}-like
while the $X_n$ witness the failure of tightness.

Let $\mathcal{A}_{\alpha}=\left\{  A_{\xi}\mid\xi<\alpha\right\}  $ and assume
$\mathcal{A}_{\alpha}\subseteq\mathcal{I}_{\alpha}.$ Let $\alpha=\left\{
\alpha_{n}\mid n\in\omega\right\}  $ and define $L_{n}=A_{\alpha_{0}}%
\cup...\cup A_{\alpha_{n}}.$ Define $E_{n}=\left\{  m\mid\left\vert L_{n}\cap
X_{m}\right\vert <\omega\right\}  $ and note that $E=\left\langle
E_{n}\right\rangle _{n\in\omega}\,\ $is a decreasing sequence of infinite
sets. Find a pseudointersection $D$ of $E$ such that $\omega\setminus D$ also
contains a pseudointersection of $E.$ Define $T_{0}={\textstyle\bigcup\limits_{n\in D}}
X_{n}$ and $T_{1}={\textstyle\bigcup\limits_{n\notin D}}
X_{n}.$ Since \textsf{fin}$\times$\textsf{fin} $\mathcal{\nleq}_{K}$
$\mathcal{I}_{\alpha}$ we know that either \textsf{fin}$\times$\textsf{fin}
$\mathcal{\nleq}_{K}$ $\mathcal{I}_{\alpha}\upharpoonright T_{0}$ or
\textsf{fin}$\times$\textsf{fin} $\mathcal{\nleq}_{K}$ $\mathcal{I}_{\alpha
}\upharpoonright T_{1}.$ First assume \textsf{fin}$\times$\textsf{fin}
$\mathcal{\nleq}_{K}$ $\mathcal{I}_{\alpha}\upharpoonright T_{0}$. Then we
can choose $A_{\alpha}\in\left(  \mathcal{I}_{\alpha}\upharpoonright
T_{0}\right)  ^{+}$ that is almost disjoint with $\mathcal{A}_{\alpha
}\mathcal{\upharpoonright}T_{0}$ which implies it is \textsf{AD} with
$\mathcal{A}_\alpha.$ We now need to prove that for every $n<\omega$ there is
$X_{m}\ \ $such that$\ \left(  L_{n}\cup A_{\alpha}\right)  \cap X_{m}$ is
finite. Since $\omega\setminus D$ contains a pseudointersection of $E$
there is $m\in E_{n}\setminus D$ and then both $L_{n}$ and $A_{\alpha}$ are
almost disjoint with $X_{m}.$ The other case is similar.
\end{proof}


\begin{center}
$\star\star\star$
\end{center}

 
\noindent Recall that the density zero ideal is defined as $\mathcal{Z}%
=\{A\subseteq\omega\mid \lim\frac{\left\vert A\cap2^{n}\right\vert }{2^{n}%
}=0\}$. $\mathcal{Z}$ is
not Kat\v{e}tov below any $F_{\sigma}$-ideal. Thus, from Lemma~\ref{Katetovgenericlemma}, we obtain:

\begin{corollary}
Let $\mathcal{A}$ be a countable \textsf{AD }family contained in 
$\mathcal{Z}.$ If $\mathcal{I}$ is an $F_{\sigma}$-ideal and $f:\omega
\longrightarrow\omega$ then there is a countable \textsf{AD }family
$\mathcal{B}$ such that $\mathcal{A\subseteq B\subseteq Z}$ and there is
$B\in\mathcal{I}\left(  \mathcal{B}\right)  $ such that $f^{-1}\left(
B\right)  \in\mathcal{I}^{+}.$
\end{corollary}


Using a suitable bookkeeping device we conclude:

\begin{proposition}   
[$\mathsf{CH}$]   \label{PMADZ-generic-prop}
There is a Laflamme \textsf{MAD }contained in $\mathcal{Z}.$ In
particular, it is not a not-P {\textsf{MAD} }family. Additionally this \textsf{MAD} family
is strongly tight and random indestructible.
\end{proposition}

Indeed, the $\PP_{\mathsf{MAD}} (\Z)$-generic \textsf{MAD} family has all these properties. To see strong tightness,
use Lemma~\ref{summablestrongtight} and $\J_{1/n} \sub \Z$. Since $\Z \not\leq_K tr (\N)$, random indestructibility
follows from Corollary~\ref{Katetovgenericcorollary}. We shall come back to this generic object in the next section
(Theorem~\ref{PMADZ-generic}).




\section{Destructibility by forcing}

In \cite{ShelahTemplates} Shelah constructed models of $\mathfrak{d<a}$ (see
also \cite{BrendleTemplates}). In these models, $\mathfrak{d}$ is bigger than
$\omega_{1}.$ It is an old question of Roitman whether $\mathfrak{d}=\omega_{1}$
implies $\mathfrak{a}=\omega_{1}.$ Even the following question of the first
and last authors is still open:

\begin{problem} [Brendle, Raghavan] \label{bsa-problem}
Does $\mathfrak{b=s}=\omega_{1}$ imply $\mathfrak{a}%
=\omega_{1}$?
\end{problem}

Constructing models of $\mathfrak{b<a}$ is much easier than constructing
models of $\mathfrak{d<a}.$ However, all the known models of $\mathfrak{b=}$
$\omega_{1}<\mathfrak{a}$ require diagonalizing an ultrafilter, which
increases the splitting number (see \cite{ProperandImproper}, \cite{MobandMad}%
, \cite{MadFamiliesSplittingFamiliesandLargeContinuum} and
\cite{BrendleRaghavan}). Problem~\ref{bsa-problem} is related to
the following: \textit{Assuming }$\mathsf{CH},$ \textit{can every }\textsf{MAD
}\textit{family be destroyed by a proper forcing that does not add dominating
or unsplit reals? }Recall that Shelah-Stepr\={a}ns \textsf{MAD} families $\A$
are indestructible for many definable forcings that do not add dominating
reals. Perhaps surprisingly, such families can be destroyed by forcings that
do not add dominating reals or unsplit reals. In fact, we will see that the
Mathias forcing associated with $\mathcal{I}\left(  \mathcal{A}\right)  $ has
these properties.

\subsection{Destroying Hurewicz ideals}

Recall that for Hurewicz ideals $\I$, Mathias forcing $\MM (\I)$ preserves unbounded families
from the ground model. We now proceed to strengthen this.

We need the following notions (see~\cite[Definition 31]{BrendleRaghavan} for a notion similar to item 2):

\begin{definition}
\begin{enumerate}
\item Let $P=\left\{  s_{n}\mid n\in\omega\right\}  \subseteq\left[
\omega\right]  ^{<\omega}$ be a collection of finite disjoint sets and
$S\in\left[  \omega\right]  ^{\omega}.$ We say that $S$ \emph{block splits
}$P$ if both of the sets $\left\{  n\mid s_{n}\subseteq S\right\}  $ and
$\left\{  s_{n}\mid s_{n}\cap S=\emptyset\right\}  $ are infinite.

\item We say that $\mathcal{S}=\left\{  S_{\alpha}\mid\alpha\in\omega
_{1}\right\}  \subseteq\left[  \omega\right]  ^{\omega}$ is a \emph{tail
block}-\emph{splitting family }if for every infinite set $P$ of finite
disjoint subsets of $\omega$ there is $\alpha<\omega_{1}$ such that
$S_{\gamma}$ block splits $P$ for every $\gamma>\alpha.$
\end{enumerate}
\end{definition}

It is easy to see that tail block splitting families exist if $\mathfrak{d}%
=\omega_{1}$ (see also~\cite[Observation 34]{BrendleRaghavan}) and tail block splitting families are splitting families. We say
that a forcing $\mathbb{P}$ preserves a tail block-splitting family if it
remains tail block-splitting after forcing with $\mathbb{P}$.

\begin{proposition}  \label{Hurewicztailsplitting}
Let $\mathcal{I}$ be a Hurewicz ideal. If $\mathcal{S=}\left\{  S_{\alpha}%
\mid\alpha\in\omega_{1}\right\}  \subseteq\left[  \omega\right]  ^{\omega}$ is
a tail block-splitting family then $\mathbb{M}\left(  \mathcal{I}\right)  $
preserves $\mathcal{S}$ as a tail block-splitting family.
\end{proposition}

\begin{proof}
Let $\mathcal{I}$ be a Hurewicz ideal and $\mathcal{S}$ a tail block-splitting
family. Let $\dot{P}$ $\mathcal{=\{}\dot{p}_{n}\mid n\in\omega\}$ be a name
for an infinite set of pairwise disjoint finite subsets of $\omega$, we may
assume $\dot{p}_{n}$ is forced to be disjoint from $n$. For every $s\in\left[
\omega\right]  ^{<\omega}$ and $m\in\omega$ we define $X_{m}\left(  s\right)
$ as the set of all $t\in\left[  \omega\right]  ^{<\omega}$ such that
$\max\left(  s\right)  <\min\left(  t\right)  $ and there are $F_{\left(
t,m,s\right)  }\in\left[  \omega\right]  ^{<\omega}$ and $B\in\mathcal{I}$
such that $\left(  s\cup t,B\right)  \Vdash``\dot{p}_{m}=F_{\left(
t,m,s\right)  }\textquotedblright.$ It is easy to see that $\left\{
X_{m}\left(  s\right)  \mid m\in\omega\right\}  \subseteq\left(
\mathcal{I}^{<\omega}\right)  ^{+}$ and since $\mathcal{I}$ is Hurewicz, we
may find $Y_{m}\left(  s\right)  \in\left[  X_{m}\left(  s\right)  \right]
^{<\omega}$ such that if $W\in\left[  \omega\right]  ^{\omega}$ then $%
{\textstyle\bigcup\limits_{m\in W}}
Y_{m}\left(  s\right)  \in\left(  \mathcal{I}^{<\omega}\right)  ^{+}.$ Let
$Z_{m}\left(  s\right)  =
{\textstyle\bigcup\limits_{t\in Y_{m}\left(  s\right)  }}
F_{\left(  t,m,s\right)  }.$ For every $s\in\left[  \omega\right]  ^{<\omega}$
we can then find $D\left(  s\right)  \in\left[  \omega\right]  ^{\omega}$ such
that $R\left(  s\right)  =\left\{  Z_{m}\left(  s\right)  \mid m\in D\left(
s\right)  \right\}  $ is pairwise disjoint.

Since $\mathcal{S}$ is tail block-splitting, we can find $\alpha$ such that if
$\gamma>\alpha$ then $S_{\gamma}$ block splits $R\left(  s\right)  $ for every
$s\in\left[  \omega\right]  ^{<\omega}.$ We claim that in this case,
$S_{\gamma}$ is forced to block split $\dot{P}.$ If this was not the case, we
could find $\left(  s,A\right)  \in\mathbb{M}\left(  \mathcal{I}\right)  $ and
$n\in\omega$ such that either $\left(  s,A\right)  \Vdash``%
{\textstyle\bigcup}
\left\{  \dot{p}_{m}\mid\dot{p}_{m}\subseteq S_{\gamma}\right\}  \subseteq
n\textquotedblright$ or $\left(  s,A\right)  \Vdash``%
{\textstyle\bigcup}
\left\{  \dot{p}_{m}\mid\dot{p}_{m}\cap S_{\gamma}=\emptyset\right\}
\subseteq n\textquotedblright.$ Assume the first case holds (the other one is
similar). Since $S_{\gamma}$ block splits $R\left(  s\right)  ,$ we know that
the set $W=\left\{  m>n\mid Z_{m}\left(  s\right)  \subseteq S_{\gamma
}\right\}  $ is infinite. Since ${\textstyle\bigcup\limits_{m\in W}}
Y_{m}\left(  s\right)  \in\left(  \mathcal{I}^{<\omega}\right)  ^{+}$,
there is $m\in W$ and $t\in Y_{m}\left(  s\right)  $ such that $t\cap
A=\emptyset.$ We  know there is $B\in\mathcal{I}$ such that $\left(  s\cup
t,B\right)  \Vdash``\dot{p}_{m}=F_{\left(  t,m,s\right)  }\textquotedblright.$
Since $t\cap A=\emptyset$, we have $\left(  s\cup t,A\cup B\right)  \leq\left(
s,A\right)  .$ But $\left(  s\cup t,A\cup B\right)  $ forces that $\dot{p}%
_{m}$ is a subset of $S_{\gamma},$ which is a contradiction. We therefore conclude
that $\mathcal{S}$ remains being a tail block-splitting family.
\end{proof}

In particular, if $V$ is a model of \textsf{CH }and $\mathcal{I}$ is a
Hurewicz ideal, then $\mathbb{M}\left(  \mathcal{I}\right)  $ preserves
$V\cap\left[  \omega\right]  ^{\omega}$ as a splitting family (this result has
also been noted by Lyubomyr Zdomskyy). Since Hurewicz ideals are Canjar
ideals, we conclude the following:

\begin{corollary}
If $\mathcal{A}$ is Shelah-Stepr\={a}ns then $\mathcal{A}$ can be destroyed
with a ccc forcing that does not add dominating nor unsplit reals.
\end{corollary}

In fact, such forcings can be iterated without adding unsplit reals, as the
following result shows:

\begin{proposition}
Let $\langle\mathbb{P}_{\alpha},\mathbb{\dot{Q}}_{\alpha}\mid\alpha
<\delta\mathbb{\rangle}$ be a finite support iteration of ccc forcings. If
$\mathbb{P}_{\alpha}$ forces that $\mathbb{\dot{Q}}_{\alpha}$ preserves tail
block-splitting families, then $\mathbb{P}_{\delta}$ preserves tail
block-splitting families.
\end{proposition}

\begin{proof}
We prove the result by induction on $\delta.$ The cases where $\delta=0$,
$\delta$ is a successor ordinal or $\delta$ has uncountable cofinality are
trivial, so we assume $\delta$ is a limit ordinal of countable cofinality. Fix
an increasing sequence $\left\langle \delta_{n}\right\rangle _{n\in\omega}$
such that $\delta={\textstyle\bigcup}
\delta_{n}.$ Let $\mathcal{S}=\left\{  S_{\alpha}\mid\alpha\in\omega
_{1}\right\}  $ be a tail block-splitting family\ and let $\dot{P}=\left\{
\dot{s}_{i}\mid i\in\omega\right\}  $ be a $\mathbb{P}_{\delta}$-name for a
collection of finite disjoint sets. For every $n\in\omega,$ we define a
$\mathbb{P}_{\delta_{n}}$-name $\dot{P}\left(  n\right)  =\left\{  \dot{s}%
_{i}\left(  n\right)  \mid i\in\omega\right\}  $ as follows:

Assume $G_{n}\subseteq\mathbb{P}_{\delta_{n}}$ is a generic filter. In $V\left[
G_{n}\right]  $ we find a family $P\left(  n\right)  =\left\{  s_{i}\left(
n\right)  \mid i\in\omega\right\}  $ with the following properties:

\begin{enumerate}
\item $P\left(  n\right)  $ is a family of pairwise disjoint sets.

\item For every $i\in\omega,$ there is $p\in\mathbb{P}_{\delta}/G_{n}$ 
such that $p\Vdash_{\mathbb{P}_{\delta}/G_{n}}``\dot{s}_{i}%
=s_{i}\left(  n\right)  \textquotedblright$ (where $\mathbb{P}_{\delta}/G_{n}$
denotes the quotient forcing).
\end{enumerate}

Let $\dot{s}_{i}\left(  n\right)  $ be a $\mathbb{P}_{\delta_n}$-name for
$s_{i}\left(  n\right)  .$ Since $\mathbb{P}_{\delta_{n}}$ preserves tail
block-splitting families, there is  a $\mathbb{P}%
_{\delta_{n}}$-name $\dot{\alpha}_{n}$ for a countable ordinal such that $1_{\mathbb{P}%
_{\delta_{n}}}$ forces that if $\beta$ is bigger that $\dot{\alpha}_{n}$ then
$S_{\beta}$ tail block-splits $\dot{P}\left(  n\right)  .$ Since each
$\mathbb{P}_{\delta_{n}}$ has the countable chain condition, we can find
$\alpha<\omega_{1}$ such that $1_{\mathbb{P}_{\delta_n}}\Vdash``\dot{\alpha}%
_{n}<\alpha\textquotedblright$ for all $n$. We claim that if $\beta>\alpha$, then
$S_{\beta}$ is forced to block-split $\dot{P}.$

Assume this is not the case, so there are $m\in\omega,$ $\beta>\alpha,$
$p\in\mathbb{P}_{\delta}$ such that either $p\Vdash_{\mathbb{P}_{\delta}}``%
{\textstyle\bigcup}
\left\{  \dot{s}_{i}\mid\dot{s}_{i}\subseteq S_{\beta}\right\}  \subseteq
m\textquotedblright$ or $p\Vdash_{\mathbb{P}_{\delta}}``%
{\textstyle\bigcup}
\left\{  \dot{s}_{i}\mid\dot{s}_{i}\cap S_{\beta}=\emptyset\right\}  \subseteq
m\textquotedblright.$ We will assume that $p\Vdash_{\mathbb{P}_{\delta}}``%
{\textstyle\bigcup}
\left\{  \dot{s}_{i}\mid\dot{s}_{i}\subseteq S_{\beta}\right\}  \subseteq
m\textquotedblright$ (the other case is similar). Let $n\in\omega$ such that
$p\in\mathbb{P}_{\delta_{n}}.$ Since $S_{\beta}$ is forced to block-split
$\dot{P}\left(  n\right)  ,$ we can find $q\leq p$ and $j\in\omega$ such
that $q\Vdash_{\mathbb{P}_{\delta_n}}``\dot{s}_{j}\left(  n\right)  \nsubseteq
m\wedge\dot{s}_{j}\left(  n\right)  \subseteq S_{\beta}\textquotedblright.$ 
Then there is a $\PP_{\delta_n}$-name $\dot r\in\mathbb{P}_{\delta}$ such that
$q\Vdash_{\mathbb{P}_{\delta_{n}}}`` \dot r \in \PP_\delta / \dot G_n$ and $\dot r \forces_{\PP_\delta / \dot G_n} \dot{s}_{j}\left(  n\right)  =\dot{s}%
_{j}\textquotedblright$. 
Therefore, we can find
$r_{1}\leq q$ in $\PP_\delta$ such that $r_{1}\Vdash_{\mathbb{P}_{\delta}}``\dot{s}%
_{j}\nsubseteq m\wedge\dot{s}_{j}\subseteq S_{\beta}\textquotedblright,$ which
is a contradiction.
\end{proof}

By iterating the Mathias forcing of all Hurewicz ideals, we obtain:

\begin{theorem} \label{thm:59}
There is a model in which the following statements hold:
\begin{enumerate}
\item $\mathfrak{c}=\omega_{2}.$

\item $\mathfrak{b=s}=\omega_{1}.$

\item No \textsf{MAD} family of size ${\aleph}_{1}$ can be extended to a Hurewicz ideal.
\end{enumerate}
\end{theorem}

We do not know the value of $\mathfrak{a}$ in the previous model. Naturally,
if every \textsf{MAD} family could be extended to a Hurewicz ideal (at least
under \textsf{CH}), then we would be able to solve the Problem~\ref{bsa-problem}. 
Unfortunately, this may not be the case, as we will prove in the
next section (Theorem~\ref{notHurewicz}).

\subsection{Variants of the Shelah-Stepr\=ans property}

Given two non-empty finite subsets  $s,t$  of $\omega,$ we write $s<t$ if
$\max\left(  s\right)  <\min\left(  t\right)  .$ We 
say that $\mathcal{B}=\left\{  s_{n}\mid n\in\omega\right\}  \subseteq\left[
\omega\right]  ^{<\omega}\setminus\left\{  \emptyset\right\}  $ is a
\emph{block sequence }if $s_{n}<s_{n+1}$ for every $n\in\omega.$ The following
are natural weakenings of being Shelah-Stepr\={a}ns.

\begin{definition}   \label{SSblockdef}
Let $\mathcal{A}$ be a \textsf{MAD }family. 
\begin{enumerate}
\item We say that $\mathcal{A\ }$is
\emph{Shelah-Stepr\={a}ns for block sequences} if for every block sequence
$\mathcal{B}=\left\{  s_{n}\mid n\in\omega\right\}  \in \left(  \mathcal{I}\left(  \mathcal{A}\right)
^{<\omega}\right)  ^{+},$  there is $W\in\left[  \omega\right]  ^{\omega}$ such that $%
{\textstyle\bigcup\limits_{n\in W}} s_{n}\in\mathcal{I}\left(  \mathcal{A}\right)  .$
\item We say that $\mathcal{A}$ is
$\omega$\emph{-Shelah-Stepr\={a}ns for block sequences }if for every sequence
$\left\langle \B_{n}\right\rangle _{n\in\omega}$ of
block sequences with $\B_{n}\in\left(  \mathcal{I}\left(  \mathcal{A}\right)
^{<\omega}\right)  ^{+},$ there is $C\in \I (\A)$ such that for every
$n\in\omega$ there are infinitely many $s\in X_{n}$ such that $s\subseteq C.$
\end{enumerate}
\end{definition}

Note that by Lemma~\ref{SScharacterization}, 
every Shelah-Stepr\={a}ns \textsf{MAD} family is $\omega$-Shelah-Stepr\={a}ns for block sequences, 
and obviously  $\omega$-Shelah-Stepr\={a}ns for block sequences implies 
Shelah-Stepr\={a}ns for block sequences.

One may wonder if Shelah-Stepr\={a}ns and ($\omega$-)Shelah-Stepr\={a}ns for block
sequences are different concepts. We are going to prove that Shelah-Stepr\={a}ns and $\omega$-Shelah-Stepr\={a}ns for block
sequences might and might not agree. We do not know whether Shelah-Stepr\={a}ns for block
sequences and $\omega$-Shelah-Stepr\={a}ns for block sequences are the same. First we have the following result:

\begin{lemma}
Let $\mathcal{A}$ be a \textsf{MAD} family such that $\left\vert
\mathcal{A}\right\vert <$ \textsf{cov}$\left(  \mathcal{M}\right)  .$ If
$\mathcal{A}$ is Shelah-Stepr\={a}ns for block sequences, then $\mathcal{A}$
is Shelah-Stepr\={a}ns.
\end{lemma}

\begin{proof}
Let $\mathcal{A}$ be a  \textsf{MAD}
family of size less than \textsf{cov}$\left(  \mathcal{M}\right)$ that is Shelah-Stepr\={a}ns for block sequences. Letting
$X=\left\{  s_{n}\mid n\in\omega\right\}  \in\left(  \mathcal{I}\left(
\mathcal{A}\right)  ^{<\omega}\right)  ^{+},$ we must show that there is
$B\in\mathcal{I}\left(  \mathcal{A}\right)  $ such that $B$ contains
infinitely many elements of $X.$ We define the forcing notion $\mathbb{P}%
\left(  X\right)  $ as the set of all $p$ with the following properties:
\begin{enumerate}
\item There is $n_{p}\in\omega$ such that $p:n_{p}\longrightarrow X.$

\item If $i<j<n_{p}$ then $\max\left(  p\left(  i\right)  \right)  <\min\left(
p\left(  j\right)  \right)  .$
\end{enumerate}
We order $\mathbb{P}\left(  X\right)  $ by inclusion. Since $\mathbb{P}\left(
X\right)  $ is countable, it is forcing equivalent to Cohen forcing. Note that
$\mathbb{P}\left(  X\right)  $ adds a block sequence contained in $X.$
Furthermore, since $\left\vert \mathcal{A}\right\vert <$ \textsf{cov}$\left(
\mathcal{M}\right)  ,$ we can find $g:\omega\longrightarrow X$ such that
$g\left[  \omega\right]  $ is a block sequence and for every $B\in
\mathcal{I}\left(  \mathcal{A}\right)  $ there is $n\in\omega$ such that
$g\left(  n\right)  \cap B=\emptyset.$ Since $\mathcal{A}$ is
Shelah-Stepr\={a}ns for block sequences, we conclude that there is
$B\in\mathcal{I}\left(  \mathcal{A}\right)  $ such that $B$ contains
infinitely many elements of $X.$
\end{proof}

It is easy to see that \textsf{MAD}   
families that are $\omega$-Shelah-Stepr\={a}ns for block sequences are tight, so in particular, they are Cohen-indestructible. We thus
obtain:

\begin{corollary}
In the Cohen model, every  \textsf{MAD}
family that is $\omega$-Shelah-Stepr\={a}ns for block sequences is Shelah-Stepr\={a}ns.
\end{corollary}

We will see now that the conclusion of the corollary is false
under the Continuum Hypothesis (Theorem~\ref{PMADZ-generic}).

Given $n\in\omega,$ let $R_{n}=[2^{n},2^{n+1})$ and for every $A\subseteq
\omega$ let $\varphi_{n}\left(  A\right)  =\frac{\left\vert A\cap
R_{n}\right\vert }{2^{n}}$. We also define the function $\varphi_{\max}:\left[
\omega\right]  ^{<\omega}\mathfrak{\longrightarrow}$ $\mathbb{Q}$ where
$\varphi_{\max}\left(  A\right)  =\max\left\{  \varphi_{n}\left(  A\right)  \mid
n\in\omega\right\}  .$

\begin{lemma}
Let $X=\left\{  s_{n}\mid n\in\omega\right\}  $ be a block sequence and
$\mathcal{A}$ a countable \textsf{AD }family such that $\mathcal{A}%
\subseteq\mathcal{Z}.$ If there is $m>0$ such that $\frac{1}{m}\leq
\varphi_{\max}\left(  s_{n}\right)  $ for every $n\in\omega,$ then there is
$B\in\mathcal{A}^{\perp}\cap\mathcal{Z}$ such that $B\cap s_{n}\neq\emptyset$
for every $n\in\omega.$
\end{lemma}

\begin{proof}
For every $n\in\omega,$ we choose $l_{n}\in\omega$ such that $\frac{1}{m}%
\leq\varphi_{l_n}\left(  s_{n}\right)  .$ Since $X$ is pairwise disjoint, for
every $l\in\omega$ the set $\left\{  n\mid l_{n}=l\right\}  $ has size at most
$m.$ Let $\mathcal{A}=\left\{  A_{n}\mid n\in\omega\right\}  $ and define
$B_{n}=A_{0}\cup...\cup A_{n}.$ Fix an increasing function $f:\omega
\longrightarrow\omega$ such that for every $n,i\in\omega,$ if $f\left(
n\right)  \leq i$ then $\varphi_{i}\left(  B_{n}\right)  <\frac{1}{m}.$ We can
then find $B=\left\{  y_{n}\mid n\in\omega\right\}  $ such that for every
$n\in\omega,$ the following holds:

\begin{enumerate}
\item $y_{n}\in s_{n}\cap R_{l_{n}}.$

\item If $f\left(  k\right)  \leq l_n$ then $y_{n}\notin
B_{k}.$
\end{enumerate}
It is easy to see that $B\in\mathcal{A}^{\perp}$ and $B\cap s_{n}\neq
\emptyset$ for every $n\in\omega.$ Finally, $B\in\mathcal{Z}$ since
$\left\vert B\cap R_{l}\right\vert \leq m$ for every $l\in\omega.$
\end{proof}

We conclude:

\begin{lemma}
Let $\overline{X}=\left\{  X_{n}\mid n\in\omega\right\}  $ be a countable
collection of block sequences and $\mathcal{A}\in\mathbb{P}_{\mathsf{MAD}%
}\left(  \mathcal{Z}\right)  .$ If $\mathcal{A}\Vdash``\overline{X}%
\subseteq(\mathcal{I(\dot{A}}_{gen}(\mathcal{Z)})^{<\omega})^{+}%
\textquotedblright$ then there is $\mathcal{B\in}$ $\mathbb{P}_{\mathsf{MAD}%
}\left(  \mathcal{Z}\right)  $ such that $\mathcal{A\subseteq B}$ and there is
$B\in\mathcal{I}\left(  \mathcal{B}\right)  $ such that $B$ contains
infinitely many elements of each $X_{n}.$\ \ 
\end{lemma}

\begin{proof}
Let $\mathcal{A}=\left\{  A_{n}\mid n\in\omega\right\}  $ and define
$B_{n}=A_{0}\cup...\cup A_{n}.$ Given $n,k\in\omega,$ we define
$X_{n}\left(  k\right)  =\left\{  s\in X_{n}\mid s\cap B_{k}=\emptyset
\right\} $  and note that each $X_{n}\left(  k\right)  $ is infinite. We claim
that for every $n,m,k\in\omega,$ $m>0$, there are infinitely many $s\in
X_{n}\left(  k\right)  $ such that $\varphi_{\max}\left(  s\right)  <\frac
{1}{m}.$ Assuming this is not the case,  there are $n,m,k\in\omega,$ $m>0$
such that $\frac{1}{m}\leq\varphi_{\max}\left(  s\right)  $ for almost all
$s\in X_{n}\left(  k\right)  .$ Let $Y=\left\{  s\in X_{n}\left(  k\right)
\mid\varphi_{\max}\left(  s\right)  <\frac{1}{m}\right\}  $ and $Z=X_{n}\left(
k\right)  \setminus Y.$ By the previous lemma, there is $B\in\mathcal{A}%
^{\perp}\cap\mathcal{Z}$ such that $B$ intersects every element of $Z.$ It
follows that $B\cup B_{k}\cup{\textstyle\bigcup}
Y$ intersects every element of $X_n.$ Therefore, $\mathcal{A\cup}\left\{
B\right\}  $ is an extension of $\mathcal{A}$ forcing that $X_{n}$ is not
positive, which is a contradiction.

Thus we know that for every $n,m,k\in\omega,$ $m>0$ there are
infinitely many $s\in X_{n}\left(  k\right)  $ such that $\varphi_{\max}\left(
s\right)  <\frac{1}{m}.$ By an easy diagonalization argument, we can find
$B\in\mathcal{A}^{\perp}\cap\mathcal{Z}$ such that $B$ contains infinitely
many elements of each $X_{n}.$
\end{proof}

Let $Z\left(  n,m\right)  =\left\{
s\subseteq R_{m}\mid\varphi_{m}\left(  R_{m}\setminus s\right)  <\frac
{1}{2^{n+1}}\right\}  $ and define $X_{n}={\textstyle\bigcup\limits_{m\in\omega}}
Z\left(  n,m\right)  .$ It is easy to see that $X_{n}\in\left(  \mathcal{I}%
\left(  \mathcal{Z}\right)  ^{<\omega}\right)  ^{+}.$

\begin{lemma}
Let $\mathcal{A}$ $\mathcal{\in}$ $\mathbb{P}_{\mathsf{MAD}}\left(
\mathcal{Z}\right)  $ and for every $n\in\omega,$ let $Y_{n}\in\left[
X_{n}\right]  ^{<\omega}.$ There is $B\in\mathcal{A}^{\perp}\cap\mathcal{Z}$
such that $B\cap s\neq\emptyset$ for every $s\in Y_{n}$ and $n\in\omega.$
\end{lemma}

\begin{proof}
Let $\mathcal{A}=\left\{  A_{n}\mid n\in\omega\right\}  $ and $B_{n}=A_{0}%
\cup...\cup A_{n}.$ We first find an increasing function $f:\omega
\longrightarrow\omega$ such that for every $n\in\omega,$ the following
conditions hold:
\begin{enumerate}
\item $f\left(  n\right)  $ is of the form $2^{m_{n}+1}$ for some $m_{n}.$

\item If $f\left(  n\right)  <i$ then $\varphi_{i}\left(  B_{n}\right)
<\frac{1}{2^{n+2}}.$

\item If $s\in Y_{n}$ and $s\subseteq R_{j},$ then $j<f\left(  n\right)  .$
\end{enumerate}
We now define a sequence $\left\langle t_{n}\right\rangle _{n\in\omega}$ such
that for every $n\in\omega,$ the following conditions hold:
\begin{enumerate}
\item $t_{n}\subseteq R_{n}.$

\item If $n<f\left(  0\right)  $ then $\varphi_{n}\left(  t_{n}\right)
=\frac{1}{2}.$

\item If $f\left(  k\right)  \leq n<f\left(  k+1\right)  $ then $t_{n}\cap
B_{k}=\emptyset.$

\item If $f\left(  k\right)  \leq n<f\left(  k+1\right)  $ then $\varphi
_{n}\left(  t_{n}\right)  =\frac{1}{2^{k+2}}.$
\end{enumerate}
Letting $B={\textstyle\bigcup\limits_{n\in\omega}}
t_{n},$ it is easy to see that $B$ has the desired properties.
\end{proof}

We thus obtain:

\begin{theorem}   \label{PMADZ-generic}
\begin{enumerate}
\item $\mathbb{P}_{\mathsf{MAD}}\left(  \mathcal{Z}\right)  $ forces that
$\mathcal{A}_{gen}\left(  \mathcal{Z}\right)  $ is an $\omega$%
-Shelah-Stepr\={a}ns \textsf{MAD} family for block sequences  that is not Canjar.

\item The Continuum Hypothesis implies that there is  a non Canjar,  $\omega$%
-Shelah-Stepr\={a}ns \textsf{MAD} family for block sequences $\mathcal{A}$
such that $\mathcal{A}\subseteq\mathcal{Z}.$ In particular, $\mathcal{A}$ is
not Shelah-Stepr\={a}ns.\qquad\ \ 
\end{enumerate}
\end{theorem}

To see that $\A$ is not Shelah-Stepr\=ans, either use $\A \subseteq \Z$ 
or the fact that $\A$ is not Canjar, and recall 
that every Shelah-Stepr\={a}ns \textsf{MAD} family is Hurewicz (Proposition~\ref{ShelahSteprans-Hurewicz}) and thus Canjar. 
For other properties of this generic \textsf{MAD} family see the earlier Proposition~\ref{PMADZ-generic-prop}.


\begin{center}
$\star\star\star$
\end{center}



\noindent We say that a block sequence $\mathcal{B}=\left\{  s_{n}\mid n\in
\omega\right\}  $ \emph{witnesses that }$\mathcal{A}$ \emph{is not
Shelah-Stepr\={a}ns for block sequences }if $\B \in \left( \I (\A)^{< \omega} \right)^+$ and there is no
$W \in \omoms$ such that $\bigcup_{n \in W} s_n \in \I (\A)$. We will say that $\mathcal{B}=\left\{  s_{n}\mid n\in
\omega\right\}  $ is an \emph{increasing block sequence }if for every
$n\in\omega,$ the set $\left\{  m\mid\left\vert s_{m}\right\vert =n\right\}  $
is finite.

\begin{lemma}
Let $\mathcal{A}$ be a \textsf{MAD }family. If a block sequence $\mathcal{B}%
=\left\{  s_{n}\mid n\in\omega\right\}  $ witnesses that $\mathcal{A}$ is not
Shelah-Stepr\={a}ns for block sequences, then $\mathcal{B}$ is an increasing
block sequence.
\end{lemma}

\begin{proof}
Assume this is not the case. So there is $m\in\omega$ such that the set
$W=\left\{  n\mid\left\vert s_{n}\right\vert =m\right\}  $ is infinite. By
applying that $\mathcal{A}$ is maximal $m$-many times, we can find $W_{0}%
\in\left[  W\right]  ^{\omega}$ and $B\in\mathcal{I}\left(  \mathcal{A}%
\right)  $ such that ${\textstyle\bigcup\limits_{n\in W_{0}}}
s_{n}\subseteq B,$ which is a contradiction.
\end{proof}

We need the following notion:

\begin{definition}
Let $\mathcal{B}=\left\{  s_{n}\mid n\in\omega\right\}  $ be an increasing
block sequence. We define the ideal $\mathcal{J}\left(  \mathcal{B}\right)  $
as the set of all $A\subseteq\omega$ such that $\lim_n \left(
\frac{\left\vert A\cap s_{n}\right\vert }{\left\vert s_{n}\right\vert
}\right)  =0.$
\end{definition}

Note that the density zero ideal has the previous form. Given an increasing block sequence 
$\mathcal{B} =\left\{  s_{n}\mid n\in\omega\right\}  $ and
$X\in\left[  \omega\right]  ^{\omega},$ we define $\mathcal{B}_{X}=\left\{
s_{n}\mid n\in X\right\}  .$ Note that if $\A$ is \textsf{MAD} then $\I (\A)$ is meager
and therefore by Talagrand's Theorem, there is an increasing interval  partition $\mathcal{B} =\left\{  s_{n}\mid n\in\omega\right\}  $
such that there is no $W \in \omoms$ with $\bigcup_{n \in W} s_n \in \I (\A)$. We can now prove the following result:

\begin{proposition}   \label{countableexception}
Let $\mathcal{A}$ be a \textsf{MAD }family and let $\mathcal{B} =\left\{  s_{n}\mid n\in\omega\right\}  $
be such that there is no $W \in \omoms$ with $\bigcup_{n \in W} s_n \in \I (\A)$.
There are $X\in\left[\omega\right]  ^{\omega}$ and $\mathcal{A}_{0}\in\left[  \mathcal{A}\right]
^{\leq\omega}$ such that $\mathcal{A}\setminus\mathcal{A}_{0}\subseteq
\mathcal{J}\left(  \mathcal{B}_{X}\right)  .$
\end{proposition}

\begin{proof}
We argue by contradiction, so assume this is not the case.

By $\mathbb{Q}^{+}$ we denote the set of all positive rational numbers. We
will now recursively define $\left\langle A_{\alpha},q_{\alpha},X_{\alpha
}\right\rangle _{\alpha\in\omega_{1}}$ such that for every $\alpha<\omega_{1}$
the following hold:

\begin{enumerate}
\item $A_{\alpha}\in\mathcal{A},$ $q_{\alpha}\in\mathbb{Q}^{+}$ and
$X_{\alpha}\in\left[  \omega\right]  ^{\omega}.$

\item If $\alpha\neq\beta$ then $A_{\alpha}\neq A_{\beta}.$

\item If~$\beta<\alpha$ then $X_{\alpha}\subseteq^{\ast}X_{\beta}.$

\item If $n\in X_{\alpha},$ then $q_{\alpha}\leq\frac{\left\vert A_{\alpha
}\cap s_{n}\right\vert }{\left\vert s_{n}\right\vert }.$
\end{enumerate}

Let $\alpha<\omega_{1}$ and assume we have already constructed $\left\langle
A_{\xi},q_{\xi},X_{\xi}\right\rangle _{\xi<\alpha}.$ We will see how to find
$A_{\alpha},q_{\alpha}$ and $X_{\alpha}.$ Since $\left\{  X_{\xi}\mid
\xi<\alpha\right\}  $ is a $\subseteq^{\ast}$-decreasing sequence, we may find
$Y\in\left[  \omega\right]  ^{\omega}$ such that $Y\subseteq^{\ast}X_{\xi}$
for every $\xi<\alpha.$ By our assumption, the set $\mathcal{C}%
=\mathcal{A\setminus J}\left(  \mathcal{B}_{Y}\right)  $ is uncountable. Note
that if $A\in\mathcal{C},$ then there is $q_{A}\in\mathbb{Q}^{+}$ such that
the set $\left\{  n\in Y\mid q_A \leq\frac{\left\vert A\cap s_{n}\right\vert
}{\left\vert s_{n}\right\vert }\right\}  $ is infinite. Since $\mathcal{C}$ is
uncountable, we may find $q_{\alpha}\in\mathbb{Q}^{+}$ such that
$\mathcal{C}_{1}=\left\{  A\in\mathcal{C\mid}q_{A}=q_\alpha\right\}  $ is
uncountable. We can then find $A_{\alpha}\in\mathcal{C}_{1}$ such that
$A_{\alpha}\neq A_{\xi}$ for every $\xi<\alpha.$ Finally, let $X_{\alpha
}=\left\{  n\in Y\mid q_{\alpha}\leq\frac{\left\vert A_{\alpha}\cap
s_{n}\right\vert }{\left\vert s_{n}\right\vert }\right\}  .$ Clearly,
$A_{\alpha},q_{\alpha}$ and $X_{\alpha}$ have the desired properties.

We can now find $W\in\left[  \omega_{1}\right]  ^{\omega_{1}}$ and
$q\in\mathbb{Q}^{+}$ such that $q_{\alpha}=q$ for every $\alpha\in W.$ Let
$m\in\omega$ such that $\frac{1}{q}<m$ and choose $\alpha_{1},...,\alpha
_{m}\in W.$ Let $X=X_{\alpha_{1}}\cap...\cap X_{\alpha_m}$ and note that $X$
is an infinite set. By construction, if $n\in X$ and $i\leq m$ then
$q\leq\frac{\left\vert A_{\alpha_{i}}\cap s_{n}\right\vert }{\left\vert
s_{n}\right\vert }.$ Since $\frac{1}{q}<m,$ for each $n\in X$ there must be
$i_{n},j_{n}\leq m$ such that $i_{n}\neq j_{n}$ and $A_{\alpha_{i_{n}}}\cap
A_{\alpha_{j_{n}}}\cap s_{n}\neq\emptyset.$ Since $X$ is infinite, there are
$i,j\leq m$ and $Y\in\left[  X\right]  ^{\omega}$ such that $i=i_{n}$ and
$j=j_{n}$ for every $n\in Y.$ This implies that $A_{\alpha_{i_{n}}}\cap
A_{\alpha_{j_{n}}}$ is infinite, which is a contradiction.
\end{proof}

We also have the following:

\begin{lemma}
If $\mathcal{B}$ is an increasing block sequence, then $\mathcal{J}\left(
\mathcal{B}\right)  \leq_{K}\mathcal{Z}.$
\end{lemma}

\begin{proof}
It is easy to see that $\mathcal{J}\left(  \mathcal{B}\right)  $ is a
non-pathological $P$-ideal (see \cite{IliasBook} or \cite{TesisDavid} for the
definition of non-pathological $P$-ideals) and in \cite{TesisDavid} it was
proved that every non-pathological $P$-ideal is Kat\v{e}tov below
$\mathcal{Z}.$
\end{proof}

Therefore, every \textsf{MAD} family is \textquotedblleft nearly
Kat\v{e}tov-below\textquotedblright\ $\mathcal{Z}$. By these results, it would
be tempting to conjecture the following: \textit{If }$\mathcal{Z}$\textit{ can
be destroyed without increasing }$\bb$\textit{ and }$\sss,$\textit{ then every
\textsf{MAD} family can be
destroyed without increasing }$\mathfrak{b}$\textit{ or }$\mathfrak{s}.$
Unfortunately, it seems that the density zero ideal can not be destroyed
without increasing $\mathfrak{b}$ or $\mathfrak{s}.$ Recently, Raghavan showed
that \textsf{cov}$^{\ast}\left(  \mathcal{Z}\right)  \leq \max\left\{
\mathfrak{b},\mathfrak{s}\left(  \mathfrak{pr}\right)  \right\}  $\footnote{If
$\mathcal{I}$ is a tall ideal, by \textsf{cov}$^{\ast}\left(  \mathcal{I}%
\right)  $ we denote the smallest size of a family $\mathcal{X\subseteq I}$
such that for every $A\in\left[  \omega\right]  ^{\omega}$ there is
$X\in\mathcal{X}$ such that $A\cap X\in\left[  \omega\right]  ^{\omega}.$}
($\mathfrak{s}\left(  \mathfrak{pr}\right)  $ is the \emph{promptly splitting
number}, which is a cardinal invariant closely related to $\mathfrak{s},$ see
\cite{MoreontheDensityZeroIdeal} for more details). This improves an earlier
work of Raghavan and Shelah (see
\cite{TwoinequalitiesbetweenCardinalInvariants}) where they showed that
\textsf{cov}$^{\ast}\left(  \mathcal{Z}\right)  \leq\mathfrak{d}.$

We know that every \textsf{MAD} family is contained up to a countable subfamily in an ideal
$\mathcal{J}\left(  \mathcal{B}\right)  $ (where $\mathcal{B}$ is an
increasing block sequence). We will now show that (consistently) this is best
possible, that is, one can not disregard the countable family in Proposition~\ref{countableexception}
(see Theorem~\ref{ZMADfamily}).

Let $\mathcal{B}=\left\{  P_{n}\mid n\in\omega\right\}  $ be the interval
partition of $\omega$ with $\left\vert P_{n}\right\vert =n+1.$ Given
$X\subseteq\omega$ and $n\in\omega,$ we define $\mathcal{B}\left(  X,n\right)
=\left\{  m\mid\left\vert P_{m}\setminus X\right\vert \leq n\right\}  .$ We
will say a family $\mathcal{A}$ is $\mathcal{B}$-\textsf{AD} if the following
conditions hold:

\begin{enumerate}
\item $\mathcal{A}$ is a countable \textsf{AD} family.

\item If $B\in\mathcal{I}\left(  \mathcal{A}\right)  $ then $\mathcal{B}%
\left(  B,n\right)  $ is finite for every $n\in\omega.$

\item If $B\in\mathcal{I}\left(  \mathcal{A}\right)  $ then there is
$n\in\omega$ such that $P_{n}\cap B=\emptyset.$
\end{enumerate}
Note that if $\mathcal{A}$ is $\mathcal{B}$-\textsf{AD} then for every
$B\in\mathcal{I}\left(  \mathcal{A}\right)  $  there are, in fact,
infinitely many $n\in\omega$ such that $P_{n}\cap B=\emptyset$ (recall every
finite set is in $\mathcal{I}\left(  \mathcal{A}\right)  $).

\begin{lemma}
Let $\mathcal{A}$ be a $\mathcal{B}$-\textsf{AD} and $X\in\mathcal{A}^{\perp
}.$ There is $A\in\left[  X\right]  ^{\omega}$ such that $\mathcal{A\cup
}\left\{  A\right\}  $ is $\mathcal{B}$-\textsf{AD}.
\end{lemma}

\begin{proof}
Letting $\mathcal{A}=\left\{  A_{n}\mid n\in\omega\right\}  ,$ we may assume that
$n\in A_{n}$ for every $n\in\omega.$ For every $n\in\omega,$ we define $B_{n}=%
{\textstyle\bigcup} \left\{  A_{i}\mid i\leq \max\left(  P_{n}\right)  \right\}  $ (note that
$P_{0}\cup...\cup P_{n}\subseteq B_{n}$ and $B_{n}\in\mathcal{I}\left(
\mathcal{A}\right)  $). We recursively construct a sequence $\left\langle
\left(  y_{n},u_{n},x_{n}\right)  \right\rangle _{n\in\omega}$ with the
following properties:

\begin{enumerate}
\item $y_{n}<u_{n}<y_{n+1}$ for every $n\in\omega.$

\item $\left\{  x_{n}\mid n\in\omega\right\}  \subseteq X.$

\item $x_{n}\in P_{y_{n}}.$

\item $x_{n}\notin B_{n}.$

\item $y_{n}\notin\mathcal{B}\left(  B_{n},n+1\right)  .$

\item $B_{n}\cap P_{u_{n}}=\emptyset.$
\end{enumerate}

Assuming we have constructed $\left(  y_{i},u_{i},x_{i}\right)  $ for every
$i<n,$ we will see how to define $\left(  y_{n},u_{n},x_{n}\right)  .$ We
first find $r\in\omega$ such that the following hold:

\begin{enumerate}
\item $\max\left(  P_{u_{i}}\right)  <r$ for every $i < n$.

\item $B_{n}\cap X\subseteq r.$

\item $\mathcal{B}\left(  B_{n},n+1\right)  \subseteq r.$
\end{enumerate}
Since $X$ is an infinite set, we can find $y_{n}$ such that $r<\min\left(
P_{y_{n}}\right)  $ and $X\cap P_{y_{n}}\neq\emptyset.$ Choose any $x_{n}\in
X\cap P_{y_{n}}$. Finally, let $u_{n}$ such that $y_{n}<u_{n}$ and $B_{n}\cap
P_{u_{n}}=\emptyset.$

We now define $A=\left\{  x_{n}\mid n\in\omega\right\}  .$ Clearly $A$ is
almost disjoint with every element of $\mathcal{A}$ and $A$ is an infinite
subset of $X.$ Letting $\mathcal{A}_{1}=\mathcal{A\cup}\left\{  A\right\}  ,$ we
need to argue that $\mathcal{A}_{1}$ is a $\mathcal{B}$-\textsf{AD} family.
Letting $n\in\omega,$ note that if $m>n$ then $\left(  A\cup B_{n}\right)  \cap
P_{u_{m}}=\emptyset,$ so $A\cup B_{n}$ is disjoint with infinitely many
elements of $\mathcal{B}.$ Finally, note that if $m>n$ then $y_{m}%
\notin\mathcal{B}\left(  A\cup B_{n},n\right)  .$
\end{proof}

We can now prove:

\begin{lemma}
Let $\mathcal{A}$ be $\mathcal{B}$-\textsf{AD}. If $f:\omega\longrightarrow
\omega$ is a Kat\v{e}tov-morphism from $\mathcal{Z}$ to $\mathcal{I}\left(
\mathcal{A}\right)  $ then there is $A\in\mathcal{A}^{\perp}$ such that
$\mathcal{A}_{1}=\mathcal{A\cup}\left\{  A\right\}  $ is $\mathcal{B}%
$-\textsf{AD} and $f$ is no longer a Kat\v{e}tov-morphism from $\mathcal{Z}$
to $\mathcal{I}\left(  \mathcal{A}_{1}\right)  .$
\end{lemma}

\begin{proof}
As before, let $\mathcal{A}=\left\{  A_{n}\mid n\in\omega\right\}  $ (we
assume again that $n\in A_{n}$ for every $n\in\omega$) and for every $n\in\omega,$
we define $B_{n}= {\textstyle\bigcup}
\left\{  A_{i}\mid i\leq \max\left(  P_{n}\right)  \right\}  .$ We recursively
construct a sequence $\left\langle \left(  y_{n},u_{n},s_{n}\right)
\right\rangle _{n\in\omega}$ with the following properties:

\begin{enumerate}
\item $y_{n}<y_{n+1}$ for every $n\in\omega.$

\item $s_{n}\subseteq f\left[  R_{y_{n}}\right]  $ (recall that $R_{n}=[2^{n},2^{n+1})$ and for every $A\subseteq
\omega$ we defined $\varphi_{n}\left(  A\right)  =\frac{\left\vert A\cap
R_{n}\right\vert }{2^{n}}$).

\item $\varphi_{y_{n}}\left(  f^{-1}\left(  s_{n}\right)  \right)  \geq
\frac{1}{3}.$

\item $s_{n}\cap B_{n}=\emptyset.$

\item If $P_{m}\cap s_{n}\neq\emptyset$ then $m\notin\mathcal{B}\left(
B_{n},3n\right)  $ for every $m\in\omega.$

\item If $P_{m}\cap s_{n}\neq\emptyset$ then $n\leq\left\vert P_{m}%
\setminus\left(  s_{n}\cup B_{n}\right)  \right\vert .$

\item $P_{u_{n}}\cap B_{n}=\emptyset.$

\item $\max\left(  P_{u_{n}}\right)  <\min\left(  s_{n+1}\right)  \leq
\max\left(  s_{n+1}\right)  <\min\left(  P_{u_{n+1}}\right)  .$
\end{enumerate}

Assuming we have constructed the triple $\left(  y_{n},u_{n},s_{n}\right)  ,$ we will show
how to construct $\left(  y_{n+1},u_{n+1},s_{n+1}\right)  .$ Let
$D=B_{n+1}\cup{\textstyle\bigcup}
\left\{  P_{m}\mid m\in\mathcal{B}\left(  B_{n+1},3(n+1) \right)  \right\}  \cup
\max\left(  P_{u_{n}}\right)  .$ Clearly $D\in\mathcal{I}\left(  \mathcal{A}%
\right)  $, hence $f^{-1}\left(  D\right)  \in\mathcal{Z}.$ Let $y_{n+1}%
>y_{n}$ with the property that $\varphi_{y_{n+1}}\left(  f^{-1}\left(  D\right)  \right)
<\frac{1}{3}.$ Let $z=f\left[  R_{y_{n+1}}\setminus f^{-1}\left(  D\right)
\right]  $ and note that if $m\in\omega$ and $P_{m}\cap z\neq\emptyset$ then
$m\notin\mathcal{B}\left(  B_{n+1},3(n+1)\right)  $ and $\max\left(  P_{u_{n}%
}\right)  <\min\left(  P_{m}\right)  .$ Let $K=\left\{  m\mid P_{m}\cap
z\neq\emptyset\right\}  $ which clearly is a finite set. For every $m\in K$
let $t_{m}=P_{m}\setminus B_{n+1}$ and define $K_{1}=\left\{  m\in K\mid
\left\vert t_{m}\setminus z\right\vert <n+1\right\}  .$ Note if $m\in K_{1},$
then $\left\vert t_{m}\right\vert >3(n+1)$ hence $2(n+1)\leq\left\vert t_{m}\cap
z\right\vert .$ For $m \in K_1$ we can now choose distinct $x_{0}^{m},...,x_{n}^{m},w_{0}^{m}%
,...,w_{n}^{m}\in t_{m}\cap z$ such that $\varphi_{y_{n+1}}\left(
f^{-1}\left(  \left\{  x_{0}^{m},...,x_{n}^{m}\right\}  \right)  \right)
\geq\varphi_{y_{n+1}}\left(  f^{-1}\left(  \left\{  w_{0}^{m},...,w_{n}%
^{m}\right\}  \right)  \right)  .$ Given $m\in K,$ we now define the set $s_{n}%
^{m}=\left(  t_{m}\cap z\right)  \setminus\left\{  w_{0}^{m},...,w_{n}%
^{m}\right\}  $ if $m\in K_{1}$ and $s_{n}^{m}=t_{m}\cap z$ if $m\in
K\setminus K_{1}.$ Let $\overline{s}_{n}^{m}=$ $\left\{  w_{0}^{m}%
,...,w_{n}^{m}\right\}  $ for every $m\in K_{1}$ and we define $s_{n+1}=%
{\textstyle\bigcup\limits_{m\in K}} s_{n}^{m}$ and $\overline{s}_{n+1}=%
{\textstyle\bigcup\limits_{m\in K_{1}}}
\overline{s}_{n}^{m}.$ Note that $R_{y_{n+1}}=R_{y_{n+1}}\cap\left(
f^{-1}\left(  D\right)  \cup f^{-1}\left(  s_{n+1}\right)  \cup f^{-1}\left(
\overline{s}_{n+1}\right)  \right)  .$ Now, $\varphi_{y_{n+1}}\left(
f^{-1}\left(  D\right)  \right)  <\frac{1}{3}$ and $\varphi_{y_{n+1}}\left(
f^{-1}\left(  \overline{s}_{n+1}\right)  \right)  \leq\varphi_{y_{n+1}}\left(
f^{-1}\left(  s_{n+1}\right)  \right)  ,$ hence $\varphi_{y_{n+1}}\left(
f^{-1}\left(  s_{n+1}\right)  \right)  \geq\frac{1}{3}.$ It is easy to see
that $s_{n+1}$ has all the desired properties. Finally we choose $u_{n+1}$ accordingly.

Letting $A={\textstyle\bigcup\limits_{n\in\omega}}
s_{n},$ it is easy to see that $f^{-1}\left(  A\right)  \notin\mathcal{Z}$ and
$\mathcal{A\cup}\left\{  A\right\}  $ is $\mathcal{B}$-\textsf{AD}.
\end{proof}

With a suitable bookkeeping device, we conclude the following:

\begin{theorem}[\textsf{CH}]    \label{ZMADfamily}
There is a $\mathcal{Z}$-\textsf{MAD} family that is not
Shelah-Stepr\={a}ns for block sequences.
\end{theorem}

Therefore, the countable family mentioned in Proposition~\ref{countableexception} is indeed
needed. Theorem~\ref{ZMADfamily} motivates the following questions:

\begin{problem}
\begin{enumerate}
\item Do $\mathcal{Z}$-\textsf{MAD} families exist (in \textsf{ZFC})?

\item Is it consistent that every \textsf{MAD} family that is not Shelah-Stepr\={a}ns
(or Shelah-Stepr\={a}ns for block sequences) is Kat\v{e}tov below
$\mathcal{Z}$?
\end{enumerate}
\end{problem}




\section{A \textsf{MAD} family that can not be extended to a Hurewicz
ideal}

We know that every Shelah-Stepr\={a}ns \textsf{MAD} family is Hurewicz (Proposition~\ref{ShelahSteprans-Hurewicz}).
Non Canjar \textsf{MAD} families can be
constructed in \textsf{ZFC}, see
\cite{MathiasForcingandCombinatorialCoveringPropertiesofFilters} or
\cite{CanjarFilters}). In order to solve the Problem~\ref{bsa-problem},
it would be enough to show that every \textsf{MAD} family can be extended to a
Hurewicz ideal. Unfortunately, we will see that it is consistent with
\textsf{CH }that this is not the case.

\begin{definition}
\begin{enumerate}
\item By \textsf{Part }we denote the set of partitions of $\omega$
into infinitely many infinite pieces.

\item Given $\mathcal{D}\in$ \textsf{Part }we define \textsf{fin}$\times
$\textsf{fin}$\left(  \mathcal{D}\right)  =\left\{  B\subseteq\omega
\mid\forall^{\infty}D\in\mathcal{D}\left(  \left\vert D\cap B\right\vert
<\omega\right)  \right\}  .$

\item Given two elements $\mathcal{D=}$ $\left\{  D\left(  n\right)  \mid n\in
\omega\right\}  ,\mathcal{C=}\left\{  C\left(  n\right)  \mid n\in
\omega\right\}  $ of \textsf{Part}, we say
$\mathcal{C\longrightarrow D}$ if one of the following conditions holds:
\begin{enumerate}
\item $C\left(  n\right)  $ is almost disjoint with every $D\left(  m\right)
$ for almost all $n\in\omega$ or

\item there is $m\in\omega$ such that $C\left(  n\right)  \subseteq^{\ast
}D\left(  m\right)  $ for almost all $n\in\omega$ or

\item for almost every $n\in\omega$ there is $m_{n}\in\omega$ such that
$C\left(  n\right)  \subseteq^{\ast}D\left(  m_{n}\right)  $ and $m_{k}\neq
m_{r}$ whenever $k\neq r.$
\end{enumerate}
\end{enumerate}
\end{definition}

Note that if $\mathcal{C\longrightarrow D}$ then almost every element of
$\mathcal{C}$ is in \textsf{fin}$\times$\textsf{fin}$\left(  \mathcal{D}%
\right)  .$ Now we define the following:

\begin{definition}
Let $\mathbb{P}$ be the set of all $p=\left(  \mathcal{A}_{p},\mathcal{K}%
_{p}\right)  $ such that $\mathcal{A}_{p}$ is a countable \textsf{AD} family
and $\mathcal{K}_{p}$ is a countable subset of \textsf{Part. }We say that $p=$
$\left(  \mathcal{A}_{p},\mathcal{K}_{p}\right)  \leq q=\left(  \mathcal{A}%
_{q},\mathcal{K}_{q}\right)  $ if the following conditions hold:

\begin{enumerate}
\item $\mathcal{A}_{q}\subseteq\mathcal{A}_{p}$ and $\mathcal{K}_{q}%
\subseteq\mathcal{K}_{p}.$

\item If $\mathcal{C\in K}_{p}\setminus\mathcal{K}_{q}$ and $\mathcal{D\in
K}_{q}$ then $\C\longrightarrow\mathcal{D}.$

\item If $A\in\mathcal{A}_{p}\setminus\mathcal{A}_{q}$ and $\mathcal{D\in
K}_{q}$ then $A\in$ \textsf{fin}$\times$\textsf{fin}$\left(  \mathcal{D}%
\right)  .$
\end{enumerate}
\end{definition}


It is easy to see that $\mathbb{P}$ is a $\sigma$-closed forcing (so it does
not add new reals). Let $\mathcal{\dot{A}}_{gen}$ be the name of ${\textstyle\bigcup}
\{\mathcal{A}_{p}\mid p\in\dot{G}\}$ (where $\dot{G}$ is the name for the
generic filter). It is easy to see that $\mathcal{\dot{A}}_{gen}$ is forced to
be an almost disjoint family, and we will see that $\mathcal{\dot{A}}_{gen}$ is
forced to be a \textsf{MAD} family that can not be extended to a Hurewicz
ideal. Recall that a family $\mathcal{H}\subseteq\left[  \omega\right]
^{\omega}$ is open dense in $\wp\left(  \omega\right)  /$ \textsf{fin} if for
every $A\in\left[  \omega\right]  ^{\omega}$ there is a $B\in\mathcal{H}$ such
that $B\subseteq^{\ast}A$ and $\mathcal{H}$ is closed under almost inclusion.
It is well known and easy to see that every tall ideal on $\omega$ is open
dense in $\wp\left(  \omega\right)  /$ \textsf{fin} and the intersection of
countably many open dense sets is open dense.

\begin{lemma}
$\mathcal{\dot{A}}_{gen}$ is forced to be a \textsf{MAD} family.
\end{lemma}

\begin{proof}
Letting $p=\left(  \mathcal{A}_{p},\mathcal{K}_{p}\right)  \in\mathbb{P}$ and
$X\in\left[  \omega\right]  ^{\omega},$ we must find $q=\left(  \mathcal{A}%
_{q},\mathcal{K}_{q}\right)  \leq p$ such that $X$ is not \textsf{AD} with
$\mathcal{A}_{q}$ (this is enough since $\mathbb{P}$ does not add new reals).
Assume that $X$ is \textsf{AD} with $\mathcal{A}_{p}.$ By the previous
remarks, we can find $A\subseteq^{\ast}X$ such that $A\in$ \textsf{fin}%
$\times$\textsf{fin}$\left(  \mathcal{D}\right)  $ for every $\mathcal{D\in
K}_{p}.$ It is clear that $q=\left(  \mathcal{A}_{p}\cup\left\{  A\right\}
,\mathcal{K}_{p}\right)  $ has the desired properties.
\end{proof}


Recall that an ideal $\mathcal{I}$ in a countable set is a $P^{+}$-ideal if for every decreasing
family $\left\{  X_{n}\mid n\in\omega\right\}  \subseteq\mathcal{I}^{+}$ there
is $X\in\mathcal{I}^{+}$ such that $X\subseteq^{\ast}X_{n}$ for every
$n\in\omega.$ We will need the following:

\begin{lemma}
Let $p=\left(  \mathcal{A}_{p},\mathcal{K}_{p}\right)  \in\mathbb{P}$,
$\mathcal{\dot{J}}$ be a name for a $P^{+}$-ideal such that $p\Vdash
``\mathcal{\dot{A}}_{gen}\subseteq$ $\mathcal{\dot{J}}\textquotedblright$ and
$X\in\left[  \omega\right]  ^{\omega}$ such that $p\Vdash``X\in\mathcal{\dot
{J}}^{+}\textquotedblright.$ There are $q\leq p$ and $Y\in\left[  X\right]
^{\omega}$ such that the following conditions hold:
\begin{enumerate}
\item $Y\in\mathcal{A}_{p}^{\perp}.$

\item For every $\mathcal{D}\in\mathcal{K}_{p}$, either $Y$ is \textsf{AD} with
all elements of $\mathcal{D}$ or there is $D\in\mathcal{D}$ such that $Y\subseteq^{\ast}D.$

\item $q\Vdash``Y\in\mathcal{\dot{J}}^{+}\textquotedblright.$
\end{enumerate}
\end{lemma}

\begin{proof}
We first note that if $\mathcal{I}$ is a $P^{+}$-ideal,
$Z\in\mathcal{I}^{+}$ and $\mathcal{C}=\left\{  C\left(  n\right)  \mid
n\in\omega\right\}  \in$ \textsf{Part}, then there is $W\in\left[  Z\right]
^{\omega}\cap\mathcal{I}^{+}$ such that either $W$ is \textsf{AD} with
$\mathcal{C}$ or there is $C\in\mathcal{C}$ such that $W\subseteq C.$ Indeed,
if there is $n\in\omega$ such that $Z\cap C\left(  n\right)  \in
\mathcal{I}^{+}$ then we define $W=Z\cap C\left(  n\right)  ,$ and if this is not
the case, then $\{Z\setminus{\textstyle\bigcup\limits_{i\leq n}}
C\left(  i\right)  \mid n\in\omega\}\subseteq\mathcal{I}^{+}$ forms a
decreasing sequence, and we just let $W\subseteq Z$ be a pseudointersection in
$\mathcal{I}^{+}.$


To prove the lemma, note that if $A\in\mathcal{A}_{p},$ then $p$ forces that $X\setminus A$ is
in $\mathcal{\dot{J}}^{+}.$ Since both $\mathcal{A}_{p}$ and
$\mathcal{K}_{p}$ are countable, the result then follows by the previous remarks.
\end{proof}


With the lemma, we can now prove the following:

\begin{lemma}   \label{Partlemma}
Let $p=\left(  \mathcal{A}_{p},\mathcal{K}_{p}\right)  \in\mathbb{P}$,
$\mathcal{\dot{J}}$ be a name for a $P^{+}$-ideal such that $p\Vdash
``\mathcal{\dot{A}}_{gen}\subseteq$ $\mathcal{\dot{J}}\textquotedblright$ and
$\left\{  X_{n}\mid n\in\omega\right\}  \subseteq\left[  \omega\right]
^{\omega}$ such that $X_{n}\cap X_{m}=\emptyset$ whenever $n\neq m$ and
$p\Vdash``X_{n}\in\mathcal{\dot{J}}^{+}\textquotedblright$ for every
$n\in\omega.$ There are $q=\left(  \mathcal{A}_{q},\mathcal{K}_{q}\right)
\leq p,$ $W\in\left[  \omega\right]  ^{\omega}$ and $\left\{  Y_{n}\mid n\in
W\right\} \in$ \textsf{Part} such that the following
conditions hold:
\begin{enumerate}
\item $Y_{n}\subseteq^{\ast}X_{n}$ for every $n\in W.$

\item $\mathcal{Y=}$ $\left\{  Y_{n}\mid n\in W\right\}  \in\mathcal{K}_{q}$
and $\Y \longrightarrow \D$ for every $\D \in \K_q \sem \{ \Y \}$.

\item $q\Vdash``Y_{n}\in\mathcal{\dot{J}}^{+}\textquotedblright$ for every
$n\in W.$

\item $q$ forces that every element in $\mathcal{\dot{A}}_{gen}$ has infinite
intersection with only finitely many elements in $\mathcal{Y}.$
\end{enumerate}
\end{lemma}

\begin{proof}
Using the previous lemma, we recursively construct a sequence $\left\langle
\left(  p_{n},Z_{n}\right)  \right\rangle _{n\in\omega}$ with the following properties:
\begin{enumerate}
\item $\left\langle p_{n}\right\rangle _{n\in\omega}$ is decreasing and
$p_{0}\leq p.$

\item $Z_{n}\in\left[  X_{n}\right]  ^{\omega}$.

\item $p_{n}\Vdash``Z_{n}\in\mathcal{\dot{J}}^{+}\textquotedblright.$

\item $Z_{n}\in\mathcal{A}_{p_{n-1}}^{\perp}$ (where $p_{-1}=p$).

\item For every $\mathcal{D}\in\mathcal{K}_{p_{n-1}}$ either $Z_{n}$ is
\textsf{AD} with $\mathcal{D}$ or there is $D\in\mathcal{D}$ such that
$Z_{n}\subseteq^{\ast}D.$
\end{enumerate}
The construction is straightforward. Let $p_{\omega
}=\left(  \mathcal{A}_{p_{\omega}},\mathcal{K}_{p_{\omega}}\right)  $ where
$\mathcal{A}_{p_{\omega}}={\textstyle\bigcup}
\mathcal{A}_{p_{n}}$ and $\mathcal{K}_{p_{\omega}}={\textstyle\bigcup}
\mathcal{K}_{p_{n}}$.   
By our construction, $p_{\omega}$ has the following properties:
\begin{enumerate}
\item If $A\in\mathcal{A}_{p_{\omega}}$ then $A\cap Z_{n}$ \ is finite for
almost every $n\in\omega.$

\item If $\mathcal{D}\in\mathcal{K}_{p_{\omega}}$ \ then for almost all
$n\in\omega,$ either $Z_{n}$ is \textsf{AD} with $\mathcal{D}$ or there is
$D\in\mathcal{D}$ such that $Z_{n}\subseteq^{\ast}D.$
\end{enumerate}
We can then find $W\in\left[  \omega\right]  ^{\omega}$ and $\mathcal{Y}%
=\left\{  Y_{n}\mid n\in W\right\}  $ such that the following conditions hold:
\begin{enumerate}
\item $Y_{n}=^{\ast}Z_{n}$ for every $n\in W.$

\item $\mathcal{Y}\in$ \textsf{Part}$\mathsf{.}$

\item $\mathcal{Y}\longrightarrow\mathcal{D}$ for every $\mathcal{D}%
\in\mathcal{K}_{p_{\omega}}.$
\end{enumerate}
Such $W$ can be found since its construction only requires intersecting
countably many open dense subsets of $\wp\left(  \omega\right)  /$
\textsf{fin}$.$ Letting $q=\left(  \mathcal{A}_{p_{\omega}},\mathcal{K}%
_{p_{\omega}}\cup\left\{  \mathcal{Y}\right\}  \right)  ,$ it is easy to see
that $q$ has the desired properties.
\end{proof}

It is well known that every Canjar ideal is a $P^{+}$-ideal (see
\cite{MathiasPrikryandLaverPrikryTypeForcing},
\cite{MathiasForcingandCombinatorialCoveringPropertiesofFilters} or
\cite{CanjarFilters}). We now have the following result, which is the heart of
the construction:

\begin{proposition}
Let $G\subseteq\mathbb{P}$ be a generic filter. The following holds in
$V\left[  G\right]  :$ If $\mathcal{J}$ is a Canjar ideal such that
$\mathcal{A}_{gen}\subseteq\mathcal{J},$ then there are no $\left\{  X_{n}\mid
n\in\omega\right\}  \subseteq\mathcal{J}^{+}$ such that $X_{n}\cap
X_{m}=\emptyset$ for every $n\in\omega.$
\end{proposition}

\begin{proof}
Let $p\in\mathbb{P},$ $\mathcal{\dot{J}}$ a name for a $P^{+}$-ideal extending
$\mathcal{\dot{A}}_{gen}$ and $\left\{  X_{n}\mid n\in\omega\right\}  $ a
pairwise disjoint family such that $p\Vdash``\left\{  X_{n}\mid n\in
\omega\right\}  \subseteq\mathcal{\dot J}^{+}\textquotedblright.$ We will find an
extension of $p$ that forces that $\mathcal{\dot{J}}$ is not Canjar. By the
previous lemma, let $q=\left(  \mathcal{A}_{q},\mathcal{K}_{q}\right)  \leq
p,$ $W\in\left[  \omega\right]  ^{\omega}$ and $\left\{  Y_{n}\mid n\in
\omega\right\}  \subseteq\left[  \omega\right]  ^{\omega}$ be such that the
following conditions hold:
\begin{enumerate}
\item $Y_{n}\subseteq^{\ast}X_{n}$ for every $n\in W.$

\item $\mathcal{Y=}$ $\left\{  Y_{n}\mid n\in W\right\}  \in\mathcal{K}_{q}$
and $\Y \longrightarrow \D$ for every $\D \in \K_q \sem \{ \Y \}$.

\item $q\Vdash``Y_{n}\in\mathcal{\dot{J}}^{+}\textquotedblright$ for every
$n\in W.$

\item $q$ forces that every element in $\mathcal{\dot{A}}_{gen}$ has infinite
intersection with only finitely many elements in $\mathcal{Y}.$
\end{enumerate}
Let $W=\left\{  w_{n}\mid n\in\omega\right\}  .$ For every $n\in\omega,$ we
define $F_{n}$ to be the set of all $\left\{  a_{0},...,a_{n}\right\}  $ such
that $a_{0}<a_{1}<...<a_{n}$ and $a_{i}\in Y_{w_{i}}$ for every $i\leq n.$
It is easy to see that $q\Vdash``F_{n}\in(\mathcal{\dot{J}
}^{<\omega})^{+}\textquotedblright$ for every $n\in\omega.$ We claim that $q$
forces that $\left\langle F_{n}\mid n\in\omega\right\rangle $ witnesses that
(in the extension) $\mathcal{\dot{J}}$ is not a Canjar ideal. It is enough to
prove the following: For every $q_{1}\leq q$ and $\left\langle H_{n}%
\right\rangle _{n\in\omega}$ such that $H_{n}\in\left[  F_{n}\right]
^{<\omega},$ there is $q_{2}=\left(  \mathcal{A}_{q_{2}},\mathcal{K}_{q_{2}%
}\right)  \leq q_{1}$ and $A\in\mathcal{I}\left(  \mathcal{A}_{q_{2}}\right)
$ such that $A$ has non empty intersection with every element of $H_{n}$ for
every \thinspace$n\in\omega$ (recall that $\mathcal{\dot{J}}$ is forced to
extend $\mathcal{\dot{A}}_{gen}$).

Let $q_{1}=\left(  \mathcal{A}_{q_{1}},\mathcal{K}_{q_{1}}\right)  $ be an
extension of $q$ and let $\left\langle H_{n}\right\rangle _{n\in\omega}$ be such that
$H_{n}\in\left[  F_{n}\right]  ^{<\omega}.$ We fix the following items:

\begin{enumerate}
\item Let $\mathcal{A}_{q_{1}}=\left\{  A_{n}\mid n\in\omega\right\}  $ and
define $B_{n}={\textstyle\bigcup\limits_{i\leq n}}
A_{i}$ for every $n\in\omega.$

\item Let $\mathcal{L=}\left\{  \mathcal{C\in K}_{q_{1}}\mid
\mathcal{Y\longrightarrow C}\right\}  .$

\begin{enumerate}
\item Let $\mathcal{L}_{\text{\textsf{AD}}}$ be the family of all
$\mathcal{C}\in\mathcal{L}$ such that almost every element of $\mathcal{Y}$ is
\textsf{AD} with $\mathcal{C}.$ Fix an enumeration $\mathcal{L}%
_{\text{\textsf{AD}}}=\{\C_{n}^{\text{\textsf{AD}}}\mid n\in\omega\}$ where
$\C_{n}^{\text{\textsf{AD}}}=\{C_{n}^{\text{\textsf{AD}}}\left(  m\right)  \mid
m\in\omega\}.$ Let $\overline{C}_{n}^{\text{\textsf{AD}}}\left(  m\right)
=C_{n}^{\text{\textsf{AD}}}\left(  0\right)  \cup...\cup C_{n}%
^{\text{\textsf{AD}}}\left(  m\right)  .$

\item Let $\mathcal{L}_{=}$ be the family of all $\mathcal{C}\in\mathcal{L}$
such that there is $C\in\mathcal{C}$ such that almost all elements of $\mathcal{Y}$
are almost contained in $C.$ Fix an enumeration $\mathcal{L}_{\mathsf{=}%
}=\{\C_{n}^{\mathsf{=}}\mid n\in\omega\}$ where $\C_{n}^{\mathsf{=}}%
=\{C_{n}^{\mathsf{=}}\left(  m\right)  \mid m\in\omega\}$ and let $c_{n}%
^{=}\in\omega$ be such that almost every element of $\mathcal{Y}$ is almost
contained in $C_{n}^{=}\left(  c_{n}^{=}\right)  .$

\item Let $\mathcal{L}_{\neq}$ be the family of all $\mathcal{C}\in
\mathcal{L}$ such that for almost all $n\in W$, the $Y_n$ are almost contained in
pairwise distinct members of $\C$.
 Fix an enumeration $\mathcal{L}_{\mathsf{\neq}}=\{\C_{n}^{\mathsf{\neq}}\mid n\in\omega\}$ where
$\C_{n}^{\mathsf{\neq}}=\{C_{n}^{\mathsf{\neq}}\left(  m\right)  \mid
m\in\omega\}.$
\end{enumerate}
Note that $\LLL$ is the disjoint union of $\LLL_ {\text{\textsf{AD}}}$, $\LLL_=$,
and $\LLL_{\neq}$.

\item Let $\mathcal{R}$ $=\left\{  \mathcal{D\in K}_{q_{1}}\mid
\mathcal{D\longrightarrow Y}\right\}  .$

\begin{enumerate}
\item Let $\mathcal{R}_{\text{\textsf{AD}}}$ be the family of all
$\mathcal{D}\in\mathcal{R}$ such that almost every element of $\mathcal{D}$ is
\textsf{AD} with $\mathcal{Y}.$ Fix an enumeration $\mathcal{R}%
_{\text{\textsf{AD}}}=\{\D_{n}^{\text{\textsf{AD}}}\mid n\in\omega\}$ where
$\D_{n}^{\text{\textsf{AD}}}=\{D_{n}^{\text{\textsf{AD}}}\left(  m\right)  \mid
m\in\omega\}.$ Let $d_{n}^{\text{\textsf{AD}}}\in\omega$ such that if
$d_{n}^{\text{\textsf{AD}}}\leq m$ then $D_{n}^{\text{\textsf{AD}}}\left(
m\right)  $ is \textsf{AD} with $\mathcal{Y}$ and  let $\overline{D}%
_{n}^{\text{\textsf{AD}}}\left(  m\right)  ={\textstyle\bigcup}
\{D_{n}^{\text{\textsf{AD}}}\left(  i\right)  \mid d_{n}^{\text{\textsf{AD}}%
}\leq i\leq m\}.$

\item Let $\mathcal{R}_{=}$ be the family of all $\mathcal{D}\in\mathcal{R}$
such there is $Y_{n}\in\mathcal{Y}$ such that almost all elements of
$\mathcal{D}$ are almost contained in $Y_{n}.$ Fix an enumeration
$\mathcal{R}_{\mathsf{=}}=\{\D_{n}^{\mathsf{=}}\mid n\in\omega\}$ where
$\D_{n}^{\mathsf{=}}=\{D_{n}^{\mathsf{=}}\left(  m\right)  \mid m\in\omega\},$
and let $d_{n}^{=},e_{n}^{=}\in\omega$ such that $D_{n}^{=}\left(  m\right)
\subseteq^{\ast}Y_{d_{n}^{=}}$ for every $m\geq e_{n}^{=}.$

\item Let $\mathcal{R}_{\neq}$ be the family of all $\mathcal{D}\in
\mathcal{R}$ such that almost all $D\in\mathcal{D}$ are almost contained
in pairwise distinct members of $\Y$. Fix an enumeration $\mathcal{R}_{\mathsf{\neq}}%
=\{\D_{n}^{\mathsf{\neq}}\mid n\in\omega\}$ where $\D_{n}^{\mathsf{\neq}%
}=\{D_{n}^{\mathsf{\neq}}\left(  m\right)  \mid m\in\omega\}.$
\end{enumerate}
Note that $\R$ is the disjoint union of $\R_ {\text{\textsf{AD}}}$, $\R_=$,
and $\R_{\neq}$.

\item Let $h:\omega\longrightarrow\omega$ be an increasing function such that
for every $n\in\omega,$ the following conditions hold:

\begin{enumerate}
\item If $h\left(  n\right)  \leq m$ then $B_{n}$ is almost disjoint with
$Y_{w_m}.$

\item If $h\left(  n\right)  \leq m$ then $Y_{w_m}$ is almost disjoint with
every element of $\C_{n}^{\text{\textsf{AD}}}.$

\item If $h\left(  n\right)  \leq m$ then $Y_{w_m}\subseteq^{\ast}C_{n}%
^{=}\left(  c_{n}^{=}\right)  .$

\item If $h\left(  n\right)  \leq m<k$ then $Y_{w_m}$ and $Y_{w_k}$ are almost
subsets of different elements in $\C_{n}^{\neq}.$

\item If $h\left(  n\right)  \leq m$ then $D_{n}^{\text{\textsf{AD}}}\left(
m\right)  $ is \textsf{AD} with $\mathcal{Y}$   (i.e. $d_n^{\text{\textsf{AD}}} \leq h(n)$).

\item If $h\left(  n\right)  \leq m$ then $D_{n}^{\mathsf{=}}\left(  m\right)
\subseteq^{\ast}Y_{d_{n}^{=}}$   (i.e. $e_n^= \leq h(n)$).

\item If $h\left(  n\right)  \leq m<k$ then $D_{n}^{\neq}\left(  m\right)  $
and $D_{n}^{\neq}\left(  k\right)  $ are almost subsets of different elements
in $\mathcal{Y}.$
\end{enumerate}
\end{enumerate}
Note that $\K_{q_1} = \LLL \cup \R \cup \{ \Y \}$. This follows from the construction of
$\Y$ and $q_1 \leq q$.

We will say that a finite sequence $\left\langle \left(  P_{0},s_{0}\right)
,...,\left(  P_{n},s_{n}\right)  \right\rangle $ is \emph{suitable }if the
following conditions hold:

\begin{enumerate}
\item $P_{0},...,P_{n}$ are non-empty finite consecutive intervals of
$\omega.$

\item $\min\left(  P_{0}\right)  =0.$

\item $s_{0}=P_{0}\cap\left({\textstyle\bigcup}
\left\{  Y_{w_j}\mid j<h\left(  0\right)  \right\}  \right)  .$

\item $s_{i+1}=P_{i+1}\cap\left({\textstyle\bigcup}
\left\{  Y_{w_j}\mid h\left(  i\right)  \leq j<h\left(  i+1\right)  \right\}
\right)  .$

\item $s_{i}\neq\emptyset$ for every $i\leq n.$

\item If $m<h\left(  n\right)  $ then ${\textstyle\bigcup\limits_{i\leq n}}
s_{i}$ has non-empty intersection with every element of $H_{m}.$

\item $B_{i}\cap\left({\textstyle\bigcup}
\left\{  Y_{w_j}\mid h\left(  i\right)  \leq j<h\left(  i+1\right)  \right\}
\right)  \subseteq \max\left(  P_{i}\right)  .$

\item $\overline{C}_{l}^{\text{\textsf{AD}}}\left(  h\left(  i\right)
\right)  \cap\left({\textstyle\bigcup}
\left\{  Y_{w_j}\mid h\left(  i\right)  \leq j<h\left(  i+1\right)  \right\}
\right)  \subseteq \max\left(  P_{i}\right)  $ for every $l\leq i.$

\item ${\textstyle\bigcup}
\left\{  Y_{w_j}\mid h\left(  i\right)  \leq j<h\left(  i+1\right)  \right\}
\setminus C_{l}^{=}\left(  c_{l}^{=}\right)  \subseteq \max\left(
P_{i}\right)  $ for every $l\leq i.$

\item For every $j$ such that $h\left(  i\right)  \leq j<h\left(  i+1\right)
,$ if $Y_{w_j}$ is not an almost subset of $C_{k}^{\neq}\left(  l\right)  $ with
$k,l<i$ then $Y_{w_j}\cap C_{k}^{\neq}\left(  l\right)  \subseteq \max\left(
P_{i}\right)  $.

\item $\overline{D}_{l}^{\text{\textsf{AD}}}\left(  h\left(  i\right)
\right)  \cap\left({\textstyle\bigcup}
\left\{  Y_{w_j}\mid h\left(  i\right)  \leq j<h\left(  i+1\right)  \right\}
\right)  \subseteq \max\left(  P_{i}\right)  $ for every $l\leq i.$

\item If $l\leq i$ and $e_{l}^{=}\leq j\leq h\left(  i\right)  $ then
$D_{l}^{=}\left(  j\right)  \setminus \max\left(  P_{i}\right)  \subseteq
Y_{d_{l}^{=}}.$

\item For every $j$ such that $h\left(  i\right)  \leq j<h\left(  i+1\right)
,$ if $D_{k}^{\neq}\left(  l\right)  $ is not an almost subset of $Y_{w_j}$ with
$k,l<i$ then $Y_{w_j}\cap D_{k}^{\neq}\left(  l\right)  \subseteq \max\left(
P_{i}\right)  $.
\end{enumerate}
Although the list of requirements neede to be verified is excessively long, it is not hard to see
that for every suitable $\left\langle \left(  P_{0},s_{0}\right)  ,...,\left(
P_{n},s_{n}\right)  \right\rangle $ there is $\left(  P_{n+1},s_{n+1}\right)
$ such that $\left\langle \left(  P_{0},s_{0}\right)  ,...,\left(  P_{n}%
,s_{n}\right)  ,\left(  P_{n+1},s_{n+1}\right)  \right\rangle $ is suitable.
Therefore, we can recursively construct $\left\langle \left(  P_{0}%
,s_{0}\right)  ,...,\left(  P_{n},s_{n}\right)  ,...\right\rangle _{n\in
\omega}$ such that every initial segment is suitable. Letting $A={\textstyle\bigcup\limits_{n\in\omega}}
s_{n}$ it is easy to see that $q_{2}=\left(  \mathcal{A}_{q_{1}}\cup\left\{
A\right\}  ,\mathcal{K}_{q_{1}}\right)  $ is a condition extending $q_{1},$
which is the extension we were looking for.
\end{proof}

We now obtain the main result of this section:

\begin{theorem}   \label{notHurewicz}
Let $G\subseteq\mathbb{P}$ be a generic filter. If $\mathcal{J}$ is a Canjar ideal extending
$\mathcal{I}\left(  \mathcal{A}_{gen}\right)$ in $V[G]$,  then the following hold:

\begin{enumerate}
\item For every $X\in\mathcal{J}^{+}$ there is $Y\in\mathcal{J}^{+}\cap\left[
X\right]  ^{\omega}$ such that $\mathcal{J}^{\ast}\upharpoonright Y$ is an ultrafilter.

\item Forcing with $\mathbb{M}\left(  \mathcal{J}\right)  $ diagonalizes an ultrafilter.

\item $\mathcal{J}$ is not a Hurewicz ideal.
\end{enumerate}
\end{theorem}

\begin{proof}
It is easy to see that 2 follows from 1 and 3 is a consequence of 2 (this follows e.g.
from Proposition~\ref{Hurewicztailsplitting}), so we
only prove 1. Assume this is not the case, so for every $Y\in\mathcal{J}%
^{+}\cap\left[  X\right]  ^{\omega}$ it is the case that $\mathcal{J}^{\ast
}\upharpoonright Y$ is not an ultrafilter. We can then recursively construct a
pairwise disjoint family $\left\{  X_{n}\mid n\in\omega\right\}
\subseteq\mathcal{J}^{+}\cap\left[  X\right]  ^{\omega},$ which contradicts
the previous result.
\end{proof}

We  point out that under the Continuum Hypothesis, every
\textsf{MAD} family can be extended to the dual of a Canjar ultrafilter~\cite{MobandMad}.
Finally, we will prove the following result:

\begin{proposition}
$\mathcal{\dot{A}}_{gen}$ is forced to be a Cohen-destructible \textsf{MAD} family.
\end{proposition}

\begin{proof}
Recall that every Cohen indestructible \textsf{MAD} family has a restriction
that is tight (see \cite{OrderingMADFamiliesalaKatetov}). In fact, we will
prove that no restriction of $\mathcal{A}_{gen}$ is weakly tight. Let
$p\in\mathbb{P}$ and $X\in\left[  \omega\right]  ^{\omega}$ such that
$p\Vdash``X\in\mathcal{I(\dot{A}}_{gen}\mathcal{)}^{+}\textquotedblright.$
Since the ideals generated by \textsf{MAD} families are hereditarily meager,
we may assume there is a pairwise disjoint family $\left\{  X_{n}\mid
n\in\omega\right\}  \subseteq\left[  X\right]  ^{\omega}$ such that
$p\Vdash``X_{n}\in\mathcal{I(\dot{A}}_{gen}\mathcal{)}^{+}\textquotedblright$
for every $n\in\omega.$ Since the ideals generated by \textsf{AD} families are
$P^{+}$-ideals (see \cite{HappyFamilies} or
\cite{AlmostDisjointFamiliesandTopology}), we know by Lemma~\ref{Partlemma} there are $q\leq p,$
$W\in\left[  \omega\right]  ^{\omega}$ and $\left\{  Y_{n}\mid n\in W\right\}
\subseteq\left[  \omega\right]  ^{\omega}$ such that the following conditions hold:

\begin{enumerate}
\item $Y_{n}\subseteq^{\ast}X_{n}$ for every $n\in W.$

\item $q\Vdash``Y_{n}\in\mathcal{I(\dot{A}}_{gen}\mathcal{)}^{+}%
\textquotedblright$ for every $n\in W.$

\item $q$ forces that every element in $\mathcal{\dot{A}}_{gen}$ has infinite
intersection with only finitely many elements in $\mathcal{Y}.$
\end{enumerate}
Clearly $q$ forces that the family $\left\{  Y_{n}\mid n\in W\right\}  $
witnesses that $\mathcal{\dot{A}}_{gen}$ is not weakly tight.
\end{proof}

The result raises the following question:

\begin{problem}
Is there (consistently) a Cohen indestructible \textsf{MAD} family that can
not be extended to a Hurewicz ideal?
\end{problem}
\section{Preserving big families} \label{sec:big}
In Sections \ref{sec:big}, \ref{sec:partial}, and \ref{sec:mainresult}, we show that it is consistent to have $\dilipnon(\dilipMMM) = {\aleph}_{1}$ and no Shelah-Stepr{\= a}ns or block Shelah-Stepr{\= a}ns a.d.\@ families of size ${\aleph}_{1}$.
This improves Theorem \ref{thm:59}.

A general strategy for preserving certain non-meager sets from the ground model is provided in this section.
Our results here allow us to keep $\dilipnon(\dilipMMM) = {\aleph}_{1}$ in countable support iterations which do not enjoy stronger preservation properties, such as the preservation of ${\sqsubseteq}_{\mathrm{Cohen}}$.
\begin{dilipDef} \label{def:ip}
 An \emph{interval partition or IP} is a sequence $I = \dilipseq{i}{n}{\in}{\omega} \in {\omega}^{\omega}$ such that ${i}_{0} = 0$ and $\forall n \in \omega\left[{i}_{n} < {i}_{n+1}\right]$.
 
 Given an IP $I$ and $n \in \omega$, ${I}_{n}$ denotes $\left[ {i}_{n}, {i}_{n+1} \right) = \{l \in \omega: {i}_{n} \leq l < {i}_{n+1} \}$.
\end{dilipDef}
The following is a slight variation of a well-known connection between eventually different reals and meagerness of the ground model.
We give a proof even though the argument is similar to the arguments in Miller~\cite{arnieMN} or Bartoszynski and Judah~\cite{BJ}.
\begin{dilipLemma} \label{lem:miller}
 Let ${\dilipV}_{0} \subseteq {\dilipV}_{1}$ be transitive models of a sufficiently large fragment of $\dilipZFC$.
 Assume that
 \begin{align*}
  \forall f \in {H({\aleph}_{0})}^{\omega} \cap {\dilipV}_{1} \exists g \in {H({\aleph}_{0})}^{\omega} \cap {\dilipV}_{0} \dilipexistsinf n \in \omega \left[f(n) = g(n)\right].
 \end{align*}
 Then the following hold:
 \begin{enumerate}
  \item
  for each $f \in {H({\aleph}_{0})}^{\omega} \cap {\dilipV}_{1}$ and each $X \in {\left[\omega\right]}^{\omega} \cap {\dilipV}_{1}$, there exists a $g \in {H({\aleph}_{0})}^{\omega} \cap {\dilipV}_{0}$ so that
  \begin{align*}
   \dilipexistsinf n \in X \left[f(n) = g(n)\right];
  \end{align*}
  \item
  for each IP $I = \dilipseq{i}{n}{\in}{\omega} \in \dilipBS \cap {\dilipV}_{1}$, there is an IP $J = \dilipseq{j}{n}{\in}{\omega} \in \dilipBS \cap {\dilipV}_{0}$ so that
  \begin{align*}
   \dilipexistsinf l \in \omega \exists n \in \omega\left[{I}_{n} \subseteq {J}_{l}\right];
  \end{align*}
  \item
  for any $M \in {\dilipV}_{1}$, if ${\left( M \subseteq {2}^{\omega} \ \text{is a meager set} \right)}^{{\dilipV}_{1}}$, then there exists $x \in {2}^{\omega} \cap {\dilipV}_{0}$ with $x \notin M$.
 \end{enumerate}
\end{dilipLemma}
\begin{proof}
 For (1): working in ${\dilipV}_{1}$, fix $f \in {H({\aleph}_{0})}^{\omega}$ and $X \in {\left[\omega\right]}^{\omega}$.
 Let $\dilipseq{x}{n}{\in}{\omega}$ be the strictly increasing enumeration of $X$.
 Recall that whenever $A \in H({\aleph}_{0})$, then ${H({\aleph}_{0})}^{A} \subseteq H({\aleph}_{0})$.
 For each $n \in \omega$, $\{{x}_{i}: i < n\}$, being a finite subset of $\omega$, is a member of $H({\aleph}_{0})$.
 Therefore $f \diliprestrict \{{x}_{i}: i < n\} \in {H({\aleph}_{0})}^{\{{x}_{i}: i < n\}} \subseteq H({\aleph}_{0})$.
 So we may define a function $F: \omega \rightarrow H({\aleph}_{0})$ by setting $F(n) = f \diliprestrict \{{x}_{i}: i < n\}$, for each $n \in \omega$.
 By hypothesis, there is a function $G \in {\dilipV}_{0}$ so that $G: \omega \rightarrow \diliphf$ and $\dilipexistsinf n \in \omega\left[F(n) = G(n)\right]$.
 Observe that whenever $F(n) = G(n)$, $G(n)$ is a function and $\dilipdom(G(n)) \in {\left[\omega\right]}^{n}$.
 Thus working in ${\dilipV}_{0}$, we see that the set $Y = \left\{n \in \omega: G(n) \ \text{is a function and} \ \dilipdom(G(n)) \in {\left[\omega\right]}^{n} \right\}$ is infinite.
 Let $\dilipseq{y}{n}{\in}{\omega}$ be the strictly increasing enumeration of $Y$.
 Define a function $e \in \dilipBS$ by induction as follows.
 Let $n \in \omega$ and assume that $e(m)$ has been defined for all $m < n$.
 Then $\diliplc \{e(m): m < n\} \diliprc \leq n$, while $\diliplc \dilipdom(G({y}_{n+1})) \diliprc = {y}_{n+1} \geq n+1 > n$.
 So define $e(n) = \min(\dilipdom(G({y}_{n+1})) \setminus \{e(m): m < n\})$.
 Note that $e$ is a 1-1 function and that $e(n) \in \dilipdom(G({y}_{n+1}))$, which means that $G({y}_{n+1})(e(n))$ is defined and is a member of $H({\aleph}_{0})$.
 Define $g: \omega \rightarrow \diliphf$ so that $g(e(n)) = G({y}_{n+1})(e(n))$, for all $n \in \omega$, while $g(k) = 0$, for all $k \notin \dilipran(e)$.
 We check that $g$ is as needed.
 We know $Z = \{n \in \omega: F({y}_{n+1}) = G({y}_{n+1})\}$ is infinite.
 If $n \in Z$, then $\dilipdom(G({y}_{n+1})) \subseteq X$, whence $e(n) \in X$ and $g(e(n)) = G({y}_{n+1})(e(n)) = F({y}_{n+1})(e(n)) = f(e(n))$.
 Thus $e''Z \subseteq \{k \in X: g(k) = f(k)\}$.
 As $e$ is a 1-1 function, $\{k \in X: g(k) = f(k)\}$ is an infinite set.
 
 For (2): working in ${\dilipV}_{1}$, fix an IP $I = \dilipseq{i}{n}{\in}{\omega}$.
 Define $f: \omega \rightarrow \diliphf$ by $f(n) = {i}_{n+2}$, for all $n \in \omega$.
 By the hypothesis, there is a $g \in {\dilipV}_{0}$ so that $g: \omega \rightarrow \diliphf$ and $\dilipexistsinf n \in \omega\left[g(n) = {i}_{n+2}\right]$.
 Working in ${\dilipV}_{0}$, define an IP $J = \dilipseq{j}{l}{\in}{\omega}$ as follows.
 ${j}_{0} = 0$.
 Fix $l \in \omega$ and suppose that ${j}_{l} \in \omega$ is given.
 Define ${j}_{l+1} = \max\left( \{{j}_{l} + 1\} \cup \left( \{g(n): n \leq {j}_{l} \} \cap \omega \right) \right)$.
 Note ${j}_{l+1} \geq {j}_{l} + 1 > {j}_{l}$, so we have an IP.
 To check that it is as needed, fix any $M \in \omega$.
 Choose $n > {j}_{M+1}$ with $g(n) = {i}_{n+2}$.
 There is a unique $k \in \omega$ with $n \in {J}_{k}$, that is ${j}_{k} \leq n < {j}_{k+1}$.
 Observe $k \geq M+1 > M$.
 If ${i}_{n+1} < {j}_{k+1}$, then we have ${j}_{k} \leq n \leq {i}_{n} < {i}_{n+1} < {j}_{k+1}$, implying that ${I}_{n} \subseteq {J}_{k}$.
 If ${j}_{k+1} \leq {i}_{n+1}$, then by the definition of ${j}_{k+2}$, we have ${j}_{k+1} \leq {i}_{n+1} < {i}_{n+2} \leq {j}_{k+2}$, implying that ${I}_{n+1} \subseteq {J}_{k+1}$.
 Thus we have found an $l > M$, namely either $l = k$ or $l = k+1$, and an $n'$, namely either $n' = n$ or $n' = n+1$, so that ${I}_{n'} \subseteq {J}_{l}$.
 
 For (3): working in ${\dilipV}_{1}$, fix a meager set $M \subseteq {2}^{\omega}$.
 Let $\dilipseq{F}{n}{\in}{\omega}$ be a sequence of closed nowhere dense subsets of ${2}^{\omega}$ such that $\forall n \in \omega\left[{F}_{n} \subseteq {F}_{n+1}\right]$ and $M \subseteq {\bigcup}_{n \in \omega}{{F}_{n}}$.
 Build an IP $I = \dilipseq{i}{n}{\in}{\omega}$ and a sequence $\dilipseq{\tau}{n}{\in}{\omega}$ as follows.
 Put ${i}_{0} = 0$.
 Let $n \in \omega$ and assume that ${i}_{n} \in \omega$ is given and that ${\tau}_{m}$ has been defined for all $m < n$.
 Find ${i}_{n+1} > {i}_{n}$ and a function ${\tau}_{n}: \left[{i}_{n}, {i}_{n+1}\right) \rightarrow 2$ such that for each $\sigma: {i}_{n} \rightarrow 2$, $\left[\sigma \cup {\tau}_{n}\right] \cap {F}_{n} = 0$.
 This is possible because ${F}_{n}$ is closed nowhere dense.
 Note that $y = {\bigcup}_{n \in \omega}{{\tau}_{n}}: \omega \rightarrow 2$.
 Using (2), find an IP $J = \dilipseq{j}{l}{\in}{\omega} \in {\dilipV}_{0}$ with the property that $\dilipexistsinf l \in \omega \exists n \in \omega\left[{I}_{n} \subseteq {J}_{l}\right]$.
 In ${\dilipV}_{1}$, define a function $F: \omega \rightarrow \diliphf$ by setting $F(l) = y\diliprestrict\left[ {j}_{l}, {j}_{l+1} \right)$, for all $l \in \omega$.
 Note that if ${I}_{n} \subseteq {J}_{l}$, then ${\tau}_{n} = y\diliprestrict{I}_{n} \subseteq y \diliprestrict {J}_{l} = F(l)$.
 Let $X = \{l \in \omega: \exists n \in \omega \left[{I}_{n} \subseteq {J}_{l}\right]\}$.
 Applying (1), find $G \in {\dilipV}_{0}$ such that $G: \omega \rightarrow \diliphf$ and $\dilipexistsinf l \in X\left[G(l) = F(l)\right]$.
 Working in ${\dilipV}_{0}$, define $x: \omega \rightarrow 2$ such that for each $l \in \omega$, if $G(l): {J}_{l} \rightarrow 2$, then $x\diliprestrict{J}_{l} = G(l)$, while if not, then $x\diliprestrict{J}_{l}$ is constantly $0$.
 Suppose $l \in X$ and that $G(l) = F(l)$.
 Choose $n \in \omega$ with ${I}_{n} \subseteq {J}_{l}$.
 Let $\sigma = x \diliprestrict {i}_{n}: {i}_{n} \rightarrow 2$.
 By the choice of ${\tau}_{n}$, $\left[\sigma \cup {\tau}_{n}\right] \cap {F}_{n} = 0$.
 Since $F(l): {J}_{l} \rightarrow 2$, ${I}_{n} \subseteq {J}_{l}$, and $G(l) = F(l)$, ${\tau}_{n} \subseteq F(l) = G(l) \subseteq x$.
 So $x \in \left[\sigma \cup {\tau}_{n}\right]$, whence $x \notin {F}_{n}$.
 Thus we conclude that $\dilipexistsinf n \in \omega\left[x \notin {F}_{n}\right]$.
 As the ${F}_{n}$ are increasing, $x \notin {\bigcup}_{n \in \omega}{{F}_{n}}$, and so $x \notin M$.
\end{proof}
\begin{dilipDef} \label{def:slalom}
 $S: \omega \rightarrow \diliphf$ is called a \emph{slalom} if $\forall n \in \omega \left[\diliplc S(n) \diliprc \leq (n+1){2}^{n+1}\right]$.
 $S: \omega \rightarrow \diliphf$ is called a \emph{small slalom} if $\forall n \in \omega\left[\diliplc S(n) \diliprc \leq {2}^{n+1}\right]$.
\end{dilipDef}
\begin{dilipDef} \label{def:An}
 Let $f \in {\diliphf}^{\omega}$.
 For $n \in \omega$, ${A}_{n} = \{k \leq n: f(k) \ \text{is a function} \ \wedge n \in \dilipdom(f(k)) \wedge \diliplc f(k)(n) \diliprc \leq {2}^{n+1} \}$.
 
 Define ${S}_{f}(n) = \bigcup \left\{ f(k)(n): k \in {A}_{n} \right\}$, $\forall n \in \omega$.
 We observe that
 \begin{align*}
  \forall n \in \omega \left[{S}_{f}(n) \in H({\aleph}_{0}) \ \text{and} \ \diliplc {S}_{f}(n) \diliprc \leq (n+1){2}^{n+1}\right].
 \end{align*}
 Thus ${S}_{f}: \omega \rightarrow H({\aleph}_{0})$ and ${S}_{f}$ is a slalom.
\end{dilipDef}
\begin{dilipLemma} \label{lem:slalombig}
 Assume $\dilipFFF \subseteq {\diliphf}^{\omega}$ is such that $\forall f \in \dilipFFF \left[{S}_{f} \in \dilipFFF\right]$.
 Then if $\forall f \in {\diliphf}^{\omega} \exists g \in \dilipFFF \dilipexistsinf n \in \omega\left[f(n) = g(n)\right]$, then for any small slalom $S: \omega \rightarrow \diliphf$ and any sequence $\dilipseq{X}{l}{\in}{\omega} \subseteq \dilippc{\omega}{\omega}$, there is a slalom $T: \omega \rightarrow \diliphf$ so that $T \in \dilipFFF$ and $\forall l \in \omega \dilipexistsinf n \in {X}_{l}\left[S(n) \subseteq T(n)\right]$.
\end{dilipLemma}
\begin{proof}
 Let $I = \dilipseq{i}{k}{\in}{\omega}$ be an IP such that $\forall k \in \omega \forall l \leq k\left[{X}_{l} \cap {I}_{k} \neq \emptyset \right]$.
 Define $f: \omega \rightarrow \diliphf$ by $f(k) = S\diliprestrict {I}_{k}$, for all $k \in \omega$.
 Note that for each $k \in \omega$, ${\diliphf}^{{I}_{k}} \subseteq \diliphf$, and so $f: \omega \rightarrow \diliphf$.
 By hypothesis, fix $g \in \dilipFFF$ so that $\dilipexistsinf k \in \omega \left[f(k) = g(k)\right]$.
 Let $T = {S}_{g}$, which by definition means that $T$ is a slalom.
 Also by hypothesis, $T \in \dilipFFF$.
 To see that $T$ is as required, fix $l \in \omega$.
 Let $M \in \omega$ be given.
 Let $K = \max\{M, l\} \in \omega$.
 Choose $k > K$ with $f(k) = g(k)$.
 Since $l \leq K < k$, ${X}_{l} \cap {I}_{k} \neq \emptyset$.
 Choose $n \in {X}_{l} \cap {I}_{k}$.
 Thus $n \in {X}_{l}$ and $M \leq K < k \leq {i}_{k} \leq n$.
 As $g(k) = f(k)$, we have that $k \in {A}_{g, n}$, whence $S(n) = g(k)(n) \subseteq {S}_{g}(n) = T(n)$.
 This is as needed because $n \in {X}_{l}$ and $n > M$.
\end{proof}
\begin{dilipDef} \label{def:f<wo}
 Let $\dilipwo$ be a well-ordering of $\diliphf$.
 For any $A \in \diliphf$, $\dilippr{A}{\dilipwo}$ is a finite well-order, and therefore there is a unique function ${e}_{A, \dilipwo}: \diliplc A \diliprc \rightarrow A$ which is an order isomorphism from $\dilippr{\diliplc A \diliprc}{\in}$ to $\dilippr{A}{\dilipwo}$.
 Define an IP $I = \dilipseq{i}{n}{\in}{\omega}$ as follows.
 ${i}_{0} = 0$.
 Given ${i}_{n} \in \omega$, ${i}_{n+1} = {i}_{n} + (n+1){2}^{n+1}$.
 
 Now suppose $S: \omega \rightarrow \diliphf$ is a slalom.
 Define ${f}_{S, \dilipwo}: \omega \rightarrow \diliphf$ as follows.
 Given $x \in \omega$, there exist unique $n, j \in \omega$ satisfying ${i}_{n} \leq x < {i}_{n+1}$, $j < (n+1){2}^{n+1}$, and $x={i}_{n}+j$.
 If $j \geq \diliplc S(n) \diliprc$, then ${f}_{S, \dilipwo}(x) = 0 \in \diliphf$.
 If $j < \diliplc S(n) \diliprc$, then ${e}_{S(n), \dilipwo}(j) \in \diliphf$.
 If ${e}_{S(n), \dilipwo}(j)$ is a function and $\dilipdom({e}_{S(n), \dilipwo}(j)) = \left[{i}_{n}, {i}_{n+1} \right)$, then ${f}_{S, \dilipwo}(x) = {e}_{S(n), \dilipwo}(j)(x) \in \diliphf$.
 Otherwise, ${f}_{S, \dilipwo}(x) = 0 \in \diliphf$.
\end{dilipDef}
The point of this somewhat cumbersome definition is the following lemma.
\begin{dilipLemma} \label{lem:cumber}
 Assume $\dilipFFF \subseteq {\diliphf}^{\omega}$ has the property that for every slalom $S: \omega \rightarrow \diliphf$, if $S \in \dilipFFF$, then ${f}_{S, \dilipwo} \in \dilipFFF$.
 Then if for every $f \in {\diliphf}^{\omega}$ there is an $S \in \dilipFFF$ such that $S: \omega \rightarrow \diliphf$ is a slalom and $\dilipexistsinf k \in \omega \left[f(k) \in S(k)\right]$, then $\forall f \in {\diliphf}^{\omega} \exists g \in \dilipFFF \dilipexistsinf k \in \omega \left[f(k) = g(k)\right]$.
\end{dilipLemma}
\begin{proof}
 Let $f: \omega \rightarrow \diliphf$ be given.
 Define the IP $I = \dilipseq{i}{n}{\in}{\omega}$ as follows: ${i}_{0} = 0$; given ${i}_{n} \in \omega$, ${i}_{n+1} = {i}_{n} + (n+1){2}^{n+1}$.
 Define $F: \omega \rightarrow \diliphf$ by setting $F(k) = f \diliprestrict {I}_{k}$.
 By hypothesis, find a slalom $S \in \dilipFFF$ such that $\dilipexistsinf k \in \omega \left[F(k) \in S(k)\right]$.
 Consider any $k \in \omega$ such that $F(k) \in S(k)$.
 There exists a unique $j < \diliplc S(k) \diliprc \leq (k+1){2}^{k+1}$ with ${e}_{S(k), \dilipwo}(j) = F(k)$.
 Let $x = {i}_{k} + j$.
 Then ${i}_{k} \leq x < {i}_{k+1}$ and by definition ${f}_{S, \dilipwo}(x) = {e}_{S(k), \dilipwo}(j)(x) = F(k)(x) = f(x)$.
 Thus for each $k \in \omega$ such that $F(k) \in S(k)$, there exists ${x}_{k} \in {I}_{k}$ such that ${f}_{S, \dilipwo}({x}_{k}) = f({x}_{k})$.
 Since there are infinitely many such $k$ and since ${I}_{k} \cap {I}_{k'} = \emptyset$ whenever $k \neq k'$, $\dilipexistsinf x \in \omega \left[{f}_{S, \dilipwo}(x) = f(x)\right]$.
 We are done because ${f}_{S, \dilipwo} \in \dilipFFF$ by hypothesis.
\end{proof}
\begin{dilipDef} \label{def:wellclosed}
 Let $\dilipwo$ be fixed.
 A family $\dilipFFF \subseteq {\diliphf}^{\omega}$ is called \emph{well-closed w.r.t.\@ $\dilipwo$} if it satisfies the following:
 \begin{enumerate}
  \item
  $\forall f \in \dilipFFF \left[{S}_{f} \in \dilipFFF\right]$;
  \item
  for every $S \in \dilipFFF$, if $S: \omega \rightarrow \diliphf$ is a slalom, then ${f}_{S, \dilipwo} \in \dilipFFF$;
  \item
  for every set $\{{f}_{n}: n \in \omega\} \subseteq \dilipFFF$, there is a slalom $S \in \dilipFFF$ so that $\forall n \in \omega \dilipforallbutfin k \in \omega \left[{f}_{n}(k) \in S(k)\right]$.
 \end{enumerate}
\end{dilipDef}
\begin{dilipremark} \label{rem:absolute}
 Suppose ${\dilipV}_{0} \subseteq {\dilipV}_{1}$ are transitive models of a sufficiently large fragment of $\dilipZFC$ with $\dilipwo \in {\dilipV}_{0}$.
 Then for any $A \in \diliphf$, the map ${e}_{A, \dilipwo}: \diliplc A \diliprc \rightarrow A$ is the same whether it is calculated in ${\dilipV}_{0}$ or in ${\dilipV}_{1}$.
 Also, if $S \in {\dilipV}_{0}$, then $S$ is a slalom in ${\dilipV}_{0}$ if and only if $S$ is a slalom in ${\dilipV}_{1}$, and the computation of ${f}_{S, \dilipwo}$ does not change.
 Similarly, for any $f \in {\diliphf}^{\omega} \cap {\dilipV}_{0}$, ${S}_{f}$ is the same when calculated in ${\dilipV}_{0}$ or ${\dilipV}_{1}$.
 So we conclude that if $\dilipFFF \in {\dilipV}_{0}$ and
 \begin{align*}
  {\left( \dilipFFF \subseteq {\diliphf}^{\omega} \ \text{satisfies (1) and (2) of Definition \ref{def:wellclosed} w.r.t.\@} \ \dilipwo \right)}^{{\dilipV}_{0}}\text{, then}
 \end{align*}
 \begin{align*}
  {\left( \dilipFFF \subseteq {\diliphf}^{\omega} \ \text{satisfies (1) and (2) of Definition \ref{def:wellclosed} w.r.t.\@} \ \dilipwo \right)}^{{\dilipV}_{1}}.
 \end{align*}
 If the models ${\dilipV}_{0} \subseteq {\dilipV}_{1}$ satisfy the further condition that
 \begin{align*}
  \forall X \in {\dilipV}_{0} \forall Y \in {\left( \dilippc{X}{\leq {\aleph}_{0}} \right)}^{{\dilipV}_{1}} \exists Z \in {\left( \dilippc{X}{\leq {\aleph}_{0}} \right)}^{{\dilipV}_{0}} \left[Y \subseteq Z\right],
 \end{align*}
 then for any $\dilipFFF \in {\dilipV}_{0}$, if
 \begin{align*}
  {\left( \dilipFFF \ \text{is well-closed w.r.t.\@} \ \dilipwo \right)}^{{\dilipV}_{0}}\text{, then}
 \end{align*}
 \begin{align*}
  {\left( \dilipFFF \ \text{is well-closed w.r.t.\@} \ \dilipwo \right)}^{{\dilipV}_{1}}.
 \end{align*}
 In particular, this holds whenever ${\dilipV}_{1}$ is a forcing extension of ${\dilipV}_{0}$ by a proper poset.
\end{dilipremark}
Remark \ref{rem:absolute} will used several times in what follows.
\begin{dilipDef} \label{def:big}
 $\dilipFFF \subseteq {\diliphf}^{\omega}$ is called \emph{big} if
 \begin{align*}
  \forall f \in {\diliphf}^{\omega} \exists g \in \dilipFFF \dilipexistsinf k \in \omega \left[f(k) = g(k) \right].
 \end{align*}
\end{dilipDef}
\begin{dilipLemma} \label{lem:twostepbig}
 Suppose $\dilipFFF \subseteq {\diliphf}^{\omega}$ is big.
 Let $\dilipPP$ be a forcing such that ${\dilipforces}_{\dilipPP} \; {\dilipFFF \ \text{is big}}$.
 Let $\mathring{\dilipQQ}$ be a $\dilipPP$-name for a forcing such that ${\dilipforces}_{\dilipPP} \; {``{\dilipforces}_{\mathring{\dilipQQ}} \; {\dilipFFF \ \text{is big}}''}$.
 Then ${\dilipforces}_{\dilipPP \ast \mathring{\dilipQQ}} \; {\dilipFFF \ \text{is big}}$.
\end{dilipLemma}
\begin{proof}
 Let $K$ be $(\dilipV, \dilipPP \ast \mathring{\dilipQQ})$-generic.
 Then there are $G$ and $H$ such that $G$ is $(\dilipV, \dilipPP)$-generic, $H$ is $(\dilipVG, \mathring{\dilipQQ}\left[G\right])$-generic and $\dilipV\left[K\right] = \dilipVG\left[H\right]$.
 By hypothesis, in $\dilipVG$ we have that $\dilipFFF$ is big, $\mathring{\dilipQQ}\left[G\right]$ is a forcing, and ${\dilipforces}_{\mathring{\dilipQQ}\left[G\right]} \; {\dilipFFF \ \text{is big}}$.
 Therefore in $\dilipV\left[K\right] = \dilipVG\left[H\right]$, $\dilipFFF$ is big.
\end{proof}
The next lemma is a variation of Theorem 61 of Raghavan~\cite{svmad}.
See also Raghavan~\cite{mythesis}.
\begin{dilipLemma} \label{lem:iterationbig}
 Let $\dilipwo$ be a well-ordering of $\diliphf$.
 Suppose $\dilipFFF \subseteq {\diliphf}^{\omega}$ is well-closed w.r.t.\@ $\dilipwo$ and big.
 Let $\gamma$ be a limit ordinal and let $\langle {\dilipPP}_{\alpha}; {\mathring{\dilipQQ}}_{\alpha}: \alpha \leq \gamma \rangle$ be a CS iteration such that $\forall \alpha < \gamma \left[{\dilipforces}_{\alpha} \; {{\mathring{\dilipQQ}}_{\alpha} \ \text{is proper}}\right]$.
 Suppose that for all $\alpha < \gamma$, ${\dilipforces}_{\alpha} \; {\dilipFFF \ \text{is big}}$.
 Then ${\dilipforces}_{\gamma} \; {\dilipFFF \ \text{is big}}$.
\end{dilipLemma}
\begin{proof}
 Let $\mathring{f} \in {\dilipV}^{{\dilipPP}_{\gamma}}$ be such that ${\dilipforces}_{\gamma} \; {\mathring{f} \in {\diliphf}^{\omega}}$ and let ${p}_{0} \in {\dilipPP}_{\gamma}$.
 Fix a sufficiently large regular $\theta$ and a countable $M \prec H(\theta)$ with $\dilipFFF, \langle {\dilipPP}_{\alpha}; {\mathring{\dilipQQ}}_{\alpha}: \alpha \leq \gamma \rangle, \mathring{f}, {p}_{0} \in M$.
 Since $\dilipFFF$ satisfies (3) of Definition \ref{def:wellclosed}, we can find a slalom $S \in \dilipFFF$ so that for all $f \in \dilipFFF \cap M$, $\dilipforallbutfin k \in \omega \left[f(k) \in S(k)\right]$.
 We will find a $q \in {\dilipPP}_{\gamma}$ such that $q \; {\dilipforces}_{\gamma} \; {p}_{0} \in {\mathring{G}}_{\gamma}$ and $q \; {\dilipforces}_{\gamma} \; {\dilipexistsinf n \in \omega \left[\mathring{f}(n) \in S(n)\right]}$.
 Put $\gamma' = \sup(M \cap \gamma)$ and let $\dilipseq{\gamma}{n}{\in}{\omega} \subseteq M \cap \gamma$ be an increasing sequence that is cofinal in $\gamma'$.
 We build two sequences $\dilipseq{q}{n}{\in}{\omega}$ and $\dilipseq{\mathring{p}}{n}{\in}{\omega}$ satisfying the following conditions for all $n \in \omega$:
 \begin{enumerate}[series=iterationbig]
  \item
  ${q}_{n} \in {\dilipPP}_{{\gamma}_{n}}$, ${q}_{n}$ is $(M, {\dilipPP}_{{\gamma}_{n}})$-generic, and ${q}_{n+1} \diliprestrict {\gamma}_{n} = {q}_{n}$;
  \item
  ${\mathring{p}}_{0} = {\check{p}}_{0}$, ${\mathring{p}}_{n} \in {\dilipV}^{{\dilipPP}_{{\gamma}_{n}}}$, and ${q}_{n} \; {\dilipforces}_{{\gamma}_{n}} \; {``{\mathring{p}}_{n} \in M \cap {\dilipPP}_{\gamma} \wedge {\mathring{p}}_{n} \diliprestrict {\gamma}_{n} \in {\mathring{G}}_{{\gamma}_{n}}''}$;
  \item
  ${q}_{n+1} \; {\dilipforces}_{{\gamma}_{n+1}} \; {{\mathring{p}}_{n+1} \leq {\mathring{p}}_{n}}$;
  \item
  ${q}_{n+1} \; {\dilipforces}_{{\gamma}_{n+1}} \; {``{\mathring{p}}_{n+1} \; {\dilipforces}_{{\dilipPP}_{\gamma} \slash {\mathring{G}}_{{\gamma}_{n+1}}} \; {\exists k \geq n \left[\mathring{f}\left[{\mathring{G}}_{{\gamma}_{n+1}}\right](k) \in S(k)\right]} ''}$.
 \end{enumerate}
 Assume for a moment that such sequences have been constructed.
 Then ${\bigcup}_{n \in \omega}{{q}_{n}}$ is a condition in ${\dilipPP}_{\gamma'}$.
 We extend ${\bigcup}_{n \in \omega}{{q}_{n}}$ to a condition $q$ in ${\dilipPP}_{\gamma}$ by setting $q(\alpha) = {\mathbbm{1}}_{{\mathring{\dilipQQ}}_{\alpha}}$ for all $\gamma' \leq \alpha < \gamma$.
 By standard arguments, $\forall n \in \omega \left[q \; {\dilipforces}_{\gamma} \; {{\mathring{p}}_{n} \in {\mathring{G}}_{\gamma}}\right]$.
 In particular, $q \; {\dilipforces}_{\gamma} \; {{p}_{0} \in {\mathring{G}}_{\gamma}}$.
 We will check that $q \; {\dilipforces}_{\gamma} \; {\dilipexistsinf n  \in \omega \left[\mathring{f}(n) \in S(n)\right]}$.
 
 Suppose not.
 Then $\exists r \leq q \exists n \in \omega \left[r \; {\dilipforces}_{\gamma} \; {\forall k \geq n \left[\mathring{f}(k) \notin S(k)\right]}\right]$.
 Let ${G}_{\gamma}$ be $(\dilipV, {\dilipPP}_{\gamma})$-generic with $r \in {G}_{\gamma}$.
 It is a standard fact that ${G}_{{\gamma}_{n+1}}$ is $(\dilipV, {\dilipPP}_{{\gamma}_{n+1}})$-generic, that ${G}_{\gamma}$ is $\left( \dilipV\left[{G}_{{\gamma}_{n+1}}\right], {\dilipPP}_{\gamma} \slash {G}_{{\gamma}_{n+1}} \right)$-generic, and that in $\dilipV\left[{G}_{\gamma}\right]$: ${G}_{\gamma} = {G}_{{\gamma}_{n+1}} \ast {G}_{\gamma}$.
 So in $\dilipV\left[{G}_{{\gamma}_{n+1}}\right]$: ${\mathring{p}}_{n+1}\left[{G}_{{\gamma}_{n+1}}\right] \in {\dilipPP}_{\gamma} \slash {G}_{{\gamma}_{n+1}}$ and
 \begin{align*}
  {\mathring{p}}_{n+1}\left[{G}_{\gamma}\right] = {\mathring{p}}_{n+1}\left[{G}_{{\gamma}_{n+1}}\right] \; {\dilipforces}_{{\dilipPP}_{\gamma} \slash {G}_{{\gamma}_{n+1}}} \; {\exists k \geq n+1\left[\mathring{f}\left[{G}_{{\gamma}_{n+1}}\right](k) \in S(k)\right]}.
 \end{align*}
 Since ${\mathring{p}}_{n+1}\left[{G}_{\gamma}\right] \in {G}_{\gamma}$, in $\dilipV\left[{G}_{\gamma}\right] = \dilipV\left[{G}_{{\gamma}_{n+1}}\right]\left[{G}_{\gamma}\right]$: there exists $k \geq n+1$ such that $\mathring{f}\left[{G}_{{\gamma}_{n+1}}\right]\left[{G}_{\gamma}\right](k) \in S(k)$.
 Since $\mathring{f}\left[{G}_{{\gamma}_{n+1}}\right]\left[{G}_{\gamma}\right] = \mathring{f}\left[{G}_{\gamma}\right]$, so $\mathring{f}\left[{G}_{\gamma}\right](k) \in S(k)$, and we have a contradiction.
 
 To build $\dilipseq{q}{n}{\in}{\omega}$ and $\dilipseq{\mathring{p}}{n}{\in}{\omega}$, proceed by induction on $n$.
 Let ${\mathring{p}}_{0} = {\check{p}}_{0}$, the canonical ${\dilipPP}_{{\gamma}_{0}}$-name for ${p}_{0}$.
 Note that ${p}_{0} \diliprestrict {\gamma}_{0} \in M \cap {\dilipPP}_{{\gamma}_{0}}$ and ${\dilipPP}_{{\gamma}_{0}} \in M$.
 By the properness of ${\dilipPP}_{{\gamma}_{0}}$, there is ${q}_{0} \in {\dilipPP}_{{\gamma}_{0}}$ which is $(M, {\dilipPP}_{{\gamma}_{0}})$-generic with ${q}_{0} \leq {p}_{0} \diliprestrict {\gamma}_{0}$.
 Then ${\mathring{p}}_{0}$ and ${q}_{0}$ satisfy (1)--(4).
 Now suppose ${\mathring{p}}_{n}$ and ${q}_{n}$ are given, for some $n \in \omega$.
 By the Properness Extension Lemma (see, for example, Lemma 2.8 of Abraham~\cite{Ab}), there is an $(M, {\dilipPP}_{{\gamma}_{n+1}})$-generic condition ${q}_{n+1} \in {\dilipPP}_{{\gamma}_{n+1}}$ such that ${q}_{n+1} \diliprestrict {\gamma}_{n} = {q}_{n}$ and ${q}_{n+1} \; {\dilipforces}_{{\gamma}_{n+1}} \; {{\mathring{p}}_{n} \diliprestrict {\gamma}_{n+1} \in {\mathring{G}}_{{\gamma}_{n+1}}}$.
 
 To find ${\mathring{p}}_{n+1}$ we proceed as follows.
 Fix a $(\dilipV, {\dilipPP}_{{\gamma}_{n+1}})$-generic filter ${G}_{{\gamma}_{n+1}}$ with ${q}_{n+1} \in {G}_{{\gamma}_{n+1}}$.
 Recall that $M\left[{G}_{{\gamma}_{n+1}}\right] \prec H(\theta)\left[{G}_{{\gamma}_{n+1}}\right] = {H}^{\dilipV\left[{G}_{{\gamma}_{n+1}}\right]}(\theta)$.
 Note that ${\dilipPP}_{\gamma} \slash {G}_{{\gamma}_{n+1}} \in M\left[{G}_{{\gamma}_{n+1}}\right]$ and that in $\dilipV\left[{G}_{{\gamma}_{n+1}}\right]$, ${\dilipforces}_{{\dilipPP}_{\gamma} \slash {G}_{{\gamma}_{n+1}}} \; {\mathring{f}\left[{G}_{{\gamma}_{n+1}}\right] \in {\diliphf}^{\omega}}$.
 Since ${q}_{n+1} \in {G}_{{\gamma}_{n+1}}$, ${\mathring{p}}_{n}\left[{G}_{{\gamma}_{n+1}}\right] \diliprestrict {\gamma}_{n+1} \in {G}_{{\gamma}_{n+1}}$.
 So in $\dilipV\left[{G}_{{\gamma}_{n+1}}\right]$, ${\mathring{p}}_{n}\left[{G}_{{\gamma}_{n+1}}\right] \in M \cap {\dilipPP}_{\gamma} \slash {G}_{{\gamma}_{n+1}}$.
 By elementarity, find a sequence $\langle {p}^{i}: i \in \omega \rangle \in M\left[{G}_{{\gamma}_{n+1}}\right]$ and a function $f \in {\diliphf}^{\omega} \cap M\left[{G}_{{\gamma}_{n+1}}\right]$ satisfying the following for each $i \in \omega$:
 \begin{enumerate}[resume=iterationbig]
  \item
  ${p}^{i} \in {\dilipPP}_{\gamma} \slash {G}_{{\gamma}_{n+1}}$;
  \item
  ${p}^{i+1} \leq {p}^{i} \leq {\mathring{p}}_{n}\left[{G}_{{\gamma}_{n+1}}\right]$;
  \item
  ${p}^{i} \; {\dilipforces}_{{\dilipPP}_{\gamma} \slash {G}_{{\gamma}_{n+1}}} \; {\mathring{f}\left[{G}_{{\gamma}_{n+1}}\right](i) = f(i)}$.
 \end{enumerate}
 By hypothesis, $\dilipFFF$ is big in $\dilipV\left[{G}_{{\gamma}_{n+1}}\right]$.
 So by elementarity and by the fact that $f \in M\left[{G}_{{\gamma}_{n+1}}\right]$, we can find $g \in \dilipFFF \cap M\left[{G}_{{\gamma}_{n+1}}\right]$ so that $\dilipexistsinf k \in \omega\left[f(k) = g(k)\right]$.
 Since ${q}_{n+1}$ is an $(M, {\dilipPP}_{{\gamma}_{n+1}})$-generic condition with ${q}_{n+1} \in {G}_{{\gamma}_{n+1}}$ and since $\dilipFFF \in \dilipV$, $\dilipFFF \cap M = \dilipFFF \cap M\left[{G}_{{\gamma}_{n+1}}\right]$.
 So $g \in \dilipFFF \cap M$.
 By the choice of $S$, $\dilipforallbutfin k \in \omega\left[g(k) \in S(k)\right]$.
 Hence we can fix $k \geq n+1$ such that $f(k) \in S(k)$.
 Note that ${p}^{k} \in {\dilipPP}_{\gamma} \cap M\left[{G}_{{\gamma}_{n+1}}\right]$, and that ${\dilipPP}_{\gamma} \cap M\left[{G}_{{\gamma}_{n+1}}\right] = {\dilipPP}_{\gamma} \cap M$ similarly to $\dilipFFF$.
 Thus ${p}^{k}$ has the following properties:
 \begin{enumerate}[resume=iterationbig]
  \item
  ${p}^{k} \in M \cap {\dilipPP}_{\gamma}$ and ${p}^{k} \diliprestrict {\gamma}_{n+1} \in {G}_{{\gamma}_{n+1}}$;
  \item
  ${p}^{k} \leq {\mathring{p}}_{n}\left[{G}_{{\gamma}_{n+1}}\right]$;
  \item
  ${p}^{k} \; {\dilipforces}_{{\dilipPP}_{\gamma} \slash {G}_{{\gamma}_{n+1}}} \; {\exists i \geq n+1\left[\mathring{f}\left[{G}_{{\gamma}_{n+1}}\right](i) \in S(i)\right]}$.
 \end{enumerate}
 Since ${G}_{{\gamma}_{n+1}}$ was an arbitrary $(\dilipV, {\dilipPP}_{{\gamma}_{n+1}})$-generic filter with ${q}_{n+1} \in {G}_{{\gamma}_{n+1}}$, we can use the maximal principle in $\dilipV$ to find a ${\mathring{p}}_{n+1} \in {\dilipV}^{{\dilipPP}_{{\gamma}_{n+1}}}$ so that in $\dilipV$:
 \begin{enumerate}[resume=iterationbig]
  \item
  ${q}_{n+1} \; {\dilipforces}_{{\dilipPP}_{{\gamma}_{n+1}}} \; {``{\mathring{p}}_{n+1} \in M \cap {\dilipPP}_{\gamma} \wedge {\mathring{p}}_{n+1} \diliprestrict {\gamma}_{n+1} \in {\mathring{G}}_{{\gamma}_{n+1}}''}$;
  \item
  ${q}_{n+1} \; {\dilipforces}_{{\dilipPP}_{{\gamma}_{n+1}}} \; {{\mathring{p}}_{n+1} \leq {\mathring{p}}_{n}}$;
  \item
  ${q}_{n+1} \; {\dilipforces}_{{\dilipPP}_{{\gamma}_{n+1}}} \; {``{\mathring{p}}_{n+1} \; {\dilipforces}_{{\dilipPP}_{\gamma} \slash {\mathring{G}}_{{\gamma}_{n+1}}} \; {\exists k \geq n+1\left[\mathring{f}\left[{\mathring{G}}_{{\gamma}_{n+1}}\right](k) \in S(k)\right]}''}$.
 \end{enumerate}
 This concludes the inductive construction.
 
 To complete the proof, fix an arbitrary $(\dilipV, {\dilipPP}_{\gamma})$-generic filter ${G}_{\gamma}$.
 In $\dilipV\left[{G}_{\gamma}\right]$, by what has been proved above, for every $f \in {\diliphf}^{\omega}$, there exists a slalom $S \in \dilipFFF$ so that $\dilipexistsinf k \in \omega\left[f(k) \in S(k)\right]$.
 Since ${\dilipPP}_{\gamma}$ is proper and since $\dilipFFF$ is well-closed w.r.t.\@ $\dilipwo$ in $\dilipV$, $\dilipFFF$ is still well-closed w.r.t.\@ $\dilipwo$ in $\dilipV\left[{G}_{\gamma}\right]$.
 So Lemma \ref{lem:cumber} applies in $\dilipV\left[{G}_{\gamma}\right]$ and implies that $\forall f \in {\diliphf}^{\omega} \exists g \in \dilipFFF \dilipexistsinf k \in \omega \left[f(k) = g(k)\right]$.
\end{proof}
\begin{dilipCor} \label{cor:bigispreserved}
 Let $\dilipwo$ be a well-ordering of $\diliphf$.
 Suppose $\dilipFFF \subseteq {\diliphf}^{\omega}$ is well-closed w.r.t.\@ $\dilipwo$ and is big.
 Let $\gamma$ be any ordinal.
 Suppose $\langle {\dilipPP}_{\alpha}; {\mathring{\dilipQQ}}_{\alpha}: \alpha \leq \gamma \rangle$ is a CS iteration such that $\forall \alpha < \gamma\left[{\dilipforces}_{\alpha} \; {{\mathring{\dilipQQ}}_{\alpha} \ \text{is proper}}\right]$.
 Suppose also that for each $\alpha < \gamma$ ${\dilipforces}_{\alpha} \; {``{\dilipforces}_{{\mathring{\dilipQQ}}_{\alpha}} \; {\dilipFFF \ \text{is big}}''}$.
 Then ${\dilipforces}_{\gamma} \; {\dilipFFF \ \text{is big}}$.
\end{dilipCor}
\begin{proof}
 The proof is by induction on $\gamma$.
 If $\gamma = 0$, then ${\dilipPP}_{\gamma}$ is the trivial forcing.
 By hypothesis, in $\dilipV$, $\dilipFFF$ is big, so there is nothing to do.
 Suppose $\gamma = \gamma' + 1$ and that the statement is true for $\gamma'$.
 So ${\dilipforces}_{\gamma'} \; {\dilipFFF \ \text{is big}}$.
 Now ${\dilipPP}_{\gamma'+1}$ is forcing equivalent to ${\dilipPP}_{\gamma'} \ast {\mathring{\dilipQQ}}_{\gamma'}$.
 By hypothesis, ${\dilipforces}_{{\dilipPP}_{\gamma'}} \; {``{\dilipforces}_{{\mathring{\dilipQQ}}_{\gamma'}} \; {\dilipFFF \ \text{is big}}''}$.
 By Lemma \ref{lem:twostepbig}, we have ${\dilipforces}_{{\dilipPP}_{\gamma'} \ast {\mathring{\dilipQQ}}_{\gamma'}} \; {\dilipFFF \ \text{is big}}$.
 So ${\dilipforces}_{\gamma} \; {\dilipFFF \ \text{is big}}$ as required.
 Finally suppose that $\gamma$ is a limit ordinal and that the statement is true for all $\beta < \gamma$.
 So $\forall \beta < \gamma \left[{\dilipforces}_{\beta} \; {\dilipFFF \ \text{is big}} \right]$.
 By Lemma \ref{lem:iterationbig}, ${\dilipforces}_{\gamma} \; {\dilipFFF \ \text{is big}}$.
 This concludes the induction and the proof.
\end{proof}
\begin{dilipLemma} \label{lem:cohenbig}
 Suppose $\dilipwo$ is a well-ordering of $\diliphf$.
 If $\dilipFFF \subseteq {\diliphf}^{\omega}$ is well-closed w.r.t.\@ $\dilipwo$ and is big, then ${\dilipforces}_{\dilipCC} \; {\dilipFFF \ \text{is big}}$.
\end{dilipLemma}
\begin{proof}
 Let $\mathring{f}$ be a $\dilipCC$-name and suppose that ${\dilipforces}_{\dilipCC} \; {\mathring{f}: \omega \rightarrow \diliphf}$.
 Enumerate $\dilipCC$ as $\left\{ {p}_{n}: n \in \omega \right\}$ in a 1-1 way.
 Define a small slalom $S: \omega \rightarrow \diliphf$ as follows.
 Given $n \in \omega$, find ${q}_{n} \leq {p}_{n}$ and ${x}_{n} \in \diliphf$ so that ${q}_{n} \; {\dilipforces}_{\dilipCC} \; {\mathring{f}(n) = {x}_{n}}$, and define $S(n) = \{{x}_{n}\}$.
 Note that for each $l \in \omega$, ${X}_{l} \in \dilippc{\omega}{\omega}$, where ${X}_{l} = \{ n \in \omega: {p}_{n} \leq {p}_{l} \}$.
 Since $\dilipFFF$ is big, Lemma \ref{lem:slalombig} applies and implies that there is a slalom $T: \omega \rightarrow \diliphf$ so that $T \in \dilipFFF$ and $\forall l \in \omega \dilipexistsinf n \in {X}_{l}\left[S(n) \subseteq T(n)\right]$.
 Now we check that ${\dilipforces}_{\dilipCC} \; {\dilipexistsinf n \in \omega \left[\mathring{f}(n) \in T(n)\right]}$.
 To see this, fix $p \in \dilipCC$ and $N \in \omega$.
 Then $p = {p}_{l}$ for some $l \in \omega$.
 Find $n \in {X}_{l}$ with $n > N$ and $S(n) \subseteq T(n)$.
 By the definition of $S$ and ${X}_{l}$, ${q}_{n} \leq {p}_{n} \leq {p}_{l}$ and ${q}_{n} \; {\dilipforces}_{\dilipCC} \; {\mathring{f}(n) = {x}_{n} \in \{{x}_{n}\} = S(n) \subseteq T(n)}$.
 This is as required.
 
 Now to complete the proof, fix a $(\dilipV, \dilipCC)$-generic filter $G$.
 As $\dilipCC$ is proper, $\dilipFFF$ is well-closed w.r.t.\@ $\dilipwo$ in $\dilipVG$.
 By what has been proved in the previous paragraph, we have in $\dilipVG$ that whenever $f: \omega \rightarrow \diliphf$, then there is a slalom $T \in \dilipFFF$ with the property that $\dilipexistsinf n \in \omega\left[f(n) \in T(n)\right]$.
 As $\dilipFFF$ is well-closed w.r.t.\@ $\dilipwo$ in $\dilipVG$, Lemma \ref{lem:cumber} tells us that $\dilipFFF$ is big in $\dilipVG$.
\end{proof}
\begin{dilipCor} \label{cor:cohenbig}
 Suppose $\dilipwo$ is a well-ordering of $\diliphf$.
 If $\dilipFFF \subseteq {\diliphf}^{\omega}$ is well-closed w.r.t.\@ $\dilipwo$ and is big, then ${\dilipforces}_{{\dilipCC}_{{\omega}_{1}}} \; {\dilipFFF \ \text{is big}}$.
\end{dilipCor}
\begin{proof}
 Suppose $G$ is $\left( \dilipV, {\dilipCC}_{{\omega}_{1}} \right)$-generic.
 Consider any $f \in {\diliphf}^{\omega} \cap \dilipVG$.
 It is well-known that for some $(\dilipV, \dilipCC)$-generic filter $H$, $f \in \dilipV\left[H\right]$.
 By Lemma \ref{lem:cohenbig}, $\dilipFFF$ is big in $\dilipV\left[H\right]$.
 So there is $g \in \dilipFFF$ so that $\dilipexistsinf n \in \omega\left[f(n) = g(n)\right]$.
 Therefore, $\dilipFFF$ is big in $\dilipVG$.
\end{proof}
\section{The partial order} \label{sec:partial}
In $\dilipV$, let $\dilipA \subseteq \dilippc{\omega}{\omega}$ be a fixed block Shelah-Stepr{\= a}ns a.d.\@ family.
Let ${\dilipV}_{{\omega}_{1}}$ denote the extension of $\dilipV$ by ${\omega}_{1}$ Cohen reals (i.e.\@ by ${\dilipCC}_{{\omega}_{1}}$).
We assume that $\dilipA$ remains block Shelah-Stepr{\= a}ns in ${\dilipV}_{{\omega}_{1}}$.
Actually, Shelah-Stepr{\= a}ns a.d.\@ families remain Shelah-Stepr{\= a}ns after adding Cohen reals, but this does not seem to hold in general for block Shelah-Stepr{\= a}ns a.d.\@ families.
However, since we are mainly interested in the minimal cardinality of Shelah-Stepr{\= a}ns and block Shelah-Stepr{\= a}ns a.d.\@ families, this is not relevant for our main result in the next section.
\begin{dilipDef} \label{def:less}
 If $F$ and $G$ are non-empty sets of ordinals, we write $F < G$ as an abbreviation for $\forall x \in F \forall y \in G\left[x < y\right]$.
 
 We use $\dilipFIN$ to denote $\dilippc{\omega}{< \omega} \setminus \{\emptyset\}$.
\end{dilipDef}
\begin{dilipLemma} \label{lem:CSS}
 Shelah-Stepr{\= a}ns a.d.\@ families remain Shelah-Stepr{\= a}ns after adding Cohen reals.
\end{dilipLemma}
\begin{proof}
 Let $\dilipB$ be a Shelah-Stepr{\= a}ns a.d.\@ family.
 It suffices to show that
 \begin{align*}
  {\dilipforces}_{\dilipCC} \; {``\dilipB \ \text{is Shelah-Stepr{\= a}ns}''}.
 \end{align*}
 To this end, let $\mathring{X}$ be a $\dilipCC$-name such that ${\dilipforces}_{\dilipCC} \; {\mathring{X} \subseteq \dilipFIN}$ and
 \begin{align*}
  {\dilipforces}_{\dilipCC} \; {``\forall B \in \dilipIII(\dilipB) \exists s \in \mathring{X}\left[s \cap B = \emptyset\right]''}.
 \end{align*}
 Let $\dilipseq{p}{n}{\in}{\omega}$ enumerate $\dilipCC$.
 For each $n \in \omega$, let
 \begin{align*}
  {X}_{n} = \left\{s \in \dilipFIN: \exists q \leq {p}_{n}\left[q \; {\dilipforces}_{\dilipCC} \; {s \in \mathring{X}}\right] \right\}.
 \end{align*}
 By hypothesis, for any $n \in \omega$ and any $B \in \dilipIII(\dilipB)$, there exists $q \leq {p}_{n}$ and $s \in \dilipFIN$ such that $s \cap B = \emptyset$ and $q \; {\dilipforces}_{\dilipCC} \; {s \in \mathring{X}}$, whence $s \in {X}_{n}$.
 Since $\dilipB$ is a Shelah-Stepr{\= a}ns a.d.\@ family, there exists $B \in \dilipIII(\dilipB)$ with the property that $\forall n \in \omega \left[\dilippc{B}{< {\aleph}_{0}} \cap {X}_{n} \ \text{is infinite}\right]$.
 We check that ${\dilipforces}_{\dilipCC} \; {``\dilippc{B}{< {\aleph}_{0}} \cap \mathring{X} \ \text{is infinite}''}$.
 To this end, fix $p \in \dilipCC$ and some $k \in \omega$.
 Then $p={p}_{n}$ for some $n \in \omega$, and there exists $s \in \dilippc{B}{< {\aleph}_{0}} \cap {X}_{n}$ such that $s \notin \dilipPset(k)$.
 By definition of ${X}_{n}$, $q \; {\dilipforces}_{\dilipCC} \; {s \in \mathring{X}}$, for some $q \leq {p}_{n}=p$.
 Thus we have proved that for each $p \in \dilipCC$ and $k \in \omega$, there exist $s \in \dilippc{B}{< {\aleph}_{0}}$ and $q \leq p$ such that $s \notin \dilipPset(k)$ and $q \; {\dilipforces}_{\dilipCC} \; {s \in \mathring{X}}$, which proves that ${\dilipforces}_{\dilipCC} \; {``\dilippc{B}{< {\aleph}_{0}} \cap \mathring{X} \ \text{is infinite}''}$.
\end{proof}
We work in ${\dilipV}_{{\omega}_{1}}$ for the remainder of this section unless the contrary is explicitly stated.
As stated above, $\dilipA \subseteq \dilippc{\omega}{\omega}$ is a fixed block Shelah-Stepr{\= a}ns a.d.\@ family which is a member of $\dilipV$ and remains block Shelah-Stepr{\= a}ns in ${\dilipV}_{{\omega}_{1}}$.
\begin{dilipDef} \label{def:po}
 Working in ${\dilipV}_{{\omega}_{1}}$, define a forcing $\dilipPP(\dilipA)$ as follows.
 A pair $p=\dilippr{{s}_{p}}{{c}_{p}}$ belongs to $\dilipPP(\dilipA)$ if and only if:
 \begin{enumerate}[series=po]
  \item
  ${s}_{p} \in \dilippc{\omega}{< \omega}$;
  \item
  ${c}_{p}: \omega \rightarrow \dilipFIN$ such that:
  \begin{enumerate}
   \item[(2a)]
   $\forall i \in \omega\left[{c}_{p}(i) < {c}_{p}(i+1)\right]$;
   \item[(2b)]
   $\forall x \in {s}_{p} \forall y \in {c}_{p}(0)\left[x < y\right]$;
  \end{enumerate}
  \item
  $\forall B \in \dilipIII(\dilipA) \exists i \in \omega \left[B \cap {c}_{p}(i) = \emptyset\right]$.
 \end{enumerate}
 For $q, p \in \dilipPP(\dilipA)$, $q \leq p$ if and only if:
 \begin{enumerate}[resume=po]
  \item
  ${s}_{q} \supseteq {s}_{p}$;
  \item
  $\exists {F}_{q, p} \in \dilippc{\omega}{< \omega}\left[{s}_{q} \setminus {s}_{p} = {\bigcup}_{i \in {F}_{q, p}}{{c}_{p}(i)}\right]$;
  \item
  $\forall i \in \omega \exists {G}_{q, p, i} \in \dilippc{\omega}{< \omega}\left[{c}_{q}(i) = {\bigcup}_{j \in {G}_{q, p, i}}{{c}_{p}(j)}\right]$.
 \end{enumerate}
 For $n \in \omega$ and $q, p \in \dilipPP(\dilipA)$, $q \; {\leq}_{n} \; p$ if and only if $\left( q \leq p \ \text{and} \ {c}_{q} \diliprestrict n = {c}_{p} \diliprestrict n \right)$.
 It is easy to see that $\leq$ and ${\leq}_{n}$ are transitive.
\end{dilipDef}
The following lemma lists some elementary consequences of the definitions which are easy to verify.
\begin{dilipLemma} \label{lem:basic}
 Let $p, q \in \dilipPP(\dilipA)$.
 The following hold:
 \begin{enumerate}
  \item
  $q \; {\leq}_{0} \; p \iff q \leq p$;
  \item
  for all $n > 0$, if $q \; {\leq}_{n} \; p$, then ${s}_{q} = {s}_{p}$;
  \item
  if $q \leq p$, then for all $i < j < \omega$, ${G}_{q, p, i} < {G}_{q, p, j}$;
  \item
  $\forall n < m < \omega\left[q \; {\leq}_{m} \; p \implies q \; {\leq}_{n} \; p\right]$.
 \end{enumerate}
\end{dilipLemma}
\begin{dilipLemma} \label{lem:pnonempty}
 $\dilipPP(\dilipA)$ is non-empty.
\end{dilipLemma}
\begin{proof}
 Working in $\dilipV$, define a partial order $\dilipSS$ as follows.
 $s \in \dilipSS$ if and only if $s: {n}_{s} \rightarrow \dilipFIN$, where ${n}_{s} \in \omega$ and $\forall i < i+1 < {n}_{s}\left[s(i) < s(i+1)\right]$.
 For $t, s \in \dilipSS$, $t \leq s$ if and only if $t \supseteq s$.
 It is easy to see that for each $i \in \omega$, ${D}_{i} = \{t \in \dilipSS: i < {n}_{t}\}$, and for each $B \in {\left( \dilipIII(\dilipA) \right)}^{\dilipV}$, ${E}_{B} = \{t \in \dilipSS: \exists i < {n}_{t}\left[B \cap t(i) = \emptyset\right]\}$ are dense subsets of $\dilippr{\dilipSS}{\leq}$ belonging to $\dilipV$.
 In ${\dilipV}_{{\omega}_{1}}$, there exists a $(\dilipV, \dilipSS)$-generic filter $H$.
 Setting $c = \bigcup{H}$, it is clear that $\dilippr{\emptyset}{c}$ is a condition in $\dilipPP(\dilipA)$.
\end{proof}
The relations ${\leq}_{n}$ do not define an Axiom A structure on $\dilipPP(\dilipA)$ because the limit of a fusion sequence will not, in general, satisfy clause (3) of Definition \ref{def:po}.
However, the next lemma says that clause (3) of Definition \ref{def:po} is the only obstruction.
The proof is straightforward and left to the reader.
\begin{dilipLemma} \label{lem:fusion}
 Suppose $\dilipseq{p}{n}{\in}{\omega}$ is a sequence of members of $\dilipPP(\dilipA)$ so that $\forall n \in \omega\left[{p}_{n+1} \; {\leq}_{n} \; {p}_{n}\right]$.
 Define $p = \dilippr{{s}_{p}}{{c}_{p}}$ by setting ${s}_{p} = {s}_{{p}_{1}}$ and ${c}_{p}(n) = {c}_{{p}_{n+1}}(n)$, for all $n \in \omega$.
 Then clauses (1) and (2) of Definition \ref{def:po} are satisfied.
 If clause (3) of Definition \ref{def:po} is also satisfied, then $p \in \dilipPP(\dilipA)$ and $\forall n \in \omega\left[p \; {\leq}_{n} \; {p}_{n}\right]$.
\end{dilipLemma}
\begin{dilipDef} \label{def:pkF}
 Let $p \in \dilipPP(\dilipA)$ and suppose $k \in \omega$ and $F \subseteq k$.
 Define $p(k, F) = \dilippr{{s}_{p(k, F)}}{{c}_{p(k, F)}}$ by ${s}_{p(k, F)} = {s}_{p} \cup \left( {\bigcup}_{j \in F}{{c}_{p}(j)} \right)$ and ${c}_{p(k, F)}(i) = {c}_{p}(i+k)$, for all $i \in \omega$.
 It is easy to check that $p(k, F) \in \dilipPP(\dilipA)$ and that $p(k, F) \leq p$.
\end{dilipDef}
\begin{dilipDef} \label{def:qpkF}
 Let $p \in \dilipPP(\dilipA)$, $k \in \omega$, and $F \subseteq k$.
 Note that $p(k+1, F\cup\{k\}) \leq p$.
 Suppose that some $q \leq p(k+1, F\cup\{k\}) \leq p$ is given.
 Define $q(p, k, F) = \dilippr{{s}_{q(p, k, F)}}{{c}_{q(p, k, F)}}$ as follows.
 Put ${s}_{q(p, k, F)} = {s}_{p}$.
 For $i < k$, define ${c}_{q(p, k, F)}(i) = {c}_{p}(i)$, define ${c}_{q(p, k, F)}(k) = \bigcup\left\{ {c}_{p}(j): j \in {F}_{q, p} \setminus k \right\}$, and for $i > k$, define ${c}_{q(p, k, F)}(i) = {c}_{q}(i-k-1)$.
 
 It is easy to check that $q(p, k, F) \in \dilipPP(\dilipA)$ and $q(p, k, F) \; {\leq}_{k} \; p$.
\end{dilipDef}
\begin{dilipLemma} \label{lem:diag}
 Let $p \in \dilipPP(\dilipA)$, $k \in \omega$, and $F \subseteq k$.
 Suppose $q \leq p(k+1, F\cup\{k\})$.
 If $r \leq q(p, k, F)$ and $i \in \omega$ are such that $k = \min\left({G}_{r, q(p, k, F), i}\right)$ and ${F}_{r, q(p, k, F)} = F$, then $r(i+1, \{i\}) \leq q$.
\end{dilipLemma}
\begin{proof}
 Note that ${c}_{r}(i) = \bigcup {\left\{ {c}_{q(p, k, F)}(j): j \in {G}_{r, q(p, k, F), i} \right\}}$.
 Let
 \begin{align*}
  J = \left\{ j \in {G}_{r, q(p, k, F), i}: j > k \right\}.
 \end{align*}
 Then
 \begin{align*}
  {s}_{r(i+1, \{i\})} = {s}_{q} \cup \bigcup {\left\{ {c}_{q(p, k, F)}(j): j \in J \right\}} = {s}_{q} \cup \bigcup {\left\{ {c}_{q}(j-k-1): j \in J \right\}}.
 \end{align*} 
 This confirms both (4) and (5) of Definition \ref{def:po}.
 For $l \in \omega$, we have ${c}_{r(i+1, \{i\})}(l) = {c}_{r}(l+i+1) =$
 \begin{align*}
  \bigcup {\left\{ {c}_{q(p, k, F)}(j): j \in {G}_{r, q(p, k, F), l+i+1} \right\}} = \bigcup {\left\{ {c}_{q}(j-k-1): j \in {G}_{r, q(p, k, F), l+i+1} \right\}}.
 \end{align*}
 This confirms (6) of Definition \ref{def:po}.
 Therefore $r(i+1, \{i\}) \leq q$.
\end{proof}
\begin{dilipLemma} \label{lem:fusion1}
 Let $p \in \dilipPP(\dilipA)$ and $k \in \omega$.
 Suppose $\mathring{x} \in {\dilipV}^{\dilipPP(\dilipA)}_{{\omega}_{1}}$ and ${\dilipforces}_{\dilipPP(\dilipA)}{\mathring{x} \in {\dilipV}_{{\omega}_{1}}}$.
 Then there are $q \; {\leq}_{k} \; p$ and $X \in {\dilipV}_{{\omega}_{1}}$ so that:
 \begin{enumerate}
  \item
  $\diliplc X \diliprc \leq {2}^{k}$, ${s}_{p} = {s}_{q}$, and ${c}_{q}(k) \supseteq {c}_{p}(k)$;
  \item
  for any $q' \; {\leq}_{k+1} \; q$, $F \subseteq k$, $r \leq q'$, and $i \in \omega$, if ${F}_{r, q'} = F$ and $k = \min\left( {G}_{r, q', i} \right)$, then $r(i+1, \{i\}) \; {\dilipforces} \; {\mathring{x} \in X}$.
 \end{enumerate}
\end{dilipLemma}
\begin{proof}
 Let $\dilipseq{F}{l}{<}{{2}^{k}}$ enumerate $\dilipPset(k)$.
 We define by induction two sequences $\dilipseq{q}{l}{\leq}{{2}^{k}}$ and $\dilipseq{x}{l}{<}{{2}^{k}}$ satisfying the following: ${q}_{0} = p$ and
 \begin{align*}
  \forall l' < l \left[{q}_{l} \; {\leq}_{k} \; {q}_{l'}\text{, } {s}_{{q}_{l}} = {s}_{{q}_{l'}}\text{, and } {c}_{{q}_{l}}(k) \supseteq {c}_{{q}_{l'}}(k)\right].
 \end{align*}
 Define ${q}_{0} = p$ and note that the induction hypothesis is vacuously satisfied.
 Now suppose $l < {2}^{k}$ and that ${q}_{l}$ satisfying the induction hypothesis is given.
 Then ${q}_{l}(k+1, {F}_{l} \cup \{k\}) \leq {q}_{l}$.
 Find ${r}_{l} \leq {q}_{l}(k+1, {F}_{l} \cup \{k\})$ and ${x}_{l} \in {\dilipV}_{{\omega}_{1}}$ with ${r}_{l} \; {\dilipforces}_{\dilipPP(\dilipA)} \; {\mathring{x} = {x}_{l}}$, and define ${q}_{l+1} = {r}_{l}({q}_{l}, k, {F}_{l})$.
 Note that ${q}_{l+1} \; {\leq}_{k} \; {q}_{l}$ and ${s}_{{q}_{l+1}} = {s}_{{r}_{l}({q}_{l}, k, {F}_{l})} = {s}_{{q}_{l}}$.
 Further, ${c}_{{q}_{l}}(k) \subseteq {s}_{{q}_{l}(k+1, {F}_{l} \cup \{k\})} \setminus {s}_{{q}_{l}} \subseteq {s}_{{r}_{l}} \setminus {s}_{{q}_{l}}$, and $k \in {F}_{{r}_{l}, {q}_{l}} \setminus k$.
 Thus ${c}_{{q}_{l}}(k) \subseteq {c}_{{q}_{l+1}}(k)$.
 Hence by the induction hypothesis,
 \begin{align*}
  \forall l' \leq l \left[{q}_{l+1} \; {\leq}_{k} \; {q}_{l} \; {\leq}_{k} \; {q}_{l'}\text{, } {s}_{{q}_{l+1}} = {s}_{{q}_{l}} = {s}_{{q}_{l'}}\text{, and } {c}_{{q}_{l+1}}(k) \supseteq {c}_{{q}_{l}}(k) \supseteq {c}_{{q}_{l'}}(k)\right].
 \end{align*}
 This concludes the inductive construction.
 
 Now define $q = {q}_{{2}^{k}}$ and $X = \{{x}_{l}: l < {2}^{k}\}$.
 Then (1) is satisfied by construction.
 To verify (2), fix any $q' \; {\leq}_{k+1} \; q$, $F \subseteq k$, $r \leq q'$, and $i \in \omega$, and assume that ${F}_{r, q'} = F$ and that $k = \min\left( {G}_{r, q', i} \right)$.
 Then $F = {F}_{l}$, for some $l < {2}^{k}$.
 Since $q \; {\leq}_{k} \; {q}_{l+1}$, ${s}_{q} = {s}_{{q}_{l+1}}$, and ${c}_{q}(k) \supseteq {c}_{{q}_{l+1}}(k)$, it follows that ${F}_{r, {q}_{l+1}} = F$ and that $\min\left( {G}_{r, {q}_{l+1}, i} \right) = k$.
 By definition, ${r}_{l} \leq {q}_{l}(k+1, F \cup \{k\})$ and ${q}_{l+1} = {r}_{l}({q}_{l}, k, F)$.
 Therefore, by Lemma \ref{lem:diag}, $r(i+1, \{i\}) \leq {r}_{l}$.
 Therefore, $r(i+1, \{i\}) \; {\dilipforces}_{\dilipPP(\dilipA)} \; {\mathring{x} = {x}_{l} \in X}$.
\end{proof}
\begin{dilipLemma} \label{lem:main}
 Work in ${\dilipV}_{{\omega}_{1}}$.
 Let $\theta$ be a sufficiently large regular cardinal.
 Suppose $M \prec H(\theta)$ is countable with $M$ containing all relevant parameters.
 Let $f: \omega \rightarrow M$ be such that
 \begin{align*}
  \forall k \in \omega \left[f(k) \in {\dilipV}^{\dilipPP(\dilipA)}_{{\omega}_{1}} \wedge {\dilipforces}_{\dilipPP(\dilipA)} \; {f(k) \in {\dilipV}_{{\omega}_{1}}}\right].
 \end{align*}
 For any $p \in \dilipPP(\dilipA) \cap M$, there exist $q$ and $S$ satisfying the following:
 \begin{enumerate}[series=mainlem]
  \item
  $q \leq p$;
  \item
  $S$ is a function, $\dilipdom(S) = \omega$, and $\forall k \in \omega\left[S(k) \subseteq M \wedge \diliplc S(k) \diliprc \leq {2}^{k} \right]$;
  \item
  for any $k \in \omega$, any $F \subseteq k$, any $r \leq q$, and any $i \in \omega$, if ${F}_{r, q} = F$ and $\min\left( {G}_{r, q, i} \right) = k$, then $r(i+1, \{i\}) \; {\dilipforces}_{\dilipPP(\dilipA)} \; {f(k) \in S(k)}$.
 \end{enumerate}
\end{dilipLemma}
\begin{proof}
 Let $\bar{M}$ denote the transitive collapse of $M$.
 Let $\pi: M \rightarrow \bar{M}$ be the collapsing map and ${\pi}^{\ast}: \bar{M} \rightarrow M$ be the inverse of $\pi$.
 Say that $a$ is \emph{an approximation} if:
 \begin{enumerate}[resume=mainlem]
  \item
  $a \in {\left( \dilipPP(\dilipA) \cap M \right)}^{< \omega} \times {\bar{M}}^{< \omega}$;
  \item
  writing $a = \dilippr{{\sigma}_{a}}{{\tau}_{a}}$, $\dilipdom({\sigma}_{a}) = \dilipdom({\tau}_{a})+1$;
  \item
  ${\sigma}_{a}(0) = p$ and $\forall n < n+1 < \dilipdom({\sigma}_{a})\left[{\sigma}_{a}(n+1) \; {\leq}_{n} \; {\sigma}_{a}(n) \right]$;
  \item
  for all $k < \dilipdom({\tau}_{a})$, $\diliplc {\tau}_{a}(k) \diliprc \leq {2}^{k}$;
  \item
  for any $k+1 < \dilipdom({\sigma}_{a})$, any $q' \; {\leq}_{k+1} \; {\sigma}_{a}(k+1)$, $F \subseteq k$, $r \leq q'$, and $i \in \omega$, if ${F}_{r, q'} = F$ and $k = \min\left( {G}_{r, q', i} \right)$, then
  \begin{align*}
   r(i+1, \{i\}) \; {\dilipforces}_{\dilipPP(\dilipA)} \; { f(k) \in \left\{ {\pi}^{\ast}(x): x \in {\tau}_{a}(k) \right\} }.
  \end{align*}
 \end{enumerate}
 Let $\dilipAA = \{a: a \ \text{is an approximation}\}$.
 Define a partial order on $\dilipAA$ by setting $\leq = \{\dilippr{b}{a} \in \dilipAA \times \dilipAA: {\sigma}_{b} \supseteq {\sigma}_{a} \wedge {\tau}_{b} \supseteq {\tau}_{a}\}$.
 It is easy to see that $\dilippr{\dilipAA}{\leq} \in H({\aleph}_{1})$.
 
 For each $B \in {\dilipIII}^{\dilipV}(\dilipA)$, define
 \begin{align*}
  D(B) = \left\{ b \in \dilipAA: \exists l \in \omega \left[ l+1 \in \dilipdom({\sigma}_{b}) \wedge {c}_{{\sigma}_{b}(l+1)}(l) \cap B = \emptyset \right] \right\}.
 \end{align*}
 \begin{dilipClaim} \label{claim:main1}
  $D(B)$ is dense in $\dilippr{\dilipAA}{\leq}$.
 \end{dilipClaim}
 \begin{proof}
  Let $a \in \dilipAA$ be given and put $k = \dilipdom({\tau}_{a})$.
  Let $p' = {\sigma}_{a}(k)$ and $\mathring{x} = f(k)$.
  Note that $\mathring{x} \in M \cap {\dilipV}^{\dilipPP(\dilipA)}_{{\omega}_{1}}$ and that ${\dilipforces}_{\dilipPP(\dilipA)} \; {\mathring{x} \in {\dilipV}_{{\omega}_{1}}}$.
  Note also that $p' \in M$.
  Let $Z = \{ z \in \omega: z \geq k \wedge {c}_{p'}(z) \cap B = \emptyset \}$.
  By the definition of $\dilipPP(\dilipA)$, $Z \in \dilippc{\omega}{\omega}$.
  Let $\dilipseq{z}{l}{\in}{\omega}$ be the strictly increasing enumeration of $Z$.
  Define $p'' = \dilippr{{s}_{p''}}{{c}_{p''}}$ by setting ${s}_{p''} = {s}_{p'}$, ${c}_{p''}(i) = {c}_{p'}(i)$, for all $i < k$, and ${c}_{p''}(k+l) = {c}_{p'}({z}_{l})$, for all $l \in \omega$.
  Then it is clear that $p'' \in \dilipPP(\dilipA)$ and that $p'' \; {\leq}_{k} \; p'$.
  Applying Lemma \ref{lem:fusion1} with $p''$, $k$, and $\mathring{x}$, find ${q}^{\ast} \; {\leq}_{k} \; p'' \; {\leq}_{k} \; p'$ and ${X}^{\ast} \in {\dilipV}_{{\omega}_{1}}$ satisfying (1) and (2) of Lemma \ref{lem:fusion1} with respect to $p''$, $k$, and $\mathring{x}$.
  Notice that ${c}_{{q}^{\ast}}(k) \cap B = \emptyset$.
  $p''$ may not be in $M$, and so ${q}^{\ast}$ and ${X}^{\ast}$ may not be in $M$ either.
  However, $p', \mathring{x}, {c}_{{q}^{\ast}}(k) \in M$.
  Therefore by elementarity and by the fact that $M$ contains all the relevant parameters, there exist $q, X \in M$ satisfying the following properties:
  \begin{enumerate}[resume=mainlem]
   \item
   $q \; {\leq}_{k} \; p'$ and $\diliplc X \diliprc \leq {2}^{k}$;
   \item
   for any $q' \; {\leq}_{k+1} \; q$, $F \subseteq k$, $r \leq q'$, and $i \in \omega$, if ${F}_{r, q'} = F$ and $k = \min\left( {G}_{r, q', i} \right)$, then $r(i+1, \{i\}) \; {\dilipforces}_{\dilipPP(\dilipA)} \; {\mathring{x} \in X}$;
   \item
   ${c}_{q}(k) = {c}_{{q}^{\ast}}(k)$.
  \end{enumerate}
  Define ${\sigma}_{b} = {\sigma}_{a} \cup \left\{ \dilippr{k+1}{q} \right\}$.
  Next, $X \subseteq M$, and so $\{\pi(x): x \in X\} \subseteq \bar{M}$ and $\diliplc \{\pi(x): x \in X\} \diliprc = \diliplc X \diliprc \leq {2}^{k}$.
  Since $\bar{M}$ is a transitive model of a sufficiently large fragment of $\dilipZFC - \mathrm{P}$ and since $\{\pi(x): x \in X\}$ is a finite subset of $\bar{M}$, $\{\pi(x): x \in X\} \in \bar{M}$.
  Define ${\tau}_{b} = {\tau}_{a} \cup \left\{ \dilippr{k}{\{\pi(x): x \in X\}} \right\}$.
  Observe that $\left\{ {\pi}^{\ast}(y): y \in {\tau}_{b}(k) \right\} = X$.
  Using these observations and (10), it is easy to verify that $b = \dilippr{{\sigma}_{b}}{{\tau}_{b}}$ satisfies (4)--(8).
  Therefore $b \in \dilipAA$ and $b \leq a$.
  Finally, ${c}_{{\sigma}_{b}(k+1)}(k) \cap B = {c}_{{q}^{\ast}}(k) \cap B = \emptyset$ because of (11).
  Therefore $l=k$ witnesses that $b \in D(B)$, establishing the density of $D(B)$.
 \end{proof}
 Now note that $\dilipAA$ is non-empty because $\dilippr{\left\{ \dilippr{0}{p} \right\}}{\emptyset} \in \dilipAA$.
 As we have observed earlier, $\dilippr{\dilipAA}{\leq} \in H({\aleph}_{1})$.
 Hence there exists $\delta < {\omega}_{1}$ such that $\dilippr{\dilipAA}{\leq} \in {\dilipV}_{\delta}$, where ${\dilipV}_{\delta}$ is the extension of $\dilipV$ by the first $\delta$ Cohen reals.
 Further, for $B \in {\dilipIII}^{\dilipV}(\dilipA)$, $D(B) \in {\dilipV}_{\delta}$ because $D(B)$ has an absolute definition in terms of the parameters $\dilipAA$ and $B$.
 Since $\dilippr{\dilipAA}{\leq}$ is a countable forcing notion, there exists $H \in {\dilipV}_{{\omega}_{1}}$ which is $\left( {\dilipV}_{\delta}, \dilipAA \right)$-generic.
 Define $P = \bigcup\left\{ {\sigma}_{a}: a \in H \right\}$ and $T = \bigcup\left\{ {\tau}_{a}: a \in H \right\}$.
 Then $P$ and $T$ are functions and $\dilipdom(P) = \sup\left\{ \dilipdom({\sigma}_{a}): a \in H \right\}$.
 \begin{dilipClaim}\label{claim:main2}
  $\dilipdom(P) = \omega$.
 \end{dilipClaim}
 \begin{proof}
  Suppose not.
  Then $\dilipdom(P) < \omega$.
  Let $B = \bigcup\left\{ {c}_{P(l+1)}(l): l+1 \in \dilipdom(P) \right\}$.
  Since $\dilipdom(P)$ is finite, $B$ is a finite subset of $\omega$, whence $B \in {\dilipIII}^{\dilipV}(\dilipA)$.
  Consider $a \in H \cap D(B)$.
  By definition of $D(B)$, there exists $l+1 \in \dilipdom({\sigma}_{a})$ such that ${c}_{{\sigma}_{a}(l+1)}(l) \cap B = \emptyset$.
  It follows that $l+1 \in \dilipdom(P)$ and $B \supseteq {c}_{P(l+1)}(l) = {c}_{{\sigma}_{a}(l+1)}(l)$, whence ${c}_{{\sigma}_{a}(l+1)}(l) = \emptyset$.
  But this is impossible by Clause (2) of Definition \ref{def:po}.
 \end{proof}
 For $n \in \omega$, define ${p}_{n} = P(n)$.
 Then ${p}_{0} = p$, $\forall n \in \omega\left[{p}_{n} \in \dilipPP(\dilipA)\right]$, and $\forall n \in \omega\left[{p}_{n+1} \; {\leq}_{n} \; {p}_{n}\right]$.
 Define $q' = \dilippr{{s}_{q'}}{{c}_{q'}}$ by setting ${s}_{q'} = {s}_{{p}_{1}}$ and ${c}_{q'}(n) = {c}_{{p}_{n+1}}(n)$, for all $n \in \omega$.
 Now if $B \in {\dilipIII}^{\dilipV}(\dilipA)$, then for some $b \in H$ and for some $l \in \omega$, ${c}_{{p}_{l+1}}(l) \cap B = \emptyset$, whence ${c}_{q'}(l) \cap B = \emptyset$.
 Thus by Lemma \ref{lem:fusion}, $q' \in \dilipPP(\dilipA)$ and $\forall n \in \omega\left[q' \; {\leq}_{n} \; {p}_{n}\right]$.
 In particular, $q' \; {\leq}_{0} \; {p}_{0} = p$, which is to say that $q' \leq p$.
 Next for each $k \in \omega$, $T(k) \in \bar{M}$ and $\diliplc T(k) \diliprc \leq {2}^{k}$.
 $T(k) \subseteq \bar{M}$ because $\bar{M}$ is transitive.
 Defining $S(k) = \{{\pi}^{\ast}(y): y \in T(k)\} \subseteq M$, $\diliplc S(k) \diliprc = \diliplc \{{\pi}^{\ast}(y): y \in T(k)\} \diliprc = \diliplc T(k) \diliprc \leq {2}^{k}$.
 Thus (1) and (2) are satisfied by $q'$ and $S$.
 To see that (3) is satisfied, fix $k \in \omega$, $F \subseteq k$, $r \leq q'$, and $i \in \omega$.
 Assume that ${F}_{r, q'} = F$ and $\min\left(  {G}_{r, q', i} \right) = k$.
 Note that $q' \; {\leq}_{k+1} \; {p}_{k+1}$ and that ${p}_{k+1} = {\sigma}_{a}(k+1)$, for some $a \in H$.
 Hence by (8),
 \begin{align*}
  r(i+1, \{i\}) \; {\dilipforces}_{\dilipPP(\dilipA)} \; {f(k) \in \left\{ {\pi}^{\ast}(x): x \in {\tau}_{a}(k) \right\} = \left\{ {\pi}^{\ast}(x): x \in T(k) \right\} = S(k)}.
 \end{align*}
 This concludes the proof.
\end{proof}
We remark that even though the filter $H$ and the function $T$ belong to an intermediate extension of the form ${\dilipV}_{\delta}$, for some $\delta < {\omega}_{1}$, this is not the case for $S$.
This is because the functions $\pi$ and ${\pi}^{\ast}$ are not in any intermediate ${\dilipV}_{\delta}$ with $\delta < {\omega}_{1}$.
\begin{dilipCor} \label{cor:proper}
 $\dilipPP(\dilipA)$ is proper.
\end{dilipCor}
\begin{proof}
 Working in ${\dilipV}_{{\omega}_{1}}$, fix a countable $M \prec H(\theta)$, where $\theta$ is a sufficiently large regular cardinal and $M$ contains the relevant parameters.
 Let $\dilipseq{\mathring{\alpha}}{n}{\in}{\omega}$ be an enumeration of all $\mathring{\alpha} \in M$ such that $\mathring{\alpha}$ is a $\dilipPP(\dilipA)$-name and ${\dilipforces}_{\dilipPP(\dilipA)} \; {``\mathring{\alpha} \ \text{is an ordinal}''}$.
 Define a function $f: \omega \rightarrow M$ as follows.
 Suppose $k \in \omega$.
 Consider a $\left( {\dilipV}_{{\omega}_{1}}, \dilipPP(\dilipA) \right)$-generic filter $G$.
 In ${\dilipV}_{{\omega}_{1}}\left[G\right]$, $\{{\mathring{\alpha}}_{0}\left[G\right], \dotsc, {\mathring{\alpha}}_{k}\left[G\right]\}$ is a finite set of ordinals and so $\{{\mathring{\alpha}}_{0}\left[G\right], \dotsc, {\mathring{\alpha}}_{k}\left[G\right]\} \in {\dilipV}_{{\omega}_{1}}$.
 Applying the maximal principal in ${\dilipV}_{{\omega}_{1}}$, there is a $\dilipPP(\dilipA)$-name $\mathring{x}$ such that ${\dilipforces}_{\dilipPP(\dilipA)} \; {\mathring{x} \in {\dilipV}_{{\omega}_{1}}}$ and
 \begin{align*}
  {\dilipforces}_{\dilipPP(\dilipA)} \; {``\forall y\left[y \in \mathring{x} \leftrightarrow \left( y = {\mathring{\alpha}}_{0} \vee \dotsb \vee y = {\mathring{\alpha}}_{k} \right)\right]''}.
 \end{align*}
 Since ${\mathring{\alpha}}_{0}, \dotsc, {\mathring{\alpha}}_{k} \in M$, we may choose such an $\mathring{x} \in M$.
 Define $f(k) = \mathring{x} \in M$.
 Notice that ${\dilipforces}_{\dilipPP(\dilipA)} \; {``f(k) \ \text{is finite}''}$ and that ${\dilipforces}_{\dilipPP(\dilipA)} \; {{\mathring{\alpha}}_{l} \in f(k)}$, for every $l \leq k$.
 
 Unfix $k$ from the previous paragraph.
 Fix any $p \in \dilipPP(\dilipA) \cap M$.
 We must find $q \leq p$ which is $(M, \dilipPP(\dilipA))$-generic.
 Applying Lemma \ref{lem:main}, find $q$ and $S$ satisfying (1)--(3) of Lemma \ref{lem:main}.
 We argue that $q$ is $(M, \dilipPP(\dilipA))$-generic.
 To this end, it suffices to show that for any $\mathring{\alpha} \in M$, if ${\dilipforces}_{\dilipPP(\dilipA)} \; {``\mathring{\alpha} \ \text{is an ordinal}''}$, then $q \; {\dilipforces}_{\dilipPP(\dilipA)} \; {\mathring{\alpha} \in M}$.
 Let a relevant $\mathring{\alpha}$ be given.
 Then $\mathring{\alpha} = {\mathring{\alpha}}_{l}$, for some $l < \omega$.
 Suppose $r \leq q$.
 Let $F = {F}_{r, q}$.
 Since $\left\langle \min\left( {G}_{r, q, i} \right): i \in \omega \right\rangle$ is a strictly increasing sequence, it is possible to find $k, i \in \omega$ such that $k = \min\left( {G}_{r, q, i} \right)$, $F \subseteq k$, and $k \geq l$.
 By (3) of Lemma \ref{lem:main}, $r(i+1, \{i\}) \; {\dilipforces}_{\dilipPP(\dilipA)} \; {f(k) \in S(k) \subseteq M}$.
 Find $r' \leq r(i+1, \{i\}) \leq r$ and $X \in M$ with $r' \; {\dilipforces}_{\dilipPP(\dilipA)} \; {f(k)=X}$.
 Since we know ${\dilipforces}_{\dilipPP(\dilipA)} \; {``f(k) \ \text{is finite}''}$, it follows that $X$ is finite, and since $X \in M$, $X \subseteq M$.
 Thus $r' \; {\dilipforces}_{\dilipPP(\dilipA)} \; {{\mathring{\alpha}}_{l} \in f(k)=X \subseteq M}$, whence $r' \; {\dilipforces}_{\dilipPP(\dilipA)} \; {{\mathring{\alpha}}_{l} \in M}$.
 Thus we have proved that $\forall r \leq q \exists r' \leq r\left[r' \; {\dilipforces}_{\dilipPP(\dilipA)} \; {{\mathring{\alpha}}_{l} \in M}\right]$, whence $q \; {\dilipforces}_{\dilipPP(\dilipA)} \; {{\mathring{\alpha}}_{l} \in M}$.
 This proves that $\dilipPP(\dilipA)$ is proper.
\end{proof}
\begin{dilipLemma} \label{lem:preservingbig}
 In $\dilipV$, suppose $\dilipwo$ is a well-ordering of $\diliphf$, and that $\dilipFFF \subseteq {\diliphf}^{\omega}$ is well-closed w.r.t.\@ $\dilipwo$ and is big.
 Then in ${\dilipV}_{{\omega}_{1}}$, ${\dilipforces}_{\dilipPP(\dilipA)} \; {\dilipFFF \ \text{is big}}$.
\end{dilipLemma}
\begin{proof}
 Work in ${\dilipV}_{{\omega}_{1}}$.
 $\dilipFFF$ is well-closed w.r.t.\@ $\dilipwo$ because ${\dilipCC}_{{\omega}_{1}}$ is proper and $\dilipFFF$ is big by Lemma \ref{cor:cohenbig}.
 Fix a $\dilipPP(\dilipA)$ name $\mathring{g}$ and assume that ${\dilipforces}_{\dilipPP(\dilipA)} \; {\mathring{g}: \omega \rightarrow \diliphf}$.
 Let $p \in \dilipPP(\dilipA)$ be fixed.
 Let $\theta$ be a sufficiently large regular cardinal.
 Suppose $M \prec H(\theta)$ is countable with $M$ containing all the relevant parameters.
 In particular, $\mathring{g}, p \in M$.
 Define $f: \omega \rightarrow M$ as follows.
 For each $k \in \omega$, find a $\dilipPP(\dilipA)$-name $\mathring{x} \in M$ such that ${\dilipforces}_{\dilipPP(\dilipA)} \; {\mathring{x} \in \diliphf}$ and ${\dilipforces}_{\dilipPP(\dilipA)} \; {\mathring{g}(k)=\mathring{x}}$, and define $f(k) = \mathring{x}$.
 Applying Lemma \ref{lem:main}, find $q$ and $S$ satisfying (1)--(3) of Lemma \ref{lem:main}.
 Note that for each $k \in \omega$, $\diliplc S(k) \cap \diliphf \diliprc \leq \diliplc S(k) \diliprc \leq {2}^{k}$, and so $S(k) \cap \diliphf$ is a finite subset of $\diliphf$, which implies that $S(k) \cap \diliphf \in \diliphf$.
 Hence we may define a small slalom ${S}^{\ast}: \omega \rightarrow \diliphf$ by ${S}^{\ast}(k) = S(k) \cap \diliphf$, for all $k \in \omega$.
 
 Next, define sequences $\dilipseq{\dilipA}{l}{\in}{\omega} \subseteq \dilippc{\dilipA}{< {\aleph}_{0}}$ and $\dilipseq{Y}{l}{\in}{\omega} \subseteq \dilippc{\omega}{\omega}$ as follows.
 Fix $l \in \omega$ and suppose that $\dilipseq{\dilipA}{l'}{<}{l} \subseteq \dilippc{\dilipA}{< {\aleph}_{0}}$ and $\dilipseq{Y}{l'}{<}{l} \subseteq \dilippc{\omega}{\omega}$ are already given.
 By the definition of $\dilipPP(\dilipA)$, the family
 \begin{align*}
  \dilipEEE = \left\{s \in \dilipran({c}_{q}): s \cap \left( \bigcup{\left( {\bigcup}_{l' < l}{{\dilipA}_{l'}} \right)} \right) = \emptyset \right\}
 \end{align*}
 has the property that $\forall B \in \dilipIII(\dilipA) \exists s \in \dilipEEE\left[s \cap B = \emptyset\right]$.
 Since $\dilipA$ is assumed to be block Shelah-Stepr{\= a}ns in ${\dilipV}_{{\omega}_{1}}$, there exists $B \in \dilipIII(\dilipA)$ with the property that $\{y \in \omega: {c}_{q}(y) \in \dilipEEE \wedge {c}_{q}(y) \subseteq B\}$ is infinite.
 Clearly for every $F \in \dilippc{\omega}{< \omega}$, the set $\{y \in \omega: {c}_{q}(y) \in \dilipEEE \wedge {c}_{q}(y) \subseteq B \setminus F\}$ is still infinite.
 Therefore, there exists ${\dilipA}^{\ast} \in \dilippc{\dilipA}{< {\aleph}_{0}}$ such that $\{y \in \omega: {c}_{q}(y) \in \dilipEEE \wedge {c}_{q}(y) \subseteq \bigcup{{\dilipA}^{\ast}}\}$ is infinite.
 Let ${\dilipA}_{l}$ be such an ${\dilipA}^{\ast}$ of minimal cardinality.
 Define ${Y}_{l} = \{y \in \omega: {c}_{q}(y) \in \dilipEEE \wedge {c}_{q}(y) \subseteq \bigcup{{\dilipA}_{l}}\} \in \dilippc{\omega}{\omega}$.
 Suppose that $A \in {\dilipA}_{l'} \cap {\dilipA}_{l}$, for some $l' < l$.
 Then for each $y \in {Y}_{l}$, ${c}_{q}(y) \cap A = \emptyset$ and ${c}_{q}(y) \subseteq \bigcup{{\dilipA}_{l}}$, whence ${c}_{q}(y) \subseteq \bigcup{\left( {\dilipA}_{l} \setminus \{A\} \right)}$, contradicting the minimality of ${\dilipA}_{l}$.
 Therefore $\forall l' < l\left[{\dilipA}_{l'} \cap {\dilipA}_{l} = \emptyset \right]$.
 This concludes the definition of $\dilipseq{\dilipA}{l}{\in}{\omega}$ and $\dilipseq{Y}{l}{\in}{\omega}$.
 
 Since $\dilipFFF$ is big and is well-closed w.r.t.\@ $\dilipwo$ in ${\dilipV}_{{\omega}_{1}}$, Lemma \ref{lem:slalombig} applies and implies that there exists $T \in \dilipFFF$ such that $T: \omega \rightarrow \diliphf$ is a slalom and $\forall l \in \omega \dilipexistsinf y \in {Y}_{l}\left[{S}^{\ast}(y) \subseteq T(y)\right]$.
 Define ${Z}_{l} = \{ y \in {Y}_{l}: {S}^{\ast}(y) \subseteq T(y) \} \in \dilippc{\omega}{\omega}$, for all $l \in \omega$.
 Let $Z = {\bigcup}_{l \in \omega}{{Z}_{l}}$ and let $\dilipseq{z}{j}{<}{\omega}$ be the strictly increasing enumeration of $Z$.
 Define ${q}^{\ast} = \dilippr{{s}_{{q}^{\ast}}}{{c}_{{q}^{\ast}}}$ by ${s}_{{q}^{\ast}} = {s}_{q}$ and ${c}_{{q}^{\ast}}(j) = {c}_{q}({z}_{j})$, for all $j \in \omega$.
 Consider any $B \in {\dilipIII}^{\dilipV}(\dilipA)$.
 Since $\forall l' < l < \omega\left[{\dilipA}_{l'} \cap {\dilipA}_{l} = \emptyset\right]$, there exists $l, m \in \omega$ such that $B \cap \left( \bigcup{{\dilipA}_{l}} \right) \subseteq m$.
 As ${Z}_{l}$ is infinite, we can find $y \in {Z}_{l}$ with $\min({c}_{q}(y)) \geq m$.
 As $y \in {Y}_{l}$, ${c}_{q}(y) \subseteq \bigcup{{\dilipA}_{l}}$, whence $B \cap {c}_{q}(y) = \emptyset$.
 Now $y = {z}_{j}$ for some $j < \omega$ because $y \in Z$.
 Therefore, ${c}_{{q}^{\ast}}(j) \cap B = {c}_{q}({z}_{j}) \cap B = {c}_{q}(y) \cap B = \emptyset$.
 This shows that ${q}^{\ast} \in \dilipPP(\dilipA)$ and that ${q}^{\ast} \leq q$.
 
 We will now argue that ${q}^{\ast} \; {\dilipforces}_{\dilipPP(\dilipA)} \; {\dilipexistsinf k \in \omega \left[\mathring{g}(k) \in T(k)\right]}$.
 To this end fix $r \leq {q}^{\ast}$ and some $l \in \omega$.
 As $r \leq q$, $\left\langle \min\left( {G}_{r, q, i} \right): i < \omega \right\rangle$ is a strictly increasing sequence.
 Let ${F}_{r, q} = F$.
 Find $k, i \in \omega$ such that $k = \min\left( {G}_{r, q, i} \right)$, $F \subseteq k$, and $k \geq l$.
 It follows from the fact that $r \leq {q}^{\ast} \leq q$ and from the definition of ${q}^{\ast}$ that $k \in Z$.
 In particular, ${S}^{\ast}(k) \subseteq T(k)$.
 By (3) of Lemma \ref{lem:main},
 \begin{align*}
  r(i+1, \{i\}) \; {\dilipforces}_{\dilipPP(\dilipA)} \; {\mathring{g}(k) = f(k) \in S(k) \cap \diliphf = {S}^{\ast}(k) \subseteq T(k)},
 \end{align*}
 whence $r(i+1, \{i\}) \; {\dilipforces}_{\dilipPP(\dilipA)} \; {\mathring{g}(k) \in T(k)}$.
 Since $r(i+1, \{i\}) \leq r$, we have proved that for each $r \leq {q}^{\ast}$ and $l \in \omega$, there exists $k \geq l$ and $r' \leq r$ with $r' \; {\dilipforces}_{\dilipPP(\dilipA)} \; {\mathring{g}(k) \in T(k)}$.
 This proves that ${q}^{\ast} \; {\dilipforces}_{\dilipPP(\dilipA)} \; {\dilipexistsinf k \in \omega\left[\mathring{g}(k) \in T(k)\right]}$.
 Since ${q}^{\ast} \leq q \leq p$, we have now proved that for any $\mathring{g} \in {\dilipV}^{\dilipPP(\dilipA)}_{{\omega}_{1}}$ and any $p \in \dilipPP(\dilipA)$, if ${\dilipforces}_{\dilipPP(\dilipA)} \; {\mathring{g}: \omega \rightarrow \diliphf}$, then there exist ${q}^{\ast} \leq p$ and $T \in \dilipFFF$ such that $T: \omega \rightarrow \diliphf$ is a slalom and ${q}^{\ast} \; {\dilipforces}_{\dilipPP(\dilipA)} \; {\dilipexistsinf k \in \omega\left[\mathring{g}(k) \in T(k)\right]}$.
 
 To complete the proof, let $G$ be a $\left( {\dilipV}_{{\omega}_{1}}, \dilipPP(\dilipA) \right)$-generic filter.
 As $\dilipPP(\dilipA)$ is proper, $\dilipFFF$ is well-closed w.r.t.\@ $\dilipwo$ in ${\dilipV}_{{\omega}_{1}}\left[G\right]$.
 By what has been proved above, we have in ${\dilipV}_{{\omega}_{1}}\left[G\right]$ that for every $g: \omega \rightarrow \diliphf$, there exists $T \in \dilipFFF$ such that $T: \omega \rightarrow \diliphf$ is a slalom and $\dilipexistsinf k \in \omega \left[g(k) \in T(k)\right]$.
 As $\dilipFFF$ is well-closed w.r.t.\@ $\dilipwo$ in ${\dilipV}_{{\omega}_{1}}\left[G\right]$, Lemma \ref{lem:cumber} tells us that $\dilipFFF$ is big in ${\dilipV}_{{\omega}_{1}}\left[G\right]$.
\end{proof}
Note the similarity in the proofs of Lemmas \ref{lem:preservingbig} and \ref{lem:cohenbig}.
Also, the proof of Lemma \ref{lem:preservingbig} is the only place where we use the assumption that $\dilipA$ is block Shelah-Stepr{\= a}ns.
\begin{dilipLemma} \label{lem:diagonalizes}
 $\dilipPP(\dilipA)$ diagonalizes $\dilipA$.
\end{dilipLemma}
\begin{proof}
 Let $\mathring{A}$ be a $\dilipPP(\dilipA)$-name such that ${\dilipforces}_{\dilipPP(\dilipA)} \; {\mathring{A} = \bigcup\left\{ {s}_{p}: p \in \mathring{G} \right\}}$, where $\mathring{G}$ is the canonical $\dilipPP(\dilipA)$-name for a generic filter over $\dilipPP(\dilipA)$.
 Suppose $p \in \dilipPP(\dilipA)$ and $B \in \dilipA$ are given.
 Then $Z = \{z \in \omega: {c}_{p}(z) \cap B = \emptyset \} \in \dilippc{\omega}{\omega}$.
 Let $\dilipseq{z}{j}{<}{\omega}$ be the strictly increasing enumeration of $Z$.
 Define $q = \dilippr{{s}_{q}}{{c}_{q}}$ by setting ${s}_{q} = {s}_{p}$ and ${c}_{q}(j) = {c}_{p}({z}_{j})$, for all $j < \omega$.
 Then $q \leq p$ and $q \; {\dilipforces}_{\dilipPP(\dilipA)} \; {\diliplc \mathring{A} \cap B \diliprc < {\aleph}_{0}}$.
\end{proof}
The following corollary is worth stating even though it is not directly used in the proof of our main result in the next section.
\begin{dilipCor} \label{cor:diagonalizes}
 Any Shelah-Stepr{\= a}ns a.d.\@ family in $\dilipV$ can be diagonalized without increasing $\dilipnon(\dilipMMM)$.
 Any block Shelah-Stepr{\= a}ns a.d.\@ family in $\dilipV$ which remains block Shelah-Stepr{\= a}ns after adding Cohen reals can be diagonalized without increasing $\dilipnon(\dilipMMM)$.
\end{dilipCor}
\begin{proof}
 If $\dilipA$ is a Shelah-Stepr{\= a}ns a.d.\@ family, then $\dilipA$ remains Shelah-Stepr{\= a}ns, and hence block Shelah-Stepr{\= a}ns, after adding Cohen reals (see Lemma \ref{lem:CSS}).
 Hence ${\dilipCC}_{{\omega}_{1}} \ast \mathring{\dilipPP}(\dilipA)$ will diagonalize $\dilipA$ and not increase $\dilipnon(\dilipMMM)$.
\end{proof}
\section{The main result} \label{sec:mainresult}
\begin{dilipTheorem} \label{thm:nosmallSS}
 There is a model in which $\dilipnon(\dilipMMM) = {\aleph}_{1}$ and there are no Shelah-Stepr{\= a}ns or block Shelah-Stepr{\= a}ns a.d.\@ families of size ${\aleph}_{1}$.
\end{dilipTheorem}
\begin{proof}
 Shelah-Stepr{\= a}ns a.d.\@ families are block Shelah-Stepr{\= a}ns.
 So it is enough to produce a model where $\dilipnon(\dilipMMM) = {\aleph}_{1}$ and there are no block Shelah-Stepr{\= a}ns a.d.\@ families of size ${\aleph}_{1}$.
 Let $\dilipV$ be a universe satisfying $\dilipGCH$.
 In $\dilipV$, let $\dilipwo$ be a well-ordering of $\diliphf$ and let $\dilipFFF = \dilipV \cap {\diliphf}^{\omega}$.
 Then $\dilipFFF$ is well-closed w.r.t.\@ $\dilipwo$ in $\dilipV$.
 $\dilipFFF$ is also big in $\dilipV$.
 Build a CS iteration $\left\langle {\dilipPP}_{\alpha}; {\mathring{\dilipQQ}}_{\alpha}: \alpha \leq {\omega}_{2} \right\rangle$ as follows.
 Using $\dilipGCH$ in $\dilipV$, fix a bookkeeping device which has the property that any ${\dilipPP}_{{\omega}_{2}}$-name for a set of reals of size ${\aleph}_{1}$ will be enumerated cofinally often.
 At a stage $\alpha < {\omega}_{2}$, suppose ${\dilipPP}_{\alpha}$ is given.
 Assume ${\dilipPP}_{\alpha}$ is proper and ${\dilipforces}_{\alpha} \; {\dilipFFF \ \text{is big}}$.
 Suppose the bookkeeping device hands us a ${\dilipPP}_{\alpha}$-name $\mathring{\dilipA}$ such that
 \begin{align*}
  {\dilipforces}_{\alpha} \; {``\mathring{\dilipA} \subseteq \dilippc{\omega}{\omega} \ \text{is an infinite a.d.\@ family}''}.
 \end{align*}
 Let ${G}_{\alpha}$ be a $\left( \dilipV, {\dilipPP}_{\alpha} \right)$-generic filter.
 In $\dilipV\left[{G}_{\alpha}\right]$, either
 \begin{align*}
  {\dilipforces}_{{\dilipCC}_{{\omega}_{1}}} \; {``\mathring{\dilipA}\left[{G}_{\alpha}\right] \ \text{is block Shelah-Stepr{\= a}ns}''}
 \end{align*}
 or ${\dilipforces}_{{\dilipCC}_{{\omega}_{1}}} \; {``\mathring{\dilipA}\left[{G}_{\alpha}\right] \ \text{is not block Shelah-Stepr{\= a}ns}''}$ because ${\dilipCC}_{{\omega}_{1}}$ is almost homogeneous.
 If the first alternative happens, then let $\dilipQQ$ be ${\dilipCC}_{{\omega}_{1}} \ast \mathring{\dilipPP}(\mathring{\dilipA}\left[{G}_{\alpha}\right])$.
 If the second alternative happens, then let $\dilipQQ$ be ${\dilipCC}_{{\omega}_{1}}$.
 Back in $\dilipV$, let ${\mathring{\dilipQQ}}_{\alpha}$ be a full ${\dilipPP}_{\alpha}$-name for $\dilipQQ$.
 If the bookkeeping device does not hand us a ${\dilipPP}_{\alpha}$-name of the form $\mathring{\dilipA}$, then we let ${\mathring{\dilipQQ}}_{\alpha}$ be a full ${\dilipPP}_{\alpha}$-name for the trivial forcing.
 Observe that ${\dilipforces}_{\alpha} \; {``{\dilipforces}_{{\mathring{\dilipQQ}}_{\alpha}} \; {\mathring{\dilipA} \ \text{is not block Shelah-Stepr{\= a}ns}}''}$ (if an appropriate $\mathring{\dilipA}$ is given) and that ${\dilipforces}_{\alpha} \; {``{\dilipforces}_{{\mathring{\dilipQQ}}_{\alpha}} \; {\dilipFFF \ \text{is big}}''}$.
 This concludes the construction.
 
 Let ${G}_{{\omega}_{2}}$ be a $\left( \dilipV, {\dilipPP}_{{\omega}_{2}} \right)$-generic filter.
 In $\dilipV\left[{G}_{{\omega}_{2}}\right]$, by Corollary \ref{cor:bigispreserved}, $\dilipFFF$ is big.
 Therefore by Lemma \ref{lem:miller}, ${2}^{\omega} \cap \dilipV$ is non-meager in $\dilipV\left[{G}_{{\omega}_{2}}\right]$.
 Therefore $\dilipnon(\dilipMMM) = {\aleph}_{1}$ in $\dilipV\left[{G}_{{\omega}_{2}}\right]$.
 
 Next, suppose for a contradiction that in $\dilipV\left[{G}_{{\omega}_{2}}\right]$, there exists $\dilipA \subseteq \dilippc{\omega}{\omega}$ which is a block Shelah-Stepr{\= a}ns a.d.\@ family with $\diliplc \dilipA \diliprc = {\aleph}_{1}$.
 For $\gamma \leq {\omega}_{2}$, ${G}_{\gamma}$ denotes the restriction of ${G}_{{\omega}_{2}}$ to ${\dilipPP}_{\gamma}$.
 There exists $\xi < {\omega}_{2}$ such that $\dilipA \in \dilipV\left[{G}_{\xi}\right]$.
 Since the property of being a block Shelah-Stepr{\= a}ns a.d.\@ family is downwards absolute, $\dilipA$ is a block Shelah-Stepr{\= a}ns a.d.\@ family in $\dilipV\left[{G}_{\gamma}\right]$, for all $\xi \leq \gamma \leq {\omega}_{2}$.
 The bookkeeping device ensured that for some $\xi \leq \alpha < {\omega}_{2}$, a ${\dilipPP}_{\alpha}$-name $\mathring{\dilipA}$ such that $\mathring{\dilipA}\left[{G}_{\alpha}\right] = \dilipA$ was considered at stage $\alpha$.
 By the choice of ${\mathring{\dilipQQ}}_{\alpha}$, ${\dilipforces}_{\alpha+1} \; {``\mathring{\dilipA} \ \text{is not block Shelah-Stepr{\= a}ns}''}$.
 This is a contradiction because $\dilipA = \mathring{\dilipA}\left[{G}_{\alpha}\right] = \mathring{\dilipA}\left[{G}_{\alpha+1}\right]$ is block Shelah-Stepr{\= a}ns in $\dilipV\left[{G}_{\alpha+1}\right]$.
\end{proof}
Note that since $\max\{\dilipbb, \dilipsss\} \leq \dilipnon(\dilipMMM)$, $\dilipbb = \dilipsss = {\aleph}_{1}$ holds in the model $\dilipV\left[{G}_{{\omega}_{2}}\right]$ constructed above.
\begin{dilipconj} \label{conj:exists}
 There are Shelah-Stepr{\= a}ns a.d.\@ families in the model $\dilipV\left[{G}_{{\omega}_{2}}\right]$ constructed in the proof of Theorem \ref{thm:nosmallSS} (necessarily of size ${\aleph}_{2}$).
\end{dilipconj}
A model with no Shelah-Stepr{\= a}ns a.d.\@ families is constructed in \cite{AModelwithnoStronglySeparableAlmostDisjointFamilies}.
This model is gotten by iterating posets of the form $\dilipLL(\dilipGGG)$, where $\dilipGGG$ is some filter on a countable set.
Since posets of this form always add dominating reals, $\dilipbb > {\aleph}_{1}$ in the model in  \cite{AModelwithnoStronglySeparableAlmostDisjointFamilies}.
\begin{dilipQuestion} \label{q:ssnonM}
 Is it consistent that there are no Shelah-Stepr{\= a}ns a.d.\@ families and $\dilipnon(\dilipMMM) = {\aleph}_{1}$?
\end{dilipQuestion}

\def\polhk#1{\setbox0=\hbox{#1}{\ooalign{\hidewidth
  \lower1.5ex\hbox{`}\hidewidth\crcr\unhbox0}}}
\providecommand{\bysame}{\leavevmode\hbox to3em{\hrulefill}\thinspace}
\providecommand{\MR}{\relax\ifhmode\unskip\space\fi MR }
\providecommand{\MRhref}[2]{%
  \href{http://www.ams.org/mathscinet-getitem?mr=#1}{#2}
}
\providecommand{\href}[2]{#2}





\end{document}